\documentclass[11pt,letterpaper,oneside,reqno]{book}

\usepackage[letterpaper,hmargin={1in},vmargin=1in,foot=0.5in]{geometry}

\usepackage{amsmath,amsthm}

\usepackage[english]{babel}

\usepackage{mathrsfs,amssymb}

\usepackage[matrix,arrow]{xy}

\usepackage{graphicx}
\usepackage{mathptmx}

\usepackage{setspace}

\usepackage{titlesec}
\titleformat{\chapter}[display]
   {\normalfont\filcenter}
   {\MakeUppercase{\chaptertitlename} \Roman{chapter}}
   {\baselineskip}
   {\MakeUppercase}
\titlespacing{\chapter}{0pt}{-\baselineskip}{\baselineskip}

\titleformat{\section}[block]
   {\normalfont\filcenter}
   {\thesection\ }
   {\baselineskip}
   {}

\usepackage{titletoc}
\titlecontents{chapter}[0em]{\addvspace{\baselineskip}}
              {\hbox to 2.3em{\thecontentslabel.\hfil}\MakeUppercase}{\MakeUppercase}
              {\titlerule*[.7pc]{.}\contentspage}{}
\dottedcontents{section}[4.6em]
   {\edef\lab{\thecontentslabel}\ex\ex\ex\if1\ex\eatuntildot\lab\addvspace{\baselineskip}\fi}
   {2.3em}{.7pc}
\def\eatuntildot#1.{}
\let\ex=\expandafter
\addto\captionsenglish{%
  \renewcommand{\contentsname}%
    {Table of contents}%
}

\usepackage[backend=bibtex,sorting=nyt]{biblatex}
\bibliography{pthesis-bib.bib}

\usepackage{fancyhdr}
 \theoremstyle{plain}    
 \newtheorem{theorem}{Theorem}[section]
 \theoremstyle{definition}
 \newtheorem{definition}[theorem]{Definition}
 
 \theoremstyle{plain}    
 \newtheorem{proposition}[theorem]{Proposition}
 \theoremstyle{remark}
 \newtheorem{remark}[theorem]{Remark}
 \newtheorem{example}[theorem]{Example}

\theoremstyle{plain}    
 \newtheorem{lemma}[theorem]{Lemma} 
 
 \theoremstyle{definition}
 \theoremstyle{plain}    
 \newtheorem{corollary}[theorem]{Corollary} 

\title{Age-structured Population Models with Applications \\
  \normalsize PhD thesis} 

\author{Min Gao}

\date{1/12/2015}

\begin{document}

\addtocontents{toc}{\vskip\baselineskip\hfill Page\par}

\frontmatter


\thispagestyle{empty}
\doublespacing
\begin{center}

Age-structured Population Models with Applications\\

By\\

Min Gao\\

Dissertation\\
Submitted to the Faculty of the\\
Graduate School of Vanderbilt University\\
in partial fulfillment of the requirements\\
for the degree of\\

DOCTOR OF PHILOSOPHY\\
in\\
Mathematics\\

August, 2015\\
Nashville, Tennessee\\
\end{center}
\hspace{1cm} 
\begin{center}
Approved: \\
Professor Glenn F. Webb \\
Professor Philip S. Crooke \\
Professor Doug Hardin \\
Professor Vito Quaranta \\

\end{center}

\newpage
\singlespacing
\thispagestyle{empty}
\vspace*{4.5in}
\centerline{Copyright {\small\copyright}\ 2015 by Min Gao}
\centerline{All Rights Reserved}

\newpage
\setcounter{page}{3}
\doublespacing
\addcontentsline{toc}{chapter}{Dedication}

\vspace*{3in}
\centerline{\em To my parents}
\centerline{\em  Jiaping Xu}
\centerline{\em and}
\centerline{\em Youshen Gao}

\newpage
\doublespacing
\chapter{Abstract}

A general model of age-structured population dynamics is developed and the fundamental properties of its solutions are analyzed. The model is a semilinear partial differential equation with a nonlinear nonlocal boundary condition. Existence, uniqueness and regularity of solutions to the model equations are proved. An intrinsic growth constant is obtained and linked to the existence and the stability of the trivial and/or the positive equilibrium. Lyapunov function is constructed to show the global stability of the trivial and/or the positive equilibrium.

Key words: Lyapunov function, Partial Differential Equation, Stability, Mathematical Model, Uniform       
                    Persistence, Reproductive Number (or Reproduction Number)



\singlespacing
\tableofcontents

\listoffigures
\listoftables

\onehalfspacing

\mainmatter


\addtocontents{toc}{\vskip\baselineskip\noindent Chapter\par}

\chapter{Introduction}

\section{Human age structure and human-microbe coevolution}

A human newborn is colonized by microbiota from the environment, especially through the exposure to the mother. The coevolution of humans and microbes occured over evolutionary time with adaptation to their environment. This shared evolutionary process involving human hosts and their symbiotic microbes, selected for mutualistic interactions that are beneficial for human health. Ecological and genetic changes that perburbed this symbiotic system resulted in disease or host demise through interactions with high-grade pathogens \cite{De}.
Studying the coevolution of human host and microorganisms in a dynamical system involving human age structure improves the current understanding of both human health and disease \cite{Wa1}.

The age structure of human females is exceptional among species, with its extremely extended pre-reproductive and post-reproductive phases. We explore how nature optimizes the age structure of human species through the mutualistic interactions between human host and microbiota. 
Recent studies in \cite{Br} show that a female baboon matures at $\approx 5$ years and has average life expectancy of $\approx 16$ years. Studies in \cite{Kr, Fo, Har,  At, Su} show that a whale matures at $\approx$ 9.5 years and could live to $\approx$ 90 years, Chimpanzees become reproductively active at $\approx 11-12$ years and have average lifespan $\approx 50$ years and Gorilla females mature at $\approx 10-11$ years with the maximum longevity $\approx 52$ years. In contrast, a female in an early human hunter-gatherer society became fertile at $\approx 15$ years and has an inherent lifespan on the order of $\approx 70$ years \cite{Gur}.  

It is commonly argued that selective pressure generally favors the good of the species in contrast to individual fitness. Child survival rate is an important indicator of human female reproductive success\cite{Jo, Re}, since human females have relatively little difference in fertility. Therefore, human age structure has to reflect the adaptation of maximizing the survival probability of newborns to ensure the reproductive success of the human species, while not overburdening their mothers. 
Clearly, young humans are not self-sufficient until much later than almost all other species, and parent or extended family care is necessary for their survival. 
The fact that the human female nonreproductive state (that is, the extended juvenility and senescence phases) accounts for a relatively high proportion of female's lifespan may have an evolutionary advantage that increases human species' reproductive success and the fitness of descendants.

The extended juvenility, senescence, and longevity of \textit{Homo sapiens} has been maintained over the past $130,000$ years \cite{Br, Gur}, and is inherited by modern humans. This is related to a dramatic fivefold increase in the ratio of older to younger adults (O-Y ratio) \cite{Gur,C} over evolutionary time. Studies in \cite{Brk, Gur} indicate that there exists a maximum age for human life span, which may be inherent in humans and resultant from evolutionary pressure. Such pressure involves optimizing limited resources in balancing juvenile and adult populations. Another consideration is the benefit of post-reproductive nurturing of juveniles. Such elements play an essential role in shaping human age structure and influencing total size of the population.  A consideration in this balance is the benefit of post-reproductive nurturing of juveniles-the so called \textit{grandmother effect} \cite{Fin, Fin1,  Haw, Kim, Fo, R, Brk, De, J, K}. 

From the coevolutionary point of view, we assumes that the mutualistic interactions between human hosts and their indigenous microbiota could be divided into two parts.  First, during reproductive life, there is selection for microbes that preserve host function, through regulation of energy homeostasis, promotion of fecundity, and interference with competing high-grade pathogens.  Second, after reproductive life, there is selection for organisms that contribute to host demise. While harmful for the individual (during their post-reproductive age), such interplay is salutary for the overall population, in terms of resource utilization, resistance to periodic diminutions in food supply, and epidemics due to high-grade pathogens \cite{B}. A question arises as to the fragility or robustness of this human structure, which is related to modern human lifespan. 
We study the age structure of early humans with models that illustrate the unique intrinsic balance and robustness of human fertility and mortality. We hypothesize that female fertility is regulated by human age structure and population density, that is, the actual population size has an influence on the quantity of newborns; and female mortality incorporates the programmed death of post-reproductive population through interactions with indigenous microorganisms \cite{B} and the effects of crowding. 

We investigate the interaction between those mortalities: one takes the form of all cause mortality, which could be total population size dependent, and affects individuals from all age classes. The other depends on the size of the post-reproductive population, but only affects the pre-reproductive subpopulation, with youngest ones most vulnerable. This models the competition between the post-reproductive population and the pre-reproductive population for the limited resources that all members of the population require. Thus, as the post-reproductive population grows, it disproportionately affects pre-reproductive individuals.

In early human populations, this effect could take the form of the competition between senescent and juvenile subpopulations. The numerical examples presented in \cite{B} and also in section 4 support the hypothesis that the senescent burden on juveniles is not negligible, and might affect the population much stronger than crowding effects. We thus argue that in early humans with especially small total population size, the fraction of senescent population is fairly small compared with modern society. This is consistent with recent anthropological findings: more than 1 million years ago, there were only 18,500 human ancestors, with relatively short average lifespan, living on earth \cite{H, Me}. Numerical simulations in \cite{B} and in section 4 illustrate that when subject to a linear or a nonlinear birth rate, the evolution of a population could exhibit relatively wild oscillations before it arrives at a steady state, if the senescent burden on juvenile individuals is significantly large. Such oscillations could reduce the overall population size drastically and might endanger the survival of the population, since the recovery from events corresponding to harsh times would be problematic. In particular, when a population becomes too small to be viable (which is known as Allee or fade-out effects \cite{B}), these oscillations would lead to the extinction of a population. 


\section{The age-structured population models}

A variety of linear and nonlinear logistic models have been developed to analyze the dynamical viability of age-structure in human populations (see \cite{M},  \cite{A}, \cite{P6},   \cite{W}, \cite{W86}, \cite{W87}, \cite{W02}, \cite{W8}, \cite{T},  \cite{M1},  \cite{Cu}, \cite{P},  \cite{F1}, \cite{F}, \cite{I}, \cite{G}, \cite{MC}, \cite{D}, \cite{D7}, \cite{Z}). In \cite{D}, the asymptotic behaviour of solutions of the following abstract differential equation (\ref{sec:it1}) has been analyzed

\begin{align}\label{sec:it1}
&u'(t)=\mathcal{A}u(t)-\mathcal{F}(u(t))u(t)+f,\;t\geq 0;
&u(0)=x\in X_{+}.
\end{align}

under the following hypotheses:
\begin{itemize}
\item[(H.1)] $\mathcal{A}$ is the infinitesimal generator of a strongly continuous semigroup of positive linear operators $T(t), t\geq 0,$ in the Banach lattice $X$ with positive cone $X_{+}$;
\item[(H.2)] $\mathcal{F}$ is a positive linear functional on $X$;
\item[(H.3)] $f\subset X_{+}$;
\item[(H.4)] $x\in X_{+}$ and $\lim_{t\rightarrow \infty}t^{-n}e^{-\lambda_{0} t}T(t)x=P_{0}x$ where $n$ is a positive integer, $\lambda_{0}\in\mathbb{R}$, and $P_{0}$ is a bounded linear operator in $X$;
\item[(H.5)] $\mathcal{F}P_{0}x>0$.
Set
\begin{align}\label{sec:it2}
&S(t)x=\frac{T(t)x}{1+\int_{0}^{t}\mathcal{F}(T(s)x)ds}.
\end{align}
\end{itemize}

The results from \cite{D} are stated as follows:

\begin{theorem}
Let (H.1)-(H.5) hold and let $f = 0$.
\begin{itemize}
\item[(i)] If $x\in D(\mathcal{A})$, then $u(t) = S(t)x$ is the unique solution of (\ref{sec:it1});
\item[(ii)] If $\lambda_{0}<0$, then $\lim_{t\rightarrow \infty}S(t)x=0$;
\item[(iii)] If $\lambda_{0}\geq 0$, then $\lim_{t\rightarrow \infty}S(t)x=\frac{\lambda_{0} P_{0} x}{\mathcal{F}(P_{0}x)}$.
\end{itemize}
\end{theorem}

\section{Formulation of a nonlinear model}\label{sec:mod}

We consider a female population with the total population size $T(t)$ at time $t$. Let $a_{1}$ be the maximum age of a female, and $a_{\min}$, $a_{\max}$ be the beginning and the end of the reproductive period, respectively. We divide the total population $T(t)$ into juvenile, reproductive, and senescent subpopulations, denoted by $J(t)$, $R(t)$ and $S(t)$. 
Let $X=L^{1}(0,a_{1})$ be the state space, endowed with the norm $\left\|\phi\right\|_{X}=\int_{0}^{a_{1}}|\phi(a)|da,\ \text{ for }\phi\in X$.
Let $p(a,t)$ denote the population density at age $a$ and time $t$. We stratify the number of females  between ages $b_{1}$ and $b_{2}$ at time $t$ as: $\int_{b_{1}}^{b_{2}}\hat{\omega}(a)p(a,t)da$, where $\hat{\omega}\in L^{\infty}_{+}(0,a_{1})$ is a specified weight function and $0\leq b_{1}< b_{2}\leq a_{1}$. 

The change of the population density $p(a,t)$ at age $a$ and time $t$ obeys the following balance law:
\begin{align}\label{sec:1}
&p_{t}(a,t)+p_{a}(a,t)=-[\mu_{0}(a, \eta_{0} (Q_{0}(t)))+\mu_{1}(a,\eta_{1}( Q_{1}(t)))+ \mu_{2}(a)]p(a,t),\\
\nonumber & 0\leq a\leq a_{1},\ t\geq 0,\\
\nonumber &p(0,t)=\int_{a_{\min}}^{a_{\max}}\beta(a; \eta_{2} (Q_{0}(t)))p(a,t)da,\ t\geq 0,\\
\nonumber &p(a,0)=p_{0}(a),\ 0\leq a\leq a_{1}.
\end{align}
where, $Q_{i}(t)=\int_{0}^{a_{1}}\omega_{i}(a)p(a,t)da$, $i=0,1$ and $p_{0}(a)$ is the initial age distribution of a population. Moreover, $\beta(a; \eta_{2} (Q_{0}(t)))$ represents the fertility rate of a female at age $a$ when the weighted number of females that affect the fertility is given by $Q_{0}(t)$ at time $t$. $\mu_{2}(a)$ represents the age-dependent mortality in the host population, which could be influenced by a particular class of microbes affecting age structure, but is not influenced by environmental limitations. Moreover, $\mu_{i}(a, \eta_{i} (Q_{i}(t))), \ i=0,1$, provide mechanisms which incorporate mortalities that affect the population disproportionately through age as the weighted number of females $Q_{i}(t)$ change over time $t$, and the effects of crowding and resource limitation take hold. In particular, $\mu_{1}(a, z)$ only affects the pre-reproductive population, that is, $\mu_{1}(a,z)=0, \  \text{ for } a>a_{\min}$ and $z \geq 0$.

We define the \textit{birth function} and the \textit{aging function} $\mathcal{F}:X\rightarrow \mathbb{R}$ and $\mathcal{G}:X\rightarrow X$ for $\forall\phi\in X$ by:
\begin{align}
&\mathcal{G}(\phi)(a)=-(\mu_{0}(a, \eta_{0}(Q_{0}\phi))+\mu_{1}(a, \eta_{1}(Q_{1}\phi))+\mu_{2}(a))\phi(a).\label{sec:2}\\
&\mathcal{F}(\phi)=\int^{a_{\max}}_{a_{\min}}\beta(a; \eta_{2}(Q_{0}\phi))\phi(a)da.\label{sec:3}
\end{align}
where $Q_{i}\phi=\int_{0}^{a_{1}}\omega_{i}(a)\phi(a)da$ for $\phi\in X$, $i=0,1$.
We make the following assumptions on $\beta, \mu_{i}, \eta_{i}, i=0,1,2, \omega_{i}, i=0,1$, and $p_{0}$ throughout this paper:
\begin{itemize}
\item[H.1.] $\beta\in C^{1}([a_{\min},a_{\max}]\times [0,\infty))$, $\beta(a,z)\geq 0$ for $(a,z)\in([a_{\min},a_{\max}]\times [0,\infty))$, $\mu_{i}\in C^{1}((0,a_{1})\times [0,\infty))$, $\mu_{i}(a,z)\geq 0$ for $(a,z)\in((0,a_{1})\times [0,\infty))$, $i=0,1$, $\mu_{1}(a, z)=0, \text{ for } a > a_{\min}$ and $z\geq 0$, $\mu_{1}(a, 0)=0$ for $a\in(0,a_{\min})$,
$\mu_{1,z}(a,z)> 0$ for $(a,z)\in((0,a_{1})\times [0,\infty))$, $\mu_{2} \in C^{1}(0,a_{1})$, $\mu_{2}(a)\geq 0$ for $a\in(0,a_{1})$, $\omega_{i}\in L^{\infty}_{+}(0,a_{1})$, $i=0,1$, and $p_{0}\in X_{+}$. We require that $\eta_{i}$ maps $[0,\infty)$ onto $[0,\infty)$, $i=0,1,2$. Further, $\eta_{i}\in C^{1}[0,\infty)$, $\eta_{i}(z)\geq 0$ and $\eta_{i}'(z)>0$, for $z\in[0,\infty)$, $i=0,1,2$.
\end{itemize}

We formulate the age-dependent population dynamics as follows (refer to \cite{W}, sec 1.4, pp.17):

Let $P(t)$ be the total population at time $t$. The average rate of change in the total population size in the time interval $(t,t+h)$ is 
\begin{align}\label{sec:char7}
\frac{P(t+h)-P(t)}{h}=h^{-1}\int_{0}^{h}p(a,t+h)da+\int_{h}^{a_{1}}\frac{p(a+h,t+h)-p(a,t)}{h}da
\end{align}

Let $T>0$, let $l\in L_{T}$, let $\mathcal{F}$ be a mapping from $X$ to $\mathbb{R}$, let $\mathcal{G}$ be a mapping from $X$ to $X$ and let $\phi\in X$. The \textit{balance law} of the population is given by 
\begin{align}\label{sec:char8}
\lim_{h\rightarrow 0^{+}}\int_{0}^{a_{1}}|\frac{p(a+h,t+h)-p(a,t)}{h}-\mathcal{G}(p(\cdot,t))(a)|da=0\;t\in[0,T].
\end{align}
The \textit{birth law} of the population is given by 
\begin{align}\label{sec:char9}
\lim_{h\rightarrow 0^{+}}\int_{0}^{h}|p(a,t+h)-\mathcal{F}(p(\cdot,t))(a)|da=0\;t\in[0,T].
\end{align}
The \textit{initial age distribution} of the population is given by 
\begin{align}
p(\cdot,0)=\phi
\end{align}
From (\ref{sec:char7}), (\ref{sec:char8}) and (\ref{sec:char9}) we see that the rate of change of the total population size follows
\begin{align}\label{sec:char88}
\frac{d}{dt}P(t)=\mathcal{F}((\cdot,t))+\int_{0}^{a_{1}}\mathcal{G}(p(\cdot,t))(a)da.
\end{align}
where $\mathcal{F}((\cdot,t))$ is the birth rate at time $t$ and $\int_{0}^{a_{1}}\mathcal{G}(p(\cdot,t))(a)da$ is the rate of change of total population at time $t$ due to aging process.

\section{Reformulation as integral equations} The method we use to solve this problem is the method of characteristics. We proceed as follows (\cite{W}, section 1.4, pp.21): Suppose that the solution $p(a,t)$ of problem (\ref{sec:1}) is known. The characteristic curves of the equations (\ref{sec:1}) are the lines $a-t=c$, where $c$ is a constant. Let $c\in\mathbb{R}$ and define the cohort function
\begin{align}
w_{c}(t):=p(t+c,t),\ t\geq t_{c}
\end{align}
From (\ref{sec:1}) we obtain,
\begin{align}
\frac{d}{dt}w_{c}(t)&=\lim_{h\rightarrow 0^{+}}\frac{p(t+h+c,t+h)-p(t+c,t)}{h}\\
&=\mathcal{G}p(t+c,t)\\
&=-(\mu_{0}(t+c, \eta_{0}(Q_{0}(t)))+\mu_{1}(t+c, \eta_{1}(Q_{1}(t)))+\mu_{2}(t+c))w_{c}(t),\ t\geq t_{c}.
\end{align}

This implies that 
\begin{align}\label{sec:char1}
w_{c}(t)=w_{c}(t_{c})e^{-\int_{t_{c}}^{t}(\mu_{0}(s+c, \eta_{0}(Q_{0}(s)))+\mu_{1}(s+c, \eta_{1}(Q_{1}(s)))+\mu_{2}(s+c))ds}\ t\geq t_{c}.
\end{align}
If we set $c=a-t$, where $a\geq t$, then
\begin{align}
w_{c}(t)=w_{c}(0)e^{-\int_{0}^{t}(\mu_{0}(s+c, \eta_{0}(Q_{0}(s)))+\mu_{1}(s+c, \eta_{1}(Q_{1}(s)))+\mu_{2}(s+c))ds}\ t\geq 0.
\end{align}
which yields
\begin{align}\label{sec:char2}
p(a,t)=p(a-t,0)e^{-\int_{0}^{t}(\mu_{0}(s+a-t, \eta_{0}(Q_{0}(s)))+\mu_{1}(s+a-t, \eta_{1}(Q_{1}(s)))+\mu_{2}(s+a-t))ds}\ a\geq t.
\end{align}
If we set $c=a-t$ where $a<t$, then
\begin{align}
w_{c}(t)=w_{c}(-c)e^{-\int_{-c}^{t}(\mu_{0}(s+c, \eta_{0}(Q_{0}(s)))+\mu_{1}(s+c, \eta_{1}(Q_{1}(s)))+\mu_{2}(s+c))ds}\ t\geq -c.
\end{align}
which yields
\begin{align}\label{sec:char3}
p(a,t)=p(0,t-a)e^{-\int_{t-a}^{t}(\mu_{0}(s+a-t, \eta_{0}(Q_{0}(s)))+\mu_{1}(s+a-t, \eta_{1}(Q_{1}(s)))+\mu_{2}(s+a-t))ds}\ a< t.
\end{align}
Combine forumlas (\ref{sec:char2}) and (\ref{sec:char3}) to obtain
\begin{align}\label{sec:char51}
p(a,t)=\begin{cases}
p(0,t-a)e^{-\int_{t-a}^{t} (\mu_{0}(s+a-t, \eta_{0}(Q_{0}(s)))+\mu_{1}(s+a-t, \eta_{1}(Q_{1}(s)))+\mu_{2}(s+a-t))ds} \ \ a \in (0,t)\cap[0,a_{1}];\\
p_{0}(a-t)e^{-\int_{0}^{t}(\mu_{0}(s+a-t, \eta_{0}(Q_{0}(s)))+\mu_{1}(s+a-t, \eta_{1}(Q_{1}(s)))+\mu_{2}(s+a-t)) ds} \ \ a\in(t,a_{1}].
\end{cases}
\end{align}
where $p(a,0)=p_{0}(a)$ for $a\in [0,a_{1}]$.

Define $B(t):=p(0,t)$ and substitute the formula for $p(a,t)$ (\ref{sec:char51}) into $Q_{i}(t)$, $i=0,1$, $t\geq 0$ and $B(t)$ to obtain,
\begin{align}\label{sec:char33}
Q_{0}(t)&=\int_{0}^{t}B(t-a)\exp[-\int_{t-a}^{t}\mathcal{G}(p(\cdot,s))(s+a-t)ds]da\\
\nonumber&+\int_{t}^{a_{1}}p_{0}(a-t)\exp[-\int_{0}^{t}\mathcal{G}(p(\cdot,s))(s+a-t)ds]da.
\end{align}
\begin{align}\label{sec:char4}
Q_{1}(t)&=\int_{a_{\max}}^{t}B(t-a)\exp[-\int_{t-a}^{t}\mathcal{G}(p(\cdot,s))(s+a-t)ds]da\\
\nonumber&+\int_{t}^{a_{1}}p_{0}(a-t)\exp[-\int_{0}^{t}\mathcal{G}(p(\cdot,s))(s+a-t)ds]da.
\end{align}
\begin{align}\label{sec:char5}
B(t)&=\int_{a_{\min}}^{t}\beta(a;Q_{0}(t))B(t-a)\exp[-\int_{t-a}^{t}\mathcal{G}(p(\cdot,s))(s+a-t)ds]da\\
\nonumber&+\int_{t}^{a_{\max}}\beta(a;Q_{0}(t))p_{0}(a-t)\exp[-\int_{0}^{t}\mathcal{G}(p(\cdot,s))(s+a-t)ds]da.
\end{align}
or equivalently,
\begin{align}\label{sec:char6}
Q_{0}(t)&=\int_{0}^{t}B(a)\exp[-\int_{a}^{t}\mathcal{G}(p(\cdot,s))(s-a)ds]da\\
\nonumber&+\int_{0}^{a_{1}}p_{0}(a)\exp[-\int_{0}^{t}\mathcal{G}(p(\cdot,s))(s+a)ds]da.
\end{align}
\begin{align}\label{sec:char77}
Q_{1}(t)&=\int_{a_{\max}}^{t}B(a)\exp[-\int_{a}^{t}\mathcal{G}(p(\cdot,s))(s-a)ds]da\\
\nonumber&+\int_{0}^{a_{1}}p_{0}(a)\exp[-\int_{0}^{t}\mathcal{G}(p(\cdot,s))(s+a)ds]da.
\end{align}
\begin{align}\label{sec:char8888}
B(t)&=\int_{a_{\min}}^{t}\beta(t-a;Q_{0}(t))B(a)\exp[-\int_{a}^{t}\mathcal{G}(p(\cdot,s))(s-a)ds]da\\
\nonumber&+\int_{0}^{a_{\max}}\beta(a-t;Q_{0}(t))p_{0}(a)\exp[-\int_{0}^{t}\mathcal{G}(p(\cdot,s))(s+a)ds]da.
\end{align}
The equations (\ref{sec:char6})-(\ref{sec:char8888}) constitute a coupled system of nonlinear Volterra integral equations in $B(t)$ and $Q_{i}(t)$, $i=0,1$.

The equivalent integral equation is given by (which is an adaptation of \cite{W}, sec 1.4, pp.21, (1.49))

\begin{align}\label{sec:inte11}
p(a,t)=\begin{cases}
\mathcal{F}(p(\cdot, t-a))+\int_{t-a}^{t}\mathcal{G}(p(\cdot, s))(s+a-t) ds & \text{ a.e. } \ \ a \in (0,t)\cap[0,a_{1}];\\
p_{0}(a-t)+\int_{0}^{t}\mathcal{G}(p(\cdot, s))(s+a-t) ds & \text{ a.e. } \ \ a\in(t,a_{1}].
\end{cases}
\end{align}

We state the following definition of solutions of the problem (\ref{sec:1}).
\begin{definition}[\cite{W}, sec 1.4, pp.17]
Let $T>0$ and let $p\in L_{T}$. We say that $p$ is a solution of problem (\ref{sec:1}) on $[0,T]$ provided that $p$ satisfies (\ref{sec:1}).
\end{definition}

\begin{definition}[\cite{W}, sec 1.4, pp.21]
Let $T>0$ and let $p\in L_{T}$. We say that $p$ is a solution of (\ref{sec:inte11}) on $[0,T]$ provided that $p(\cdot,t)$ satisfies (\ref{sec:inte11}) for $t\in [0,T]$.
\end{definition}

\chapter{Basic properties of the solutions}
In this chapter we establish some properties of solutions of the nonlinear problem (\ref{sec:1}) under the framework of \cite{W}.
\section{Preliminaries}
We derive from H.1 that the birth function $\mathcal{F}$ and aging function $\mathcal{G}$ are locally Lipschitz continuous in the following sense,
\begin{align}\label{sec:LC111}
&\mathcal{F}: X\rightarrow \mathbb{R},\text{ there is an increasing function }c_{1}: [0,\infty)\rightarrow[0,\infty)\text{ such that }|\mathcal{F}(\phi_{1})-\mathcal{F}(\phi_{2})|\\
\nonumber&\leq c_{1}(r)\left\|\phi_{1}-\phi_{2}\right\|_{X}\text{ for all }\phi_{1},\phi_{2}\in X\text{ such that }\left\|\phi_{1}\right\|_{X},\left\|\phi_{2}\right\|_{X}\leq r.\\
&\mathcal{G}: X\rightarrow X,\text{ there is an increasing function }c_{2}: [0,\infty)\rightarrow[0,\infty)\text{ such that }\left\|\mathcal{G}(\phi_{1})-\mathcal{G}(\phi_{2})\right\|\\
\nonumber&\leq c_{2}(r)\left\|\phi_{1}-\phi_{2}\right\|_{X}\text{ for all }\phi_{1},\phi_{2}\in X\text{ such that }\left\|\phi_{1}\right\|_{X},\left\|\phi_{2}\right\|_{X}\leq r.
\end{align}

We state an adapted version of three lemmas from \cite{W}, chapter 2, the first of which allows us to view an element in $L_{T}=C([0,T];X)$ as an element in $L^{1}((0,a_{1})\times(0,T);\mathbb{R})$. 

\begin{lemma}
Let $T>0$ and let $p\in L_{T}$. There is a unique element in $L^{1}((0,a_{1})\times(0,T);\mathbb{R})$ such that 
\begin{itemize}
\item[(i)] For each $t\in[0,T], p(a,t)=p(t)(a)$ for almost everywhere $a>0$.
\item[(ii)] 
\begin{align*}
&\int_{0}^{T}\left\|p(t)\right\|_{X}dt=\int_{0}^{T}[\int_{0}^{a_{1}}|p(a,t)|da]dt\\
\nonumber&=\int_{0}^{a_{1}}[\int_{0}^{T}|p(a,t)|dt]da\\
\nonumber&=\int_{0}^{a_{1}}\int_{0}^{T}|p(a,t)|dtda.
\end{align*}
\end{itemize}
\end{lemma}

This lemma establishes the existence of integrals in (\ref{sec:inte11}) when $p\in L_{T}$.

\begin{lemma}
Let H.1 hold, let $T>0$, let $\Gamma_{T}:=\left\{(c,s):0<s<T,-s<c<a_{1}\right\},\text{ and let }l\in L_{T}$. The following hold:
\begin{itemize}
\item[(i)] The function $t\rightarrow\mathcal{G}(p(\cdot,t))$ from $[0,T]$ to $L^{1}$ belongs to $L_{T}$.
\item[(ii)] There exists $h\in L^{1}((0,a_{1})\times(0,T);\mathbb{R})$ such that for each $t\in[0,T], h(a,t)=\mathcal{G}(p(\cdot,t))(a)$ for almost everywhere $a>0$.
\item[(iii)] There exists $k\in L^{1}(\Gamma_{T};\mathbb{R})$ such that $k(c,s)=h(s+c,s)$ for almost everywhere $(c,s)\in \Gamma_{T}$, and $\int_{0}^{T}[\int_{-s}^{a_{1}-s}k(c,s)dc]ds=\int_{-T}^{a_{1}-s}[\int_{\max\left\{0,-c\right\}}^{T}k(c,s)ds]dc$.
\end{itemize}
\end{lemma}

The following lemma characterizes compact sets in $X$.

\begin{lemma}
A closed and bounded subset $M$ of $L^{1}$ is compact if and only if the following condition hold:
\begin{align}
&\lim_{h\rightarrow0}\int_{0}^{a_{1}}|\phi(a)-\phi(a+h)|da=0\text{ uniformly for }\phi\in M(\text{ where } \phi(a+h)\text{ is taken as }0 \text{ if } a+h<0).
\end{align}
\end{lemma}

\section{Local existence and continuous dependence on initial values}

We collect some results from \cite{W} chapter 2 and adapt them for our nonlinear problem. 
We first state the following result to establish that a solution of the integral euqation (\ref{sec:inte11}) is also a solution of (\ref{sec:1}). 

\begin{proposition}\label{sec:be1}
Let H.1 hold and let $T>0$, let $\phi\in X$, and let $p\in L_{T}$. If $p$ is a solution of (\ref{sec:inte11}) on $[0,T]$, then $p$ is solution of the problem (\ref{sec:1}) on $[0,T]$.
\end{proposition}

\begin{proposition}\label{sec:be2}
Let H.1 hold and let $r>0$. There exists $T>0$ such that if $\phi\in X$ and $\left\|\phi\right\|_{X}\leq r$, then there is a unique function $p\in L_{T}$ such that  $p$ is a solution of (\ref{sec:inte11}) on $[0,T]$. 
\end{proposition}

\begin{proposition}\label{sec:be3}
Let H.1 hold, let $\phi,\hat{\phi}\in X$, let $T>0$, and let $p,\hat{p}\in L_{T}$ such that $p$, $\hat{p}$ is the solution of (\ref{sec:inte11}) on $[0,T]$ for $\phi$, $\hat{\phi}$, respectively. Let $r>0$ such that $\left\|\hat{p}\right\|_{L_{T}},\left\|p\right\|_{L_{T}}\leq r$. Then,
\begin{align}
\left\|p(\cdot, t)-\hat{p}(\cdot,t)\right\|_{X}\leq e^{(c_{1}(r)+c_{2}(r))t}\left\|\phi-\hat{\phi}\right\|_{X}\text{ for }0\leq t \leq T.
\end{align}
\end{proposition}

Collecting Proposition \ref{sec:be1}, \ref{sec:be2} and \ref{sec:be3} leads to,

\begin{theorem}
Let H.1 hold and let $\phi\in X$. There exists $T>0$ and $p\in L_{T}$ such that $p$ is a solution of the problem (\ref{sec:1}) on $[0,T]$. Furthermore, if $T>0$, then there is at most one solution of the problem (\ref{sec:1}) on $[0,T]$.
\end{theorem}

\section{The semigroup property and continuability of the solutions}

The following proposition shows that solutions of (\ref{sec:inte11}) has the semigroup property.

\begin{proposition}\label{sec:be4}
Let H.1 hold, let $\phi\in X$, let $T>0$, and let $p\in L_{T}$ such that $p$ is a solution of (\ref{sec:inte11}) on $[0,T]$. Let $\hat{T}>0$ and let $p\in L_{\hat{T}}$ such that for $t\in[0,\hat{T}]$
\begin{align}
\hat{p}(a,t)=\begin{cases}
   \mathcal{F}(\hat{p}(\cdot, t-a))+\int_{0}^{a} \mathcal{G}(\hat{p}(\cdot,s+t-a))(s)ds        & \text{a.e.} a\in(0,t)\cap[0,a_{1}] \\
   p(a-t, T)+\int_{a-t}^{a} \mathcal{G}(\hat{p}(\cdot,s+t-a))(s)ds      & \text{a.e.}a\in(t,a_{1}]
  \end{cases}.
\end{align}
Define $p(\cdot,t)=\hat{p}(\cdot,t-T)$ for $T<t\leq T+\hat{T}$. Then, $p\in L_{T+\hat{T}}$ and $p$ is a solution of (\ref{sec:inte11}) on $[0,T+\hat{T}]$.
\end{proposition}

Proposition \ref{sec:be4} has the following important consequence.

\begin{theorem}
Let H.1 hold, let $T>0$, let $\phi\in X$, and let $p\in L_{T}$. Then, $p$ is a solution of the problem (\ref{sec:1}) on $[0,T]$ if and only if $p$ is a solution of (\ref{sec:inte11}) on $[0, T]$.
\end{theorem}

We give the following definition of the maximal interval of existence of the solution of the problem (\ref{sec:1})-(\ref{sec:3}) since the continuability of the local solution of the problem (\ref{sec:1}) defined for all time depends on the existence of a priori bound.

\begin{definition}\label{sec:maxe1}
Let $\phi\in X$. \textit{The maximal interval of existence of the solution of the problem (\ref{sec:1})-(\ref{sec:3})}, denoted by $[0,T_{\phi})$, is the interval with the property that if $0<T<T_{\phi}$, then there exists $p\in L_{T}$ such that $p$ is a solution of the problem (\ref{sec:1}) on $[0,T]$. 
\end{definition}
The following definition says that by the uniqueness of solutions to (\ref{sec:1}) on $[0,T]$, if $0<T<\hat{T}$, $p\in L_{T}$, $\hat{p}\in L_{\hat{T}}$, such that $p$ is a solution of (\ref{sec:1}) on $[0,T]$ and $\hat{p}$ is a solution of (\ref{sec:1}) on $[0,\hat{T}]$, then $p$ and $\hat{p}$ have to agree on $[0,T]$.
\begin{definition}
Let $\phi\in X$ and let $p$ be a function from $[0,T_{\phi})$ to $X$. We define $p$ to be the solution of (\ref{sec:1}) on $[0,T_{\phi})$ provided that for all $T\in(0,T_{\phi})$, $p$ restricted to $[0,T]$ is the solution of (\ref{sec:1}) on $[0,T]$.
\end{definition}
Definition \ref{sec:maxe1} allows the possibility that $T_{\phi}=\infty$, which states as follows
\begin{theorem}
Let H.1 hold, let $\phi\in X$, and let $p$ be the solution of (\ref{sec:1}) on $[0,T_\phi)$. If $T_{\phi}<\infty$, then $\limsup_{t\rightarrow T_{\phi}^{-}}\left\|p(\cdot,t)\right\|_{X}=\infty$.
\end{theorem}

\section{Positivity of solutions}

We derive from H.1 that the birth function $\mathcal{F}$ and aging function $\mathcal{G}$ given as (\ref{sec:2})-(\ref{sec:3}) satisfy
\begin{align}
&\mathcal{F}(X_{+})\subset \mathbb{R}_{+}.\label{sec:be5}\\
&\text{There is an increasing function } c_{3}:[0,\infty)\rightarrow[0,\infty)\text{ such that if }r>0 \text{ and }\phi\in X_{+}\text{ with } \left\|\phi\right\|_{X}\leq r,\\
\nonumber&\text{then }\mathcal{G}(\phi)+c_{3}(r)\phi\in X_{+}.
\end{align}

The following results follow from (\cite{W}, section 2.4, pp.49) :

\begin{proposition}
Let H.1 hold and let $\phi\in X_{+}$. There exists $T>0$ and a function $p\in L_{T,+}$ satisfying (\ref{sec:1}).
\end{proposition}

\begin{theorem}\label{sec:T3.1} Let H.1 hold and let $\phi\in X_{+}$. The solution $p$ of problem (\ref{sec:1}) on $[0,T_{\phi})$ has the property that $p(\cdot,t)\in X_{+}$, for $0\leq t< T_{\phi}$. 
\end{theorem}

\begin{theorem}\label{sec:T3.3} Let H.1 hold and for each $\phi\in X_{+}$, let $p$ be the solution of the problem (\ref{sec:1}) on $[0,T_{\phi})$. Let there exists $\omega\in\mathbb{R}$ such that
\begin{align*}
\mathcal{F}(p(\cdot, t))+\int_{0}^{a_{1}}\mathcal{G}(p(\cdot, t))(a)da\leq\omega \int_{0}^{a_{1}}p(a,t)da\ t\in[0,T_{\phi}).
\end{align*}
Then, $T_{\phi}=\infty$ and 
\begin{align*}
\left\|p(\cdot,t)\right\|_{X}\leq e^{\omega t}\left\|\phi\right\|_{X}, \ \text{for}\ 0\leq t< T_{\phi}.
\end{align*}
\end{theorem}

\section{Regularity of solutions}

If we assume differentiability conditions on $\mathcal{F}, \mathcal{G}$ as in  (\ref{sec:2})-(\ref{sec:3}), we can obtain further regularity for solutions of the system (\ref{sec:1}).

\begin{definition}\label{def:Frechet}
Let $K$ be a mapping from a Banach space $X_{1}$ to a Banach space $X_{2}$. We require $K$ to be \textit{Frechet differentiable} at $\hat{x}\in D(K)$, in the following sense: $K(x)=K(\hat{x})+K'(\hat{x})(x-\hat{x})+o(x-\hat{x})$ for all $x\in D(K)$, where $K'(\hat{x})$ is a bounded linear operator from $X_{1}$ to $X_{2}$, $o $ is a function from $X_{1}$ to $X_{2}$, and $b$ is a continuous increasing function from $[0,\infty)$ to $[0,\infty)$ such that $b(0)=0$ and $\left\|o(x)\right\|\leq b(r)\left\|x\right\|$ for all $x\in X_{1}$ such that $\left\|x\right\|\leq r$. If $K$ is \textit{Frechet differentiable} at each $\hat{x}\in D(K)$, then $K$ is \textit{Lipschitz continuously Frechet differentiable} on $D(K)$, provided that $\left\|K'(x_{1})-K'(x_{2})\right\|\leq d(r)\left\|x_{1}-x_{2}\right\|$ for all $x_{1}, x_{2}\in D(K)$ such that $\left\|x_{1}\right\|, \left\|x_{2}\right\|\leq r$, where $d$ is a continuous increasing function from $[0,\infty)$ to $[0,\infty)$.
\end{definition}

By H.1, $\mathcal{F}, \ \mathcal{G}$ (\ref{sec:2})-(\ref{sec:3}) are continuously Frechet differentiable at $\hat{\phi}\in X$ in the sense of the definition (see \cite{W}, sec 2.6, pp.63), since
\begin{align}\label{sec:2'}
&(\mathcal{G}'(\hat{\phi})\phi)(a)=(\mathcal{C}_{1}(\hat{\phi})\phi)(a)+(\mathcal{C}_{2}(\hat{\phi})\phi)(a)\ \text{ a.e. }a\in[0,a_{1}],\ \text{ for }\phi\in X.
\end{align}
\begin{align}\label{sec:3'}
&\mathcal{F}'(\hat{\phi})(\phi)=\int^{a_{\max}}_{a_{\min}}\beta(a; \eta_{2}(Q_{0}\hat{\phi}))\phi(a)da\\
\nonumber&+\eta_{2}'(Q_{0}\hat{\phi})(Q_{0}\phi)\int^{a_{\max}}_{a_{\min}}\frac{\partial \beta(a, z)}{\partial z}|_{z = \eta_{2}(Q_{0}\hat{\phi})}\hat{\phi}(a)da,\ \text{ for }\phi\in X.
\end{align}
where,
\begin{align}\label{sec:4'}
(\mathcal{C}_{1}(\hat{\phi})\phi)(a)&=-\frac{\partial \mu_{0}(a, z)}{\partial z}|_{z = \eta_{0}(Q_{0}\hat{\phi})}\eta_{0}'(Q_{0}\hat{\phi})(Q_{0}\phi)\hat{\phi}(a)\\
\nonumber &-\frac{\partial \mu_{1}(a, z)}{\partial z}|_{z = \eta_{1}(Q_{1}\hat{\phi})}\eta_{1}'(Q_{1}\hat{\phi})(Q_{1}\phi)\hat{\phi}(a);\\
(\mathcal{C}_{2}(\hat{\phi})\phi)(a)&=-\mu_{0}(a, \eta_{0}(Q_{0}\hat{\phi}))\phi(a)-\mu_{1}(a, \eta_{1}(Q_{1}\hat{\phi}))\phi(a)-\mu_{2}(a)\phi(a).\label{sec:5'}
\end{align}

\begin{theorem} [\cite{W}, section 2.6, pp.63]Let H.1 hold. Let $\phi\in X$ such that $\phi$ is absolutely continuous on $[0,a_{1})$, $\phi'\in L^{1}$, $\phi(0)=\mathcal{F}(\phi)$ and let $p$ be the solution of the nonlinear problem (\ref{sec:1}) on $[0,T_{\phi})$. The following hold:

\begin{itemize}
\item[(i)] The mapping $t\mapsto p(\cdot, t)$ is continuously differentiable from $[0, T_{\phi})$ to $L^{1}$.
\item[(ii)] For $0\leq t \leq T< T_{\phi}$, $\left\|\frac{d}{dt}p(\cdot,t)\right\|_{X}\leq \left\|\phi'-\mathcal{G}(\phi)\right\|_{X}e^{t(\sup_{s\in [0,t]}|\mathcal{F}'(p(\cdot,s))|+\sup_{s\in [0,t]}|\mathcal{G}'(p(\cdot,s))|)}$.
\end{itemize}
\end{theorem}

\chapter{The induced nonlinear semigroup}\label{chap:solu}

In this chapter we first state some results in the theory of the nonlinear semigroup theory from \cite{W}, pp.74 and we will establish that the solutions of the model (\ref{sec:1}) form a strongly continuous nonlinear semigroup in the state space $X$ which is set up within the framework of general age-structured nonlinear population model from \cite{W} chapter 3.


\section{The induced nonlinear semigroup}

We first introduce the definition of a strongly continuous semigroup,

\begin{definition}
Let $Y$ be a Banach space and let $C$ be a closed set in $Y$. A strongly continuous semigroup in $C$ is a family of mappings $U(t), \ t\geq 0$, satisfying the following:
\begin{itemize}
\item[(i)] $U(t)$ is a continuous mapping from $C$ into $C$ for each $t\geq 0$.
\item[(ii)] $U(0)=I$ (where $I$ is the identity mapping in $Y$ restricted to $C$).
\item[(iii)] $U(t_{1}+t_{2})x=U(t_{1})U(t_{2})x$\  for all $t_{i}\geq 0, \ i=1,2$, $x\in C$.
\item[(iv)] $U(t)x$ is continuous in $t$ as a function from $[0,\infty)$ to $C$ for each fixed $x\in C$.
\end{itemize}
\end{definition}
Fundamental properties of the infinitesimal generator of a strongly continuous semigroup determine the regularity, asymptotic behavior of of the trajectories of the semigroup. We define,

\begin{definition}\label{sec:definge}
Let $C$ be a closed subset of the Banach space $Y$ and let $U(t), \ t\geq 0$, be a strongly continuous semigroup in $C$. The \textit{infinitesimal generator} of $U(t), t\geq 0$, is the mapping $\mathcal{A}$ from a subset of $C$ to $Y$ such that
\begin{align}\label{sec:inge11}
 \lim_{t \rightarrow 0^{+} }\frac{U(t)x-x}{t}=\mathcal{A}x.
\end{align}
with domain $D(\mathcal{A})$ the set of all $x\in C$ for which the limit (\ref{sec:inge11}) exists.
\end{definition}

In the following theorem, we establish that the generalized solutions of the associated nonlinear problem (\ref{sec:1})
form a strongly continuous nonlinear semigroup in $X_{+}$.

\begin{theorem}\label{sec:semiform1}
Let H.1 hold and for each $\phi\in X_{+}$ let the maximal interval of existence $[0,T_{\phi})$ of the solution of the problem (\ref{sec:1}) be $[0,\infty)$. Let $U(t),t\geq 0$, be the family of mappings in $X_{+}$ defined as follows: for $t\geq 0$, $\phi\in X_{+}$, $U(t)\phi:=p(\cdot,t)$, where $p$ is the generalized solution of (\ref{sec:1}) on $[0,\infty)$. Then, $U(t),t\geq 0$, is a strongly continuous nonlinear semigroup in $X_{+}$.
\end{theorem}

We next provide a $L^{1}$ norm estimation of the generalized solutions of the nonlinear problem (\ref{sec:1}).

\begin{theorem}\label{sec:T3.3}  Let H.1 hold and for each $\phi\in X_{+}$ let the maximal interval of existence $[0,T_{\phi})$ of the solution of the problem (\ref{sec:1}). Let $U(t), \ t\geq 0$, be the family of mappings in $X_{+}$ defined as follows: for $t \geq 0$, $\phi\in X_{+}$, $U(t)\phi:=p(\cdot, t)$, where $p$ is the generalized solution of the system (\ref{sec:1}) on $[0, \infty)$ as in Theorem \ref{sec:T3.1}. Let there exists $\omega\in\mathbb{R}$ such that
\begin{align*}
\mathcal{F}(p(\cdot, t))+\int_{0}^{a_{1}}\mathcal{G}(p(\cdot, t))(a)da\leq\omega \int_{0}^{a_{1}}p(a,t)da\ t\in[0,T_{\phi}).
\end{align*}

Then, $T_{\phi}=\infty$ and $\left\{U(t),\ t \geq 0\right\}$ in $X_{+}$ forms a positive strongly continuous nonlinear semigroup in $X_{+}$ satisfying
\begin{align*}
\left\|U(t)\phi\right\|_{X}\leq e^{\omega t}\left\|\phi\right\|_{X}, \ \text{for}\ \forall\phi\in X_{+}.
\end{align*}
\end{theorem}

If $\mathcal{F}$ and $\mathcal{G}$ are bounded linear operators, then the solutions of (\ref{sec:1}) may be associated with a strongly continuous semigroup of bounded linear operators in $X$. We state this result as follows:

\begin{theorem}[\cite{W}, sec 3.1, pp. 75]\label{sec:T33} 
Let $\mathcal{F}$ be a bounded linear operator from $X$ to $\mathbb{R}$ and let $\mathcal{G}$ be a bounded linear operator from $X$ to $X$. If $\phi\in X$, then the solution of (\ref{sec:1}) is defined on $[0,\infty)$. Further, the family of mappings $U(t),t\geq 0$, in $X$ defined by $U(t)\phi:=p(\cdot,t)$, where $p(\cdot, t)$ is the generalized solution of (\ref{sec:1}) on $[0,\infty)$, is a strongly continuous semigroup of bounded linear operators in $X$ satisfying
\begin{align*}
\left\|U(t)\right\|_{X}\leq e^{\omega t}, \ t\geq 0,\ \text{ where }\omega:=|\mathcal{F}|+|\mathcal{G}|.
\end{align*}
\end{theorem}


We state the definition of nonlinear accretive operators in Banach space $X$. 

\begin{definition}[\cite{W},sec 3.1, pp.77]
Let $A$ be a mapping from a subset of a Banach space $X$ to $X$. $A$ is said to be \textit{accretive} in $X$ provided that if $x_{1}$, $x_{2}$ belong to the domain $D(A)$ of $A$ and $\lambda>0$, then
\begin{align*}
\left\|(I+\lambda A)x_{1}-(I+\lambda A)x_{2}\right\|\geq\left\|x_{1}-x_{2}\right\|.
\end{align*}
\end{definition}

We state the following results from \cite{W}:

\begin{proposition}[M. Crandall and T. Liggett]
Let $\mathcal{A}$ be a mapping from a subset of a Banach space $X$ to $X$ and let there exist $\omega\in\mathbb{R}$ such that $\mathcal{A}+\omega I$ is accretive in $X$. Let there exist $\lambda_{1}>0$ such that if $0<\lambda<\lambda_{1}$, then $R(I+\lambda\mathcal{A})\supset {\overline{D(\mathcal{A})}}$. Then, for each $x\in \overline{D(\mathcal{A})}$
\begin{align*}
\lim_{n\rightarrow\infty}(I+t/n\mathcal{A})^{-n}x:=T(t)x\text{ exists uniformly in bounded intervals of }t\geq 0.
\end{align*}
Moreover, the family of mappings $T(t),t\geq 0$, so defined is a strongly continuous nonlinear semigroup in $\overline{D(\mathcal{A})}$ satisfying
\begin{align*}
\left\|T(t)x_{1}-T(t)x_{2}\right\|\leq e^{\omega t}\left\|x_{1}-x_{2}\right\|\text{ for all }t\geq 0, x_{1},x_{2}\in \overline{D(\mathcal{A})}.
\end{align*}
\end{proposition} 

\section{The infinitesimal generator associated with the problem (\ref{sec:1})}

It is well-known that if $T(t),\;t\geq 0$, is a strongly continuous semigroup of linear operators in a Banach space $X$, then $T(t), t\geq 0$, has a densely defined infinitesimal generator $\hat{B}$ and $T(t), t\geq 0$ is generated by $-\hat{B}$. For a strongly continuous nonlinear semigroup in a general Banach space this result may not be true and the infinitesimal generator may have a nondensely defined domain. In this section we will establish that the infinitesimal generator $\mathcal{A}$ of the strongly continuous nonlinear semigroup $U(t),t\geq 0$, associated with problem (\ref{sec:1}) has a densely defined domain.
We first define the infinitesimal generator of the strongly continuous nonlinear semigroup associated with the solutions of (\ref{sec:1}).

\begin{definition}
Let H.1 hold and define the mapping $\mathcal{A}$ from $X_{+}$ to $X$ by 
\begin{align}\label{sec:infi1}
&\mathcal{A}:=\phi'-\mathcal{G}(\phi)\;\text{ for }\phi\in D(\mathcal{A}), \text{ where }\\
&D(\mathcal{A})=\left\{\phi\in X_{+}:\phi \text{ is absolutely continuous on }[0,\infty), \phi'\in L^{1}, \text{ and } \phi(0)=\mathcal{F}(\phi)\right\}.
\end{align}  
\end{definition}

\begin{theorem}
Let H.1 hold, let $\mathcal{A}$ be defined as in (\ref{sec:infi1}), and let $U(t),t\geq 0$, be the strongly continuous nonlinear semigroup in $X_{+}$ as in Theorem \ref{sec:T3.3}. Then, $-\mathcal{A}$ is the infinitesimal generator of $U(t),t\geq 0$.
\end{theorem}

The following result establishes the fact that the domain of the infinitesimal generator is invariant under the nonlinear semigroup $U(t),t\geq 0$.
\begin{proposition}
Let H.1 hold, let $T_{\phi}=\infty$ for all $\phi\in X_{+}$, let $U(t),t\geq 0$, be the strongly continuous nonlinear semigroup in $X_{+}$ as in Theorem \ref{sec:T3.3}, and let $\mathcal{A}$ be the infinitesimal generator of $U(t),t\geq 0$, as in (\ref{sec:infi1}). If $t>0$, then $U(t)[D(\mathcal{A})]\subset D(\mathcal{A})$. Further, if $\phi\in D(\mathcal{A})$, then $\frac{d^{+}}{dt}U(t)\phi=\mathcal{A}U(t)\phi,\;t\geq 0$ holds.
\end{proposition}

If $\mathcal{F}$ and $\mathcal{G}$ are bounded linear operators as in theorem \ref{sec:T33}, then we have,
\begin{proposition} 
Let $\mathcal{F}$ be a bounded linear operator from $X$ to $\mathbb{R}$ and let $\mathcal{G}$ be a bounded linear operator from $X$ to $X$, and let $U(t),t\geq 0$, in $X$ be the strongly continuous semigroup of bounded linear operators in $X$ as in Theorem \ref{sec:T33}. The infinitesimal generator of $U(t),t\geq 0$, is 
\begin{align*}
&\mathcal{A}:=\phi'-\mathcal{G}(\phi)\;\text{ for }\phi\in D(\mathcal{A}), \text{ where }\\
&D(\mathcal{A})=\left\{\phi\in X_{+}:\phi \text{ is absolutely continuous on }[0,\infty), \phi'\in L^{1}, \text{ and } \phi(0)=\mathcal{F}(\phi)\right\}.
\end{align*}
Further,
For all $t\geq 0$, $U(t)[D(\mathcal{A})]\subset D(\mathcal{A})$ and $\frac{d}{dt}U(t)\phi=\mathcal{A}U(t)\phi=U(t)\mathcal{A}\phi$ for all $\phi\in D(\mathcal{A})$.  
\end{proposition}

\section{The exponential expression}

In this section we formulate the nonlinear semigroup by the exponential formula of its infinitesimal generator as in \cite{W}, sec. 3.3, pp. 91.  

\begin{proposition}
Let H.1 hold, let $\mathcal{A}$ be defined as in (\ref{sec:infi1}). Let $\mathcal{F}$ and $\mathcal{G}$ as in (\ref{sec:2})-(\ref{sec:3}) be globally Lipschitz continuous and let  $\omega=|\mathcal{F}|+|\mathcal{G}|$. The following hold:
\begin{itemize}
\item[(i)] $R(I+\lambda\mathcal{A})=X_{+}, \text{ for }0<\lambda<\omega^{-1}$;
\item[(ii)] $\mathcal{A}+\omega I \text{ is accretive in }X$;
\item[(iii)] $\overline{D(\mathcal{A})}=X_{+}$.
\end{itemize}
\end{proposition}

The following proposition gives the exponential expression when $\mathcal{F}$ and $\mathcal{G}$ are globally Lipschitz continuous. 

\begin{proposition}
Let H.1 hold, let $\mathcal{F}$ and $\mathcal{G}$ be globally Lipschitz continuous, let $U(t),t\geq 0$, be the strongly continuous nonlinear semigroup as in Theorem \ref{sec:T3.3}, and let $\mathcal{A}$ be defined as in (\ref{sec:infi1}). If $\phi\in X_{+}$, then
\begin{align}
\lim_{n\rightarrow\infty}(I+\frac{t}{n}\mathcal{A})^{-n}\phi=U(t)\phi\text{ uniformly in bounded intervals of }t\geq 0.
\end{align}
\end{proposition}

Next we consider that the birth functin $\mathcal{F}$ and the aging function $\mathcal{G}$ are locally Lipschitz continuous in the sense of (\ref{sec:LC111}) Using a truction method, we have the following, for more details we refer to \cite{W},

\begin{proposition}
Let H.1 hold and let $r> 0$. Define
\begin{align}\label{sec:tr1}
\mathcal{F}_{r}(\phi):=\begin{cases}
\mathcal{F}(\phi) & \text{if } \phi\in X \text{ and }\left\|\phi\right\|_{X}\leq r\\
\mathcal{F}(\frac{r\phi}{\left\|\phi\right\|_{X}}) & \text{if } \phi\in X \text{ and }\left\|\phi\right\|_{X}> r.
\end{cases}
\end{align}
\begin{align}\label{sec:tr2}
\mathcal{G}_{r}(\phi):=\begin{cases}
\mathcal{G}(\phi) & \text{if } \phi\in X \text{ and }\left\|\phi\right\|_{X}\leq r\\
\mathcal{G}(\frac{r\phi}{\left\|\phi\right\|_{X}}) & \text{if } \phi\in X \text{ and }\left\|\phi\right\|_{X}> r.
\end{cases}
\end{align}
Then, $\mathcal{F}_{r}$ and $\mathcal{G}_{r}$ satisfy the following:
\begin{align}
&|\mathcal{F}_{r}(\phi)-\mathcal{F}_{r}(\hat{\phi})|\leq 2c_{1}(r)\left\|\phi-\hat{\phi}\right\|_{X},\;\phi,\hat{\phi}\in X\\
&\left\|\mathcal{G}_{r}(\phi)-\mathcal{G}_{r}(\hat{\phi})\right\|\leq 2c_{2}(r)\left\|\phi-\hat{\phi}\right\|_{X},\;\phi,\hat{\phi}\in X\\
&\mathcal{F}_{r}(X_{+})\subset \mathbb{R}_{+}\\
&\mathcal{G}_{r}(\phi)+c_{3}(r_{1})\phi\in X_{+} \text{ for all } \phi\in X_{+} \text{ such that }\left\|\phi\right\|_{X} \leq r_{1}.
\end{align}

\end{proposition}

\begin{definition}
Let H.1 hold, let $r>0$, let $\mathcal{F}_{r}$ and $\mathcal{G}_{r}$ be defined as in (\ref{sec:tr1})-(\ref{sec:tr2}) and let mapping $\mathcal{A}_{r}$ from $X_{+}$ to $X$ be defined by 
\begin{align}\label{sec:tr3}
\mathcal{A}_{r}:=\phi'-\mathcal{G}_{r}(\phi)\;\text{ for }\phi\in D(\mathcal{A}_{r}),
\end{align}
where $D(\mathcal{A}_{r})=$ $\left\{\phi\in X_{+}:\phi \text{ is absolutely continuous on }[0,a_{1}],\phi'\in X, \phi(0)=\mathcal{F}_{r}(\phi)\right\}$
\end{definition}

\begin{theorem}
Let H.1 hold, let $T_{\phi}=\infty$ for each $\phi\in X_{+}$, and let $U(t), t \geq 0$, be the strongly continuous nonlinear semigroup in $X_{+}$ as in Theorem \ref{sec:T33}. If $\phi\in X_{+}$, $t>0$, $r\geq\sup_{0\leq s \leq t}\left\|U(s)\phi\right\|_{X}$, and $\mathcal{A}_{r}$ is defined as in (\ref{sec:tr3}), then
\begin{align}
\lim_{n\rightarrow\infty}(I+\frac{s}{n}\mathcal{A}_{r})^{-n}\phi=U(s)\phi\;\text{ uniformly for }s\in[0,t].
\end{align}
\end{theorem}

\begin{theorem}
Let H.1 hold and let $\mathcal{A}$ be defined as in (\ref{sec:infi1}). Then, $\overline{D(\mathcal{A})}=X_{+}$. Further, let $T_{\phi}=\infty$ for all $\phi\in L_{+}$, let $U(t),t \geq 0$, be the strongly continuous nonlinear semigroup in $X_{+}$ as in Theorem \ref{sec:T3.3}, let $\phi\in D(\mathcal{A})$, and let $u$ be a Lipschitz continuous function from $[0,t]$ to $L^{1}$ such that $u(0)=\phi$ and for almost all $s\in(0,t)$, $u$ is differentiable at $s$, $u(s)\in D(A)$, and $(d/ds)u(s)=-\mathcal{A} u(s)$. Then, $u(s)=S(s)\phi$ for $s\in[0,t]$.
\end{theorem}

We conclusion this section by stating some results for the strongly continuous semigroup of bounded linear operators associated with the problem (\ref{sec:1}) in the case that $\mathcal{F}$ and $\mathcal{G}$ are bounded linear operators. This results follows from \cite{W}, sec 3.3, pp. 98.

\begin{theorem}
Let $\mathcal{F}$ be a bounded linear operator from $L^{1}$ into $\mathbb{R}^{n}$, let $\mathcal{G}$ be a bounded linear operator from $L^{1}$ to $L^{1}$, let $U(t),t\geq 0$, be the strongly continuous semigroup of bounded linear operators in $L^{1}$ as in Theorem \ref{sec:T33}, let $\mathcal{B}$ be the infinitesimal generator of $U(t), t \geq 0$, and let $\omega=|\mathcal{F}|+|\mathcal{G}|$. The following hold:
\begin{itemize}
\item[(i)] $\overline{D(\mathcal{B})}=L^{1}$;
\item[(ii)] $-\mathcal{B}+\omega I \text{ is accretive in }L^{1}$;
\item[(iii)] $(I-\lambda \mathcal{B})^{-1}$ is a bounded everywhere defined linear operator in $L^{1}$, for all $0<\lambda<\omega^{-1}$;
\item[(iv)] For each $\phi\in L^{1}$, $\lim_{n\rightarrow\infty}(I-t/n\mathcal{B})^{-n}\phi=U(t)\phi$ uniformly in bounded intervals of $t$.
\end{itemize} 
\end{theorem}




\chapter{Equilibria and their stability}\label{chap:equi}

\section{Existence and uniqueness of either the positive equilibrium or a trivial equilibrium}\label{sec:nonequi}
In previous sections we use the semigroup theory to establish the existence of unique, postive solutions of the model for all positive time, the next natural step is to use the mathematical population models to predict whether or not a biological population will survive. More precisely, we are interested in analyzing the existence and uniqueness of the nontrivial (or trivial) steady states or equilibria of the models. Mathematical analysis of the existence and stability of a nontrivial equilibrium can be used to show the convergence of a population to the nontrivial steady state. In the first section of this chapter we study the problem of existence of a nontrivial equilibrium to (\ref{sec:1}). In the next section we will investigate the stability and instability of a nontrivial equilibrium of (\ref{sec:1}).

\subsection{Existence and uniqueness of an equilibrium solution}

We begin with some basic definition and proposition (see \cite{W}, sec 4.1, pp.136) used in finding the nontrivial equilibrium of the model (\ref{sec:1}).

\begin{definition}\label{def:equi}
Let H.1 hold, let $\phi\in X_{+}$, and let $p$ be the solution of (\ref{sec:1}) on $[0,T_{\phi})$. Then, $p$ is an equilibrium of (\ref{sec:1}) if and only if $T_{\phi}=\infty$ and $p(\cdot, t)=\phi$ for all $t\geq 0$. 
\end{definition}

We obtain the following proposition (see \cite{W}, sec 4.1, pp.136) from definition \ref{def:equi}:

\begin{proposition}
Let H.1 hold. Let $\mathcal{A}$ be defined as in Definition \ref{sec:definge}, let $\phi\in X_{+}$ and let $p$ be a solution of (\ref{sec:1}) on $[0,T_{\phi})$. Then, $p$ is an equilibrium of (\ref{sec:1}) if and only if $\mathcal{A}\phi=0$.
\end{proposition}

We define,

\begin{align*}
\Pi(b,a; s,t):=\exp[-\int^{a}_{b}(\mu_{0}(\hat{a}, s)+\mu_{1}(\hat{a},t)+ \mu_{2}(\hat{a}))d\hat{a}],\ \text{ for } 0\leq b\leq a.
\end{align*}
$\Pi(b,a;\eta_{0}(Q_{0}\phi), \eta_{1}(Q_{1}\phi))$ represents the probability that a member of the population of age b will survive to age a when exposed to the all-cause mortality $\mu_{0}(\hat{a}, \eta_{0}(Q_{0}\phi))$, the age-dependent and post-reproductive population-dependent mortality on the pre-reproductive population $\mu_{1}(\hat{a}, \eta_{1}(Q_{1}\phi))$, and the age-dependent mortality $\mu_{2}(\hat{a})$, for $\hat{a}\in[0,a_{1}]$ with age distribution $\phi\in X_{+}$.

We define the net reproduction function, for $\phi\in X_{+}$,
\begin{align}\label{sec:4.1}
\mathcal{R}(\phi):=\int_{a_{\min}}^{a_{\max}}\beta(a; \eta_{2}(Q_{0}\phi))\Pi(0,a; \eta_{0}(Q_{0}\phi), \eta_{1}(Q_{1}\phi))da.
\end{align}
Let $\mathcal{R}(0)$ be the intrinsic growth constant (IGC), which is an indicator of the capacity of the species to survive independent of the effects of crowding and all other nonlinear effects in this model.

Let $\Phi(x):=\int_{0}^{a_{\min}}\mu_{1}(\hat{a},x)d\hat{a},$ for $x \geq 0$. By H.1, we have $\Phi(0)=\eta_{1}(0)=0$. Let $\tilde{\Phi}(z)$ and $\tilde{\eta}_{1}(z)$ be the odd extensions of $\Phi(z)$ and $\eta_{1}(z)$ to $\mathbb{R}$. We deduce from H.1 that $\tilde{\Phi}^{-1}(z)$ is continuous for $z\in\mathbb{R}$. We make the following assumptions:
\begin{itemize}
\item[H.2.] Let $\int_{a_{\min}}^{a_{\max}}\beta(a; \eta_{2}(Q_{0}))e^{-\int_{0}^{a}(\mu_{0}(\hat{a}, \eta_{0}(Q_{0}))+\mu_{2}(\hat{a}))d\hat{a}}da < 1$, for $Q_{0}$ sufficiently large.
\item[H.3.] The mapping $\Theta: [0, \infty)\rightarrow \mathbb{R}$, defined by,
\begin{align}\label{sec:EQE1}
&\Theta(x):=\frac{Q_{1}(x)\int^{a_{1}}_{0}\omega_{0}(a)\Pi(0,a;\eta_{0}(x), \eta_{1}(Q_{1}(x)))da}{\int^{a_{1}}_{0}\omega_{1}(a)\Pi(0,a; \eta_{0}(x), \eta_{1}(Q_{1}(x)))da}\ \text{ for }x\geq 0,\\
\nonumber&\text{ is decreasing, where }\\
\nonumber&Q_{1}(x):=\tilde{\eta}_{1}^{-1}\circ \tilde{\Phi}^{-1}[\ln(\int_{a_{\min}}^{a_{\max}}\beta(a; \eta_{2}(x))e^{-\int_{0}^{a}(\mu_{0}(\hat{a}, \eta_{0}(x))+\mu_{2}(\hat{a}))d\hat{a}}da)].
\end{align}
\end{itemize}

\begin{remark}
\begin{itemize}
\item[]
\item[(a)]
H.2 requires that the intrinsic growth constant IGC=$\int_{a_{\min}}^{a_{\max}}\beta(a; \eta_{2}(Q_{0}))\Pi(0,a; \eta_{0}(Q_{0}),\eta_{1}(Q_{0}))da$ $\leq \int_{a_{\min}}^{a_{\max}}\beta(a; \eta_{2}(Q_{0}))e^{-\int_{0}^{a}(\mu_{0}(\hat{a}, \eta_{0}(Q_{0}))+\mu_{2}(\hat{a}))d\hat{a}}da < 1$ as  $Q_{0}$ sufficiently large.
\item[(b)]
H.3 requires that the fixed point mapping defined by $\Theta(x)$ for $x \geq 0$ is monotone decreasing. 
\item[(c)]
If H.3 is violated, we have constructed a numerical example to illustrate the existence of multiple poistive equilibria. 
\end{itemize}
\end{remark}

In the following theorem, we show the existence and uniqueness of either the nontrivial equilibrium or only the trivial equilibrium of (\ref{sec:1}) under certain conditions. The method involves a fixed-point approach in an operator theoretic framework combined with utilizing the special property of $\mu_{1}(a,x)$ for $(a,x)\in [0,a_{1}]\times (0,\infty)$. 

\begin{theorem}\label{sec:T4.1'} Let H.1-H.2 hold.

\begin{itemize}
\item[$(a)$] If IGC $> 1$, there exists a positive equilibrium for the system (\ref{sec:1}).
\item[$(b)$] If IGC $\leq 1$, the trivial equilibrium is the only equilibrium solution for the system (\ref{sec:1}).
\item[$(c)$] If IGC $> 1$ and H.3 is satisfied, then system (\ref{sec:1}) admits a unique positive equilibrium.
\end{itemize}

\end{theorem}

\begin{proof}A time-independent solution $\hat{\phi} \in X_{+}$ of the system (\ref{sec:1}) satisfies:
\begin{align}\label{sec:4.2}
&\hat{\phi}'(a)=-(\mu_{0}(a, \eta_{0}(Q_{0}\hat{\phi}))+\mu_{1}(a, \eta_{1}(Q_{1}\hat{\phi}))+\mu_{2}(a))\hat{\phi}(a).
\end{align}
\begin{align}\label{sec:4.3}
&\hat{\phi}(0)=\int^{a_{\max}}_{a_{\min}}\beta(a; \eta_{2}(Q_{0}\hat{\phi}))\hat{\phi}(a)da.
\end{align}
Because $0$ satisfies (\ref{sec:4.2})-(\ref{sec:4.3}), there always is the trivial equilibrium $\hat{\phi}=0$. We solve 
(\ref{sec:4.2})-(\ref{sec:4.3}) to obtain,
\begin{align}\label{sec:4.4}
\hat{\phi}(a)&=\hat{\phi}(0)\Pi(0,a; \eta_{0}(Q_{0}\hat{\phi}), \eta_{1}(Q_{1}\hat{\phi})).
\end{align}
We plug (\ref{sec:4.4}) into the initial condition (\ref{sec:4.3}) to get,
\begin{align*}
\hat{\phi}(0)=\hat{\phi}(0)\int_{a_{\min}}^{a_{\max}}\beta(a; \eta_{2}(Q_{0}\hat{\phi}))\Pi(0,a; \eta_{0}(Q_{0}\hat{\phi}), \eta_{1}(Q_{1}\hat{\phi}))da.
\end{align*}
Dividing both sides by $\hat{\phi}(0)$ gives,
\begin{align}\label{sec:4.5}
1=\int_{a_{\min}}^{a_{\max}}\beta(a; \eta_{2}(Q_{0}\hat{\phi}))\Pi(0,a; \eta_{0}(Q_{0}\hat{\phi}), \eta_{1}(Q_{1}\hat{\phi}))da.
\end{align}
We use the assumption, $\mu_{1}(a, z)=0$ when $a> a_{\min}$, to obtain,
\begin{align}\label{sec:4.6}
&1=e^{-\int_{0}^{a_{\min}}\mu_{1}(\hat{a},\eta_{1}(Q_{1}\hat{\phi}))d\hat{a}}\\
\nonumber&\times\int_{a_{\min}}^{a_{\max}}\beta(a; \eta_{2}(Q_{0}\hat{\phi}))e^{-\int_{0}^{a}(\mu_{0}(\hat{a}, \eta_{0}(Q_{0}\hat{\phi}))+\mu_{2}(\hat{a}))d\hat{a}}da.
\end{align}
We take the natural logarithm on both sides of (\ref{sec:4.6}), to obtain,
\begin{align}\label{sec:4.6'}
&\tilde{\Phi}(\tilde{\eta}_{1}(Q_{1}\hat{\phi}))=\ln[\int_{a_{\min}}^{a_{\max}}\beta(a; \eta_{2}(Q_{0}\hat{\phi}))e^{-\int_{0}^{a}(\mu_{0}(\hat{a}, \eta_{0}(Q_{0}\hat{\phi}))+\mu_{2}(\hat{a}))d\hat{a}}da].
\end{align}
Since $\eta_{1}(z)$ for $z\geq 0$ is strictly monotone increasing, onto, continuous and $\tilde{\eta}_{1}(z)$ is the odd continuous extension of $\eta_{1}(z)$ to $\mathbb{R}$, we have $\tilde{\eta}_{1}(z)$ is invertible and its inverse is continuous on $\mathbb{R}$.  We apply $\tilde{\eta}_{1}^{-1}\circ \tilde{\Phi}^{-1}$ on both sides of (\ref{sec:4.6'}) to obtain,
\begin{align*}
Q_{1}\hat{\phi}=\tilde{\eta}_{1}^{-1}\circ \tilde{\Phi}^{-1}[\ln(\int_{a_{\min}}^{a_{\max}}\beta(a; \eta_{2}(Q_{0}\hat{\phi}))e^{-\int_{0}^{a}(\mu_{0}(\hat{a}, \eta_{0}(Q_{0}\hat{\phi}))+\mu_{2}(\hat{a}))d\hat{a}}da)].
\end{align*}
We define $Q_{1}:[0,\infty)\rightarrow \mathbb{R}$ by, for $Q_{0}\geq 0$,
\begin{align}\label{sec:4.7}
Q_{1}(Q_{0}):=\tilde{\eta}_{1}^{-1}\circ \tilde{\Phi}^{-1}[\ln(\int_{a_{\min}}^{a_{\max}}\beta(a; \eta_{2}(Q_{0}))e^{-\int_{0}^{a}(\mu_{0}(\hat{a}, \eta_{0}(Q_{0}))+\mu_{2}(\hat{a}))d\hat{a}}da)].
\end{align}
If $\mathcal{R}(0)=IGC >1$, we derive from H.1 that $\int_{0}^{a_{\min}}\mu_{1}(a, 0)da=0$, $\eta_{1}(0)=0$ and from (\ref{sec:4.7}), we obtain,
\begin{align*}
Q_{1}(0)&=\tilde{\eta}_{1}^{-1}\circ \tilde{\Phi}^{-1}[\ln(\int_{a_{\min}}^{a_{\max}}\beta(a; \eta_{2}(0))e^{-\int_{0}^{a}(\mu_{0}(\hat{a}, \eta_{0}(0))+\mu_{1}(\hat{a}))d\hat{a}}da)]\\
&=\eta_{1}^{-1}\circ \Phi^{-1}[\ln(\int_{a_{\min}}^{a_{\max}}\beta(a; \eta_{2}(0))\Pi(0,a; \eta_{0}(0), \eta_{1}(0))da)]>0.
\end{align*}
We have constructed a continuous function $Q_{1}(z)$
for $z\geq 0$. By H.2, $Q_{1}(Q_{0})<0$ for $Q_{0}$ sufficiently large while $Q_{1}(0)>0$. Then, we apply the intermediate value theorem to obtain some $\tilde{Q}_{0}>0$ such that
$Q_{1}(\tilde{Q}_{0})=0$.
In order to derive a second equation for $Q_{0}\hat{\phi}$ and $Q_{1}\hat{\phi}$, we integrate $\omega_{i}(a)\hat{\phi}(a)$ (\ref{sec:4.4}), for $i=0,1$, over $[0,a_{1}]$ to obtain,
\begin{align}\label{sec:4.8}
Q_{0}\hat{\phi}&=\int_{0}^{a_{1}}\omega_{0}(a)\hat{\phi}(a)da=\hat{\phi}(0)\int^{a_{1}}_{0}\omega_{0}(a)\Pi(0,a;\eta_{0}(Q_{0}\hat{\phi}), \eta_{1}(Q_{1}\hat{\phi}))da;
\end{align}
\begin{align}\label{sec:4.9}
Q_{1}\hat{\phi}&=\int_{0}^{a_{1}}\omega_{1}(a)\hat{\phi}(a)da=\hat{\phi}(0)\int^{a_{1}}_{0}\omega_{1}(a)\Pi(0,a;\eta_{0}(Q_{0}\hat{\phi}), \eta_{1}(Q_{1}\hat{\phi}))da.
\end{align}
We use (\ref{sec:4.8}) to divide (\ref{sec:4.9}) to obtain,
\begin{align}\label{sec:4.10}
\frac{Q_{0}\hat{\phi}}{Q_{1}\hat{\phi}}=\frac{\int^{a_{1}}_{0}\omega_{0}(a)\Pi(0,a;\eta_{0}(Q_{0}\hat{\phi}), \eta_{1}(Q_{1}\hat{\phi}))da}{\int^{a_{1}}_{0}\omega_{1}(a)\Pi(0,a;\eta_{0}(Q_{0}\hat{\phi}), \eta_{1}(Q_{1}\hat{\phi}))da}.
\end{align}
From (\ref{sec:4.7}) and (\ref{sec:4.10}), we define $\Theta: [0, \infty)\rightarrow \mathbb{R}$:
\begin{align}\label{sec:4.11}
\Theta(Q_{0}):=\frac{Q_{1}(Q_{0})\int^{a_{1}}_{0}\omega_{0}(a)\Pi(0,a;\eta_{0}(Q_{0}), \eta_{1}(Q_{1}(Q_{0})))da}{\int^{a_{1}}_{0}\omega_{1}(a)\Pi(0,a; \eta_{0}(Q_{0}), \eta_{1}(Q_{1}(Q_{0})))da},\ \text{ for }Q_{0}\geq 0.
\end{align}
Our next goal is to show there exists a fixed point for $\Theta(z)$, $z\geq 0$.
It follows from (\ref{sec:4.11}) that,
\begin{align*}
\Theta(0)=\frac{Q_{1}(0)\int^{a_{1}}_{0}\omega_{0}(a)\Pi(0,a;\eta_{0}(0),\eta_{1}(Q_{1}(0)))da}{\int^{a_{1}}_{0}\omega_{1}(a)\Pi(0,a;\eta_{0}(0),\eta_{1}(Q_{1}(0)))da}>0.
\end{align*}
We compute $\Theta(z)$ at $\tilde{Q}_{0}$,
\begin{align*}
\Theta(\tilde{Q}_{0})&=\frac{Q_{1}(\tilde{Q}_{0})\int^{a_{1}}_{0}\omega_{0}(a)\Pi(0,a;\eta_{0}(\tilde{Q}_{0}),\eta_{1}(Q_{1}(\tilde{Q}_{0})))da}{\int^{a_{1}}_{0}\omega_{1}(a)\Pi(0,a;\eta_{0}(\tilde{Q}_{0}),\eta_{1}(Q_{1}(\tilde{Q}_{0})))da}=0<\tilde{Q}_{0}.
\end{align*}
Therefore, the function $\Theta(z)$ is continuous for $z\geq 0$ and $\Theta(0)>0$, $\Theta(\tilde{Q}_{0})<\tilde{Q}_{0}$. We apply the intermediate value theorem again to obtain some $\hat{Q}_{0}>0$, ($0<\hat{Q}_{0}<\tilde{Q}_{0}$) such that $\Theta(\hat{Q}_{0})=\hat{Q}_{0}$. Furthermore, $\hat{Q}_{1}$ is uniquely given by (\ref{sec:4.7}). The following form of a positive equilibrium $\hat{\phi}$ follows from (\ref{sec:4.4}), (\ref{sec:4.8})-(\ref{sec:4.9}).
\begin{align}\label{sec:4.12}
\hat{\phi}(a)&=\frac{\hat{Q}_{i}\Pi(0,a;\eta_{0}(\hat{Q}_{0}),\eta_{1}(\hat{Q}_{1}))}{\int^{a_{1}}_{0}\omega_{i}(a)\Pi(0,a;\eta_{0}(\hat{Q}_{0}),\eta_{1}(\hat{Q}_{1}))da}, \ i=0 \text{ or }1.
\end{align}
The uniqueness of the positive equilibrium is a consequence of H.3.

\end{proof}

In the following Theorem \ref{sec:E4.3} and Example \ref{sec:E4.13}, we let $\eta_{i}(x)=x$ for $x\geq 0$, $i=0,1,2$, $\omega_{0}(a)=\chi_{[0,a_{1}]}(a)$ and $\omega_{1}(a)=\chi_{[a_{\max},a_{1}]}(a)$ for $a\in [0,a_{1}]$, where $\chi_{[c,d]}(a)=1$ if $a\in[c,d]\subseteq [0,a_{1}]$; otherwise, $\chi_{[c,d]}(a)=0$ if $a\notin[c,d]$. It then follows that $Q_{0}(t)=T(t)=\int_{0}^{a_{1}}p(a,t)da$, and $Q_{1}(t)=S(t)=\int_{a_{\max}}^{a_{1}}p(a,t)da$, where $T(t),\: S(t)$ denote total population and senescent population at time $t$, respectively. Furthermore, $J(t)=\int_{0}^{a_{\min}}p(a,t)da$, and $R(t)=\int_{a_{\min}}^{a_{\max}}p(a,t)da$, where $J(t),\: R(t)$ denote juvenile population and reproductive population at time $t$, respectively. In the following Theorem \ref{sec:E4.3}, we provide a sufficient condition for the uniqueness of the nontrivial equilibrium for (\ref{sec:E4.8}) by applying Theorem \ref{sec:T4.1'}:

\begin{theorem}\label{sec:E4.3}
Let H.1 hold. Let $\beta(a,z)=\beta(a)$ for $a\in[a_{\min},a_{\max}]$ and $z\geq 0$. 
 Let $\mu_{0}(a,z)=\eta(a)z$ and $\mu_{1}(a,z)=\mu(a)z$, for $(a,z)\in [0,a_{1}]\times [0,\infty)$, where $\eta,\ \mu\in C_{+}[0,a_{1}]$, and $\mu(a)=0$ for $a>a_{\min}$. Consider the following system,
\begin{align}\label{sec:E4.8}
&p_{t}(a,t)+p_{a}(a,t)=-[\eta(a)T(t)+\mu(a)S(t)+\mu_{2}(a)]p(a,t),\\
\nonumber & 0<a<a_{1},\ t>0,\\
\nonumber &p(0,t)=\int_{a_{\min}}^{a_{\max}}\beta(a)p(a,t)da,\ t>0,\\
\nonumber &p(a,0)=p_{0}(a),\ 0<a<a_{1}.
\end{align}
Then,
\begin{itemize}
\item[$(a)$] If IGC $> 1$, there exists a positive equilibrium for the system (\ref{sec:E4.8}).
\item[$(b)$] If IGC $\leq 1$, the trivial equilibrium is the only equilibrium solution for the system (\ref{sec:E4.8}).
\item[$(c)$] If $e^{\Lambda}\geq IGC > 1$, then system (\ref{sec:E4.8}) admits a unique positive equilibrium.
\end{itemize}
where, $\Lambda:=\frac{\int^{a_{\min}}_{0}\eta(\hat{a})d\hat{a}}{\int^{a_{1}}_{0}\eta(\hat{a})d\hat{a}}(1+\frac{\bar{\eta}}{\bar{\mu}}\frac{\int^{a_{1}}_{0}e^{-\int^{a}_{0}\mu_{2}(\hat{a})d\hat{a}}da}{\int^{a_{1}}_{a_{\max}}e^{-\int^{a}_{0}\mu_{2}(\hat{a})d\hat{a}}da})$
with $\bar{\mu}:=\frac{\int_{0}^{a_{\min}}\mu(a)da}{a_{\min}},\ \bar{\eta}:=\frac{\int_{0}^{a_{\min}}\eta(a)da}{a_{\min}}$.
\end{theorem}

\begin{proof}
For part (a), one can use a similar argument as in Theorem \ref{sec:T4.1'}(a) to show the existence of a positive equilibrium as in (\ref{sec:4.12}).  Part (b) follows from Theorem \ref{sec:T4.1'}(b).
 Our goal is to show, similarly as in Theorem \ref{sec:T4.1'}(c), that if $e^{\Lambda}\geq IGC > 1$, $\Theta$ as in (\ref{sec:4.11}) is monotone, and therefore, there exists a unique nontrivial equilibrium. We derive from (\ref{sec:4.7}) that,
\begin{align}\label{sec:44.1}
\mathcal{S}(E_{T})=\frac{\ln(\int_{a_{\min}}^{a_{\max}}\beta(a)e^{-\int_{0}^{a}(\eta(\hat{a})E_{T}+\mu_{2}(\hat{a}))d\hat{a}}da)}{\int_{0}^{a_{\min}}\mu(\hat{a})d\hat{a}}\ \text{ for }E_{T} \geq 0.
\end{align}
We observe from (\ref{sec:44.1}) that $\mathcal{S}(E_{T})$ is differentiable for $E_{T} \geq 0$, and it is monotone decreasing, since its derivative satisfies,
\begin{align}\label{sec:44.2}
-\mathcal{S}'(E_{T})&=\frac{\int^{a_{\max}}_{a_{\min}}\beta(a)e^{-\int^{a}_{0}(\mu_{2}(\hat{a})+\eta(\hat{a}) E_{T})d\hat{a}}\int_{0}^{a}\eta(\hat{a})d\hat{a}da}{\int^{a_{\max}}_{a_{\min}}\beta(a)e^{-\int^{a}_{0}(\mu_{2}(\hat{a})+\eta(\hat{a}) E_{T})d\hat{a}}da\int^{a_{\min}}_{0}\mu(a)da}\geq\frac{\int_{0}^{a_{\min}}\eta(\hat{a})d\hat{a} }{\int^{a_{\min}}_{0}\mu(a)da}.
\end{align}
If $IGC=\int_{a_{\min}}^{a_{\max}}\beta(a)e^{-\int_{0}^{a}\mu_{2}(\hat{a})d\hat{a}}da>1$, we have $\mathcal{S}(0)>0$.  We derive from (\ref{sec:44.1}) that $\mathcal{S}(E_{T})\rightarrow -\infty$, as $E_{T}\rightarrow \infty$, and $\mathcal{S}(E_{T})$ is a continuous function for $E_{T}\geq 0$. Therefore, by the intermediate value theorem, there exists a zero $\tilde{T}$ of $\mathcal{S}(E_{T})$ and it is unique since $\mathcal{S}(E_{T})$ is monotone decreasing for $E_{T}\geq 0$. Moreover, we estimate that $\tilde{T}$ satisfies:
\begin{align}\label{sec:44.3}
&\frac{\ln(\int^{a_{\max}}_{a_{\min}}\beta(a)e^{-\int^{a}_{0}\mu_{2}(\hat{a})d\hat{a}}da)}{\int_{0}^{a_{\max}}\eta(\hat{a})d\hat{a} }\leq \tilde{T}\leq \frac{\ln(\int^{a_{\max}}_{a_{\min}}\beta(a)e^{-\int^{a}_{0}\mu_{2}(\hat{a})d\hat{a}}da)}{\int_{0}^{a_{\min}}\eta(\hat{a})d\hat{a} }.
\end{align}
We deduce from (\ref{sec:EQE1}) that $\Theta(x)=x$ if and only if $\Delta(x)=0$ for $x\geq 0$, where $\Delta: [0,\infty) \rightarrow \mathbb{R}$ is defined by, for $x\geq 0$
\begin{align}\label{sec:44.4}
&\Delta(x):=\mathcal{S}(x)\int^{a_{\min}}_{0}e^{-\int_{0}^{a}\mu_{2}(\hat{a})d\hat{a}}e^{-x\int_{0}^{a}\eta(\hat{a})d\hat{a}}e^{\mathcal{S}(x)\int^{a_{\min}}_{a}\mu(\hat{a})d\hat{a}}da\\
\nonumber&+\mathcal{S}(x)\int^{a_{1}}_{a_{\min}}e^{-\int_{0}^{a}\mu_{2}(\hat{a})d\hat{a}}e^{-x\int_{0}^{a}\eta(\hat{a})d\hat{a}}da-x\int^{a_{1}}_{a_{\max}}e^{-\int_{0}^{a}\mu_{2}(\hat{a})d\hat{a}}e^{-x\int_{0}^{a}\eta(\hat{a})d\hat{a}}da.
\end{align}
Our next goal is to show that if $e^{\Lambda}>IGC$, $\Delta(z)$ is non-increasing for $z \geq 0$. It is easy to show that $\Delta'(E_{T}) \leq 0$ for $E_{T}\geq 0$ if and only if (\ref{sec:44.5}) holds, for $E_{T}\geq 0$:
\begin{align}\label{sec:44.5}
&-[\mathcal{S}'(E_{T})\int^{a_{\min}}_{0}e^{-\int_{0}^{a}\mu_{2}(\hat{a})d\hat{a}}e^{-E_{T}\int^{a}_{0}\eta(\hat{a})d\hat{a} }e^{\mathcal{S}(E_{T})\int^{a_{\min}}_{a}\mu(\hat{a})d\hat{a}}da\\
\nonumber&+\mathcal{S}(E_{T})\int^{a_{\min}}_{0}e^{-\int_{0}^{a}\mu_{2}(\hat{a})d\hat{a}}e^{-E_{T}\int^{a}_{0}\eta(\hat{a})d\hat{a}}e^{\mathcal{S}(E_{T})\int^{a_{\min}}_{a}\mu(\hat{a})d\hat{a}}(-\int^{a}_{0}\eta(\hat{a})d\hat{a}\\
\nonumber&+\mathcal{S}'(E_{T})\int^{a_{\min}}_{a}\mu(\hat{a})d\hat{a})da+\mathcal{S}'(E_{T})\int^{a_{1}}_{a_{\min}}e^{-\int_{0}^{a}\mu_{2}(\hat{a})d\hat{a}}e^{-E_{T}\int^{a}_{0}\eta(\hat{a})d\hat{a}}da\\
\nonumber&+\mathcal{S}(E_{T})\int^{a_{1}}_{a_{\min}}e^{-\int_{0}^{a}\mu_{2}(\hat{a})d\hat{a}}e^{-E_{T}\int^{a}_{0}\eta(\hat{a})d\hat{a}}(-\int^{a}_{0}\eta(\hat{a})d\hat{a})da]\geq-\int^{a_{1}}_{a_{\max}}e^{-\int_{0}^{a}\mu_{2}(\hat{a})d\hat{a}}\\
\nonumber&\times e^{-E_{T}\int^{a}_{0}\eta(\hat{a})d\hat{a}}da+E_{T}\int^{a_{1}}_{a_{\max}}e^{-\int_{0}^{a}\mu_{2}(\hat{a})d\hat{a}}e^{-E_{T}\int^{a}_{0}\eta(\hat{a})d\hat{a}}(\int^{a}_{0}\eta(\hat{a})d\hat{a})da.
\end{align}
Furthermore, we have, for $0\leq E_{T}\leq \tilde{T}$,
\begin{align}\label{sec:44.6}
\text{The left hand side of }(\ref{sec:44.5})&\geq-\mathcal{S}'(E_{T})\int^{a_{1}}_{0}e^{-\int_{0}^{a}\mu_{2}(\hat{a})d\hat{a}}e^{-E_{T}\int^{a}_{0}\eta(\hat{a})d\hat{a}}da\\
\nonumber&\geq\frac{\int^{a_{\min}}_{0}\eta(\hat{a})d\hat{a}}{\int^{a_{\min}}_{0}\mu(a)da}\int^{a_{1}}_{0}e^{-\int_{0}^{a}\mu_{2}(\hat{a})d\hat{a}}e^{-E_{T}\int^{a}_{0}\eta(\hat{a})d\hat{a}}da;\\
\nonumber\text{The right hand side of }(\ref{sec:44.5})&\leq (\tilde{T}\int^{a_{1}}_{0}\eta(\hat{a})d\hat{a}-1)\int^{a_{1}}_{a_{\max}}e^{-\int_{0}^{a}\mu_{2}(\hat{a})d\hat{a}}e^{-E_{T}\int^{a}_{0}\eta(\hat{a})d\hat{a}}da.
\end{align}
Let $K(x):=\frac{\int^{a_{1}}_{0}e^{-\int_{0}^{a}\mu_{2}(\hat{a})d\hat{a}}e^{-x\int^{a}_{0}\eta(\hat{a})d\hat{a}}da}{\int^{a_{1}}_{a_{\max}}e^{-\int_{0}^{a}\mu_{2}(\hat{a})d\hat{a}}e^{-x\int^{a}_{0}\eta(\hat{a})d\hat{a}}da}$, for $x\geq 0$. One can easily verify that $K(x)$ is non-decreasing for $x \geq 0$. Therefore, if $e^{\Lambda}>IGC$, we obtain, for $0\leq E_{T}\leq \tilde{T}$,
\begin{align*}
\nonumber&\frac{\int^{a_{\min}}_{0}\eta(\hat{a})d\hat{a} }{\int^{a_{\min}}_{0}\mu(a)da}\frac{\int^{a_{1}}_{0}e^{-\int_{0}^{a}\mu_{2}(\hat{a})d\hat{a}}e^{-E_{T}\int^{a}_{0}\eta(\hat{a})d\hat{a}}da}{\int^{a_{1}}_{a_{\max}}e^{-\int_{0}^{a}\mu_{2}(\hat{a})d\hat{a}}e^{-E_{T}\int^{a}_{0}\eta(\hat{a})d\hat{a}}da}\geq\frac{\int^{a_{\min}}_{0}\eta(\hat{a})d\hat{a} }{\int^{a_{\min}}_{0}\mu(a)da}\frac{\int^{a_{1}}_{0}e^{-\int_{0}^{a}\mu_{2}(\hat{a})d\hat{a}}da}{\int^{a_{1}}_{a_{\max}}e^{-\int_{0}^{a}\mu_{2}(\hat{a})d\hat{a}}da}\\
&\geq\frac{\int^{a_{1}}_{0}\eta(\hat{a})d\hat{a}\ln(\int^{a_{\max}}_{a_{\min}}\beta(a)e^{-\int^{a}_{0}\mu_{2}(\hat{a})d\hat{a}}da)-\int^{a_{\min}}_{0}\eta(\hat{a})d\hat{a}}{ \int^{a_{\min}}_{0}\eta(\hat{a})d\hat{a}}\geq(\tilde{T}\int^{a_{1}}_{0}\eta(\hat{a})d\hat{a}-1).
\end{align*}
It is readily seen that (\ref{sec:44.5}) holds and the uniqueness of the nontrivial equilibrium follows.
\end{proof}

\subsection{Numerical examples}

Let $\hat{\phi}$ given by (\ref{sec:4.12}) be a nontrivial equilibrium of (\ref{sec:1}). We define the mean age of a female by $\int_{0}^{a_{1}}a\hat{\phi}(a)da$, and the average prospective lifespan from birth of a female by $\frac{\int_{0}^{a_{1}}\hat{\phi}(a)da}{\int_{0}^{a_{1}}(\mu_{0}(a;\hat{T})+\mu_{1}(a;\hat{S})+\mu_{2}(a))\hat{\phi}(a)da}$ (see \cite{Ke}). We hypothesize that there is no significant change in the reproductive age interval of a female (who survives past juvenility) throughout the recent evolutionary time from early humans living in hunter-gatherer society until modern agricultural civilization \cite{Gur,S,S1}.
Therefore, we choose baseline values $a_{\min}=15$ years, $a_{\max}=35$ years and $a_{1}=80$ years in all numerical examples used throughout this paper.

\begin{example}

The baseline parameters are set as 
$\mu_{0}(a,T)=\eta(a)T=3\times 10^{-6}$T, for $a\in[0, a_{1}]$ and $T\geq 0$, $\mu_{1}(a,S)=\mu(a)S=10^{-7}(-a +a_{\min})S$, if $ a \leq a_{\min}$ and $S\geq 0$; otherwise, if $a>a_{\min}$ and $S\geq 0$, $\mu_{1}(a,S)=0$; $\beta(a)= 0.5(a - a_{\min})e^{-0.4(a - a_{\min})}$, if $a>a_{\min}$; otherwise, if $a\leq a_{\min}$, $\beta(a)=0$ and $\mu_{2}(a)=0.03+0.01e^{-0.04a}$ for $a\in[0, a_{1}]$. We numerically find that if $a_{\min}=15$ years and $a_{\max}=35$ years, IGC$\approx 1.5$ (indicated by yellow dots and dashed lines in Fig.\ref{sec:Fig1}). As Figure \ref{sec:Fig1}(A) illustrates, the age structure of human beings is robust with extended juvenility.
 When $a_{\max}=35$ years and all other baseline parameters are held fixed, increasing the juvenility, the reproductive period decreases and therefore IGC decreases and falls below $1$ as $a_{\min}$ exceeds $\approx 25$ years (see Figure \ref{sec:Fig1}(B)). Figure \ref{sec:Fig1}(C) shows that when $a_{\min}=15$ years is held fixed, IGC increases sharply and then more slowly as $a_{\max}$ increases. We numerically find that the average lifespan and the mean age of a female at the nontrivial equilibrium is $\approx 22.4$ years and $\approx 17.6$ years, respectively.
 We refer to \cite{B} for numerical simulations and sensitivity analyses of (\ref{sec:E4.8}).
\end{example}

In the following example, we revisit the system (\ref{sec:E4.8}), when it is subject to a nonlinear boundary condition.

\begin{example}\label{sec:E4.13}
Consider the following system:
\begin{align}\label{sec:E4.18}
&p_{t}(a,t)+p_{a}(a,t)=-[\eta(a)T(t)+\mu(a)S(t)+\mu_{2}(a)]p(a,t),\\
\nonumber & 0<a<a_{1},\ t>0,\\
\nonumber &p(0,t)=\int_{a_{\min}}^{a_{\max}}\tilde{\beta}(a)f(T(t))p(a,t)da,\ t>0,\\
\nonumber &p(a,0)=p_{0}(a),\ 0<a<a_{1}.
\end{align}
All parameters are set as in Table \ref{table:ba}. Numerical illustrations of $\beta$ and $\mu_{0}$ are given by Figure \ref{sec:Fig2}A and Figure \ref{sec:Fig3}A. Since $f(x)=\frac{1}{1+c_{2}x}$ for $x\geq 0$ and $c_{2}>0$, it directly follows that $f(x)\rightarrow 0$, as $x\rightarrow \infty$. We numerically find that IGC$\approx 1.5$ for values in Table \ref{table:ba}. One could argue as in Theorem \ref{sec:T4.1'}(a) to obtain the existence of a positive equilibrium of (\ref{sec:E4.18}). Moreover, since $\mathcal{S}(E_{T})=\frac{1}{\int_{0}^{a_{\min}}\mu(\hat{a})d\hat{a}}\ln(\int_{a_{\min}}^{a_{\max}}\frac{\tilde{\beta}(a)}{1+c_{2}E_{T}}$ $e^{-\int_{0}^{a}(\eta(\hat{a})E_{T}+\mu_{2}(\hat{a}))d\hat{a}}da)$, for $E_{T} \geq 0$, it is readily seen that $\mathcal{S}'(E_{T})\leq 0$ for $E_{T} \geq 0$. We numerically find that $\tilde{T}\int_{0}^{a_{1}}\eta(\hat{a})d\hat{a}<1$, which implies that condition (\ref{sec:44.5}) holds, where $\tilde{T}$ is the zero of $\mathcal{S}(E_{T})$ for $E_{T} \geq 0$. Therefore, $\Delta(E_{T})$ given by (\ref{sec:44.4}) is monotone decreasing for $0 \leq E_{T}\leq \tilde{T}$ and the system (\ref{sec:E4.18}) admits a unique positive equilibrium. We numerically find that the average lifespan of a female at the nontrivial equilibrium is $\approx 32.9$ years. In the baseline model, with parametric values set as in Table \ref{table:ba}, starting from a small founding population $p_{0}(a)$ (with initial total population $\approx 392$), the total population, juvenile, reproductive and senescent subpopulations converge to the nontrivial equilibrium over approximately 400 years with IGC$\approx 1.5$. Figure \ref{sec:Fig2}B illustrates the evolution of the population density $p(a,t)$ in approximately 400 years and Figure \ref{sec:Fig2}C demonstrates the change of total, juvenile, reproductive and senescent subpopulations over about 400 years. The total population exceeds $1878$ with mean age $\approx 23.5$ years (see Figure \ref{sec:Fig2}B). We repeat the simulation with all parameters set as in the baseline model except that we increase the juvenile mortality due to the senescent population burden and the fertility rate. We numerically find IGC$\approx 2.56$. From Figure \ref{sec:Fig3}B and \ref{sec:Fig3}C, we observe that total population, juvenile, reproductive and senescent subpopulations all exhibit oscillatory behavior as the population converges to the equilibrium in more than 800 years. This indicates that age structure of humans is robust and could recover from oscillatory behavior as the population stabilizes at the nontrivial equilibrium.
As Figure \ref{sec:Fig5}A and \ref{sec:Fig5}B indicate that if the initial population consists of a large fraction of juvenile and senescent populations and very few reproductive individuals, and all baseline parameter values are held fixed, the total population, juvenile, reproductive and senescent subpopulations all exhibit oscillatory behavior as the population converges to the equilibrium in approximately 400 years. The numerical simulations here and in \cite{B} support the hypothesis that human age structure from early hunter-gatherer society to present (intrinsically shaped by age and population density dependent fertility and mortality and also regulated by evolutionary benefits and costs) is robust and stable \cite{ B2, B3, B,B1,B4}.
\end{example}

\subsection{Further discussion on the uniqueness of the nontrivial equilibrium}

In this section, we investigate different combinations of fertility and mortality functions such that the system (\ref{sec:1}) has a unique nontrivial equilibrium. In the following theorem, we assume $\mu_{0}(a,z)$ or $\mu_{1}(a,z)$, for $(a,z)\in [0,a_{1}]\times [0,\infty)$ to be only age dependent, and let $\beta(a,z)$, for $(a,z)\in [a_{\min},a_{\max}]\times [0,\infty)$ to be only age dependent.

\begin{theorem}\label{sec:T4.2} Let H.1 hold. Let $\mu_{0}(a,z)=\mu_{0}(a)$ or $\mu_{1}(a,z)=\mu_{1}(a)$ for $(a,z)\in [0,a_{1}]\times [0,\infty)$, and $\beta(a,z)=\beta(a)$ for $(a,z)\in[a_{\min},a_{\max}]\times [0,\infty)$. Let $\eta_{i}'(x)$ $\times\frac{\partial \mu_{i}(\hat{a}, \eta_{i}(x))}{\partial \eta_{i}(x)}\geq 0$, where $(\hat{a},x)\in [0,a_{1}]\times [0,\infty)$. If IGC $> 1$, the system (\ref{sec:1}) admits the unique positive equilibrium.

\end{theorem}

\begin{proof}
We only show the case that $\mu_{0}(a,z)=\mu_{0}(a)$. Define $\mathcal{K}:[0,\infty) \rightarrow [0,\infty)$ by:
\begin{align*}
\mathcal{K}(x):=\int_{a_{\min}}^{a_{\max}}\beta(a)e^{-\int_{0}^{a}(\mu_{1}(\hat{a}, \eta_{1}(x))+\mu_{0}(\hat{a})+\mu_{2}(\hat{a}))d\hat{a}}da, \ \text{ for }x\geq 0.
\end{align*}
We observe that $\mathcal{K}(0)=IGC>1$. Since $\eta_{1}'(x)\frac{\partial \mu_{1}(\hat{a}, \eta_{1}(x))}{\partial \eta_{1}(x)}$ $\geq 0$, where $(\hat{a},x)\in [0,a_{1}]\times [0,\infty)$, we obtain, for $Q_{1}\geq 0$,
\begin{align*}
\mathcal{K}'(Q_{1})&=-\int_{a_{\min}}^{a_{\max}}\beta(a)e^{-\int_{0}^{a}(\mu_{1}(\hat{a}, \eta_{1}(Q_{1}))+\mu_{0}(\hat{a})+\mu_{2}(\hat{a}))d\hat{a}}\\
&\times[\eta_{1}'(Q_{1})\int_{0}^{a}\frac{\partial \mu_{1}(\hat{a}, \eta_{1}(Q_{1}))}{\partial \eta_{1}(Q_{1})}d\hat{a}]da\leq 0.
\end{align*}
Therefore, $\mathcal{K}(z)$ is a non-increasing function for $z\geq 0$. As $z\rightarrow \infty$, $\mathcal{K}(z)\rightarrow 0$. By the intermediate value theorem, the continuous function $\mathcal{K}(z)$ for $z\geq 0$ has a unique, positive equilibrium.
\end{proof}

Following example gives an illustration of Theorem \ref{sec:T4.2}:

\begin{example}\label{sec:E4.9}
Let H.1 hold. Let $\mu_{0}(a,z)=\eta(a)z$ and $\mu_{1}(a,z)=\mu(a)z$, for $z\geq 0$, where $\eta,\ \mu\in C_{+}[0,a_{1}]$, and $\mu(a)=0$ for $a>a_{\min}$. Consider (\ref{sec:E4.1}) as follows:
\begin{align}\label{sec:E4.1}
&p_{t}(a,t)+p_{a}(a,t)=\begin{cases}
-[\eta(a)T(t)+\mu_{1}(a)+\mu_{2}(a)]p(a,t), & \ \text{ if } \mu_{1}(a,z)=\mu_{1}(a);\\
-[\mu(a)S(t)+ \mu_{0}(a)+\mu_{2}(a)]p(a,t), & \ \text{ if } \mu_{0}(a,z)=\mu_{0}(a).
\end{cases}\\
\nonumber &\text{ for }z\geq 0, \ 0<a<a_{1},\ t>0,\\
\nonumber &p(0,t)=\int_{a_{\min}}^{a_{\max}}\beta(a)p(a,t)da,\ t>0,\\
\nonumber &p(a,0)=p_{0}(a),\ 0<a<a_{1}.
\end{align}
It is readily seen that conditions of Theorem \ref{sec:T4.2} are satisfied, therefore, the existence and uniqueness of a positive equilibrium of (\ref{sec:E4.1}) follow if IGC $>1$.

\end{example}

In the following theorem, we derive the existence and uniqueness of the positive equilibrium from the system (\ref{sec:1}) when it is subject to a nonlinear boundary condition and linear mortalities $\mu_{i}(a,z)=\mu_{i}(a),\ i=0,1$, for $(a,z)\in [0,a_{1}]\times [0,\infty)$. 

\begin{theorem}\label{sec:T4.12} Let H.1 hold. Let $\mu_{i}(a,z)=\mu_{i}(a)$, for $i=0,1$, where $(a,z)\in [0,a_{1}]\times [0,\infty)$, and $\beta(a,z)=\tilde{\beta}(a)f(z)$ for $(a,z)\in[a_{\min},a_{\max}]\times [0,\infty)$, where $\tilde{\beta}\in C_{+}[a_{\min},a_{\max}]$
 and $f\in C^{1}[0,\infty)$, $f\geq 0$, $f'\leq 0$ on $[0,\infty)$. If IGC $> 1$, the system (\ref{sec:1}) admits the unique positive equilibrium.

\end{theorem}

\begin{proof}
Define $\mathcal{K}_{\tilde{\beta}}:[0,\infty) \rightarrow [0,\infty)$ by:
\begin{align*}
\mathcal{K}_{\tilde{\beta}}(T):=\int_{a_{\min}}^{a_{\max}}\tilde{\beta}(a)f(T)e^{-\int_{0}^{a}(\mu_{0}(\hat{a})+\mu_{1}(\hat{a})+\mu_{2}(\hat{a}))d\hat{a}}da, \ \text{ for }T\geq 0.
\end{align*}
We observe that $\mathcal{K}_{\tilde{\beta}}(0)=IGC>1$, and since $f'(z)\leq 0$, for $z\in[0,\infty)$, we have,
\begin{align*}
\mathcal{K}'_{\tilde{\beta}}(T)=\int_{a_{\min}}^{a_{\max}}\tilde{\beta}(a)f'(T)e^{-\int_{0}^{a}(\mu_{0}(\hat{a})+\mu_{1}(\hat{a})+\mu_{2}(\hat{a}))d\hat{a}}da\leq 0,\ \text{ for }T\geq 0.
\end{align*}
Therefore, $\mathcal{K}_{\tilde{\beta}}(z)$ is a continuous non-increasing function for $z\geq 0$. As $z\rightarrow \infty$, $\mathcal{K}_{\tilde{\beta}}(z)\rightarrow 0$. By the intermediate value theorem, 
the existence and uniqueness of a positive equilibrium follow.
\end{proof}

Following example illustrates Theorem \ref{sec:T4.12}:

\begin{example}\label{sec:E4}
Let H.1 hold. Let $\mu_{i}(a,z)=\mu_{i}(a)$, for $i=0,1$, where $(a,z)\in[0,a_{1}]\times[0,\infty)$. 
 Let $\alpha>0$ and $\beta(a,z)=\frac{1}{1+\alpha z}$, for $(a,z)\in[a_{\min},a_{\max}]\times[0,\infty)$. Consider (\ref{sec:E4.2}) given as follows:
\begin{align}\label{sec:E4.2}
&p_{t}(a,t)+p_{a}(a,t)=-(\mu_{0}(a)+\mu_{1}(a)+\mu_{2}(a))p(a,t), \ 0<a<a_{1},\ t>0,\\
\nonumber &p(0,t)=\int_{0}^{a_{1}}\frac{1}{1+\alpha T(t)}p(a,t)da,\ t>0,\\
\nonumber &p(a,0)=p_{0}(a),\ 0<a<a_{1}.
\end{align}
By Theorem \ref{sec:T4.12}, if IGC $>1$, the existence and uniqueness of a positive equilibrium of the system (\ref{sec:E4.2}) follow.

\end{example}

\section{The linear problem}

In this section we use the method of linearization to address the local stability or instability problem of the equilibrium solutions of the model (\ref{sec:1}). First we formulate our linear problem by finding the Frechet derivatives of nonlinear operators $\mathcal{F}, \ \mathcal{G}$ given by (\ref{sec:2})-(\ref{sec:3}) at an equilibrium solution $\hat{\phi}\in X$ of the system (\ref{sec:1}) and then we study some basic properties of a linear operator. Next we discuss the state space decomposition by invariant subspaces. We close this section by deriving the characteristic equation for the linear problem. We first state the following definition:

\begin{definition}[\cite{W},sec 4.2, pp.145]
Let $U(t),t\geq 0$, be a strongly continuous nonlinear semigroup in the closed subset $C$ of the Banach space $X$ and let $\hat{x}\in C$. \textit{The trajectory} $\gamma(\hat{x})$ is $\gamma(\hat{x}):=\left\{U(t)\hat{x}:t\geq 0\right\}$. If $\gamma(\hat{x})=\hat{x}$, then $\hat{x}$ is an equilibrium solution for $U(t),t\geq 0$. The trajectory $\gamma(\hat{x})$ is stable if and only if for each $\epsilon >0$, there exists $\delta >0$ such that if $x\in C$ and $\left\|x-\hat{x}\right\|<\delta$, then $\left\|U(t)x-U(t)\hat{x}\right\|<\epsilon$ for all $t\geq 0$. The trajectory $\gamma(\hat{x})$ is unstable if and only if it is not stable. The trajectory is \textit{asymptotically stable} if and only if it is stable and there exists some $\delta >0$ such that if $x\in C$ and  $\left\|x-\hat{x}\right\|<\delta$, then $\lim_{t\rightarrow \infty}\left\|U(t)x-U(t)\hat{x}\right\|=0$. The trajectory is \textit{exponentially asymptotically stable} if and only if it is asymptotically stable and there exists $\delta>0$, $\omega>0$ and $K>0$ such that if $x\in C$ and  $\left\|x-\hat{x}\right\|<\delta$, then $\left\|U(t)x-U(t)\hat{x}\right\|\leq Ke^{-\omega t}\left\|x-\hat{x}\right\|$. If $\delta$ can be chosen arbitrarily large in each of these last two definitions, then the corresponding property is \textit{ said to be global}.
\end{definition}

\subsection{The linear problem}

We assume as in H.1, that $\mathcal{F}, \ \mathcal{G}$ given by (\ref{sec:2})-(\ref{sec:3}) are continuously Frechet differentiable at an equilibrium solution $\hat{\phi}\in X$ of the system (\ref{sec:1}) and the Frechet derivatives are given by $\mathcal{F}'(\hat{\phi}), \ \mathcal{G}'(\hat{\phi})$ as in (\ref{sec:2'})-(\ref{sec:5'}).  The associated linearization of (\ref{sec:1}) is as follows:
\begin{align}
&u_{t}(a,t)+u_{a}(a,t)=-[\mu_{0}(a, \eta_{0}(Q_{0}\hat{\phi}))+\mu_{1}(a, \eta_{1}(Q_{1}\hat{\phi}))+\mu_{2}(a)]u(a,t)\label{sec:6.2}\\
\nonumber& - \frac{\partial \mu_{0}(a, z)}{\partial z}|_{z = \eta_{0}(Q_{0}\hat{\phi})}\eta_{0}'(Q_{0}\hat{\phi})(Q_{L,0}(t))\hat{\phi}(a)-\frac{\partial \mu_{1}(a, z)}{\partial z}|_{z = \eta_{1}(Q_{1}\hat{\phi})}\eta_{1}'(Q_{1}\hat{\phi})\\
\nonumber&\times(Q_{L,1}(t))\hat{\phi}(a),\  0<a<a_{1},\ t>0,\\
\nonumber &u(0,t)=\int^{a_{\max}}_{a_{\min}}\beta(a; \eta_{2}(Q_{0}\hat{\phi}))u(a,t)da\\
\nonumber&+\eta_{2}'(Q_{0}\hat{\phi})(Q_{L,0}(t))\int^{a_{\max}}_{a_{\min}}\frac{\partial \beta(a, z)}{\partial z}|_{z =
\eta_{2}(Q_{0}\hat{\phi})}\hat{\phi}(a)da, \ t>0,\\
\nonumber &u(a,0)=p_{0}(a),\ 0<a<a_{1}.
\end{align}
where $Q_{L,i}(t)=\int_{0}^{a_{1}}\omega_{i}(a)u(a,t)da,\ i=0,1.$

We apply Proposition 3.2 and Proposition 3.7 (in \cite{W}, section 3.1, pp.76) to obtain the semigroup property for solutions of (\ref{sec:6.2}):
\begin{theorem}\label{sec:T6.1}
Let H.1-H.2 hold. Let $\mathcal{F}'(\hat{\phi})$ and $\mathcal{G}'(\hat{\phi})$ be bounded linear operators from $X$ to $\mathbb{R}$ and from $X$ into $X$ as in (\ref{sec:2'})-(\ref{sec:5'}). If $p_{0}\in X$, then the generalized solution $u(a, t)$
of the system (\ref{sec:6.2}) is defined on $[0,\infty)$. Further, the family of mappings $T_{L}(t), \ t\geq 0$ in $X$, defined by $(T_{L}(t)p_{0})(a)=u(a, t)$ is a strongly continuous semigroup of bounded linear operators in $X$ satisfying,
\begin{align*}
|T_{L}(t)|\leq e^{\omega t}, \ \text{for}\ t\geq 0\ \ \ \ \text{ where }\omega=|\mathcal{F}'(\hat{\phi})|+|\mathcal{G}'(\hat{\phi})|.
\end{align*}
The infinitesimal generator of $T_{L}(t), \ t\geq 0$, is
\begin{align}\label{sec:6.1}
&\hat{B}\phi=-\phi'+\mathcal{G}'(\hat{\phi})\phi,\ \ \text{ for } \phi\in D(\hat{B}).
\end{align}
where,
\begin{align*}
D(\hat{B})=\left\{\phi\in X: \phi \text{ is absolutely continuous on }[0,a_{1}], \phi'\in L^{1}, \phi(0)=\mathcal{F}'(\hat{\phi})\phi\right\}.
\end{align*}
Further, for all $t\geq 0$, $T_{L}(t)(D(\hat{B}))\subset D(\hat{B})$ and $(d/dt )T_{L}(t)\phi=\hat{B}T_{L}(t)\phi=T_{L}(t)\hat{B}\phi$ for all $\phi\in D(\hat{B})$.
\end{theorem}

\subsection{Basic properties of the linear semigroup}

Let $T_{L}(t), \ t \geq 0$ be  the strongly continuous linear semigroup, in $X$ as in Theorem \ref{sec:T6.1}. Let $\sigma(T_{L})$, $E\sigma (T_{L})$, and $P\sigma (T_{L})$ be the spectrum, the essential spectrum and the point spectrum of the linear operator $T_{L}$ (for details see \cite{W}). We begin this section with some definitions and results from \cite{Ce, W, E}.


\begin{definition}
The spectrum of a closed linear operator $T$ in the complex Banach space $Y$, denoted by $\sigma(T)$, is the complement of the resolvent set of $T$, denoted by $\rho(T)$, in the complex plane $\mathbb{C}$. $\rho(T)$ is the set of complex numbers $\lambda$ for which $(\lambda I- T)^{-1}$ exists and is an everywhere defined bounded linear operator in $Y$. \textit{The continuous spectrum} of $T$, denoted by $C\sigma(T)$, is the set of complex numbers $\lambda$ such that $(\lambda I -T)^{-1}$ exists, is densely defined in $Y$, but not bounded. \textit{The residual spectrum}, denoted by $R\sigma(T)$, is the set of complex numbers $\lambda$ such that $(\lambda I -T)^{-1}$ exists, but is not densely defined in $Y$. \textit{The point spectrum} of $T$, denoted by $P\sigma (T)$, is the set of complex numbers $\lambda$ such that $Tx=\lambda x$ for some nonzero $x\in Y$. If $\lambda\in P\sigma(T)$, then $\lambda$ is called an eigenvalue of $T$ and a nonzero vector $x\in X$ such that $Tx=\lambda x$ is called an eigenvector of $T$ corresponding to the eigenvalue $\lambda$. If $\lambda$ is an eigenvalue of $T$, then  the Null space $N(\lambda I- T)$ is called the geometric eigenspace of $T$ with respect to $\lambda$, and its dimension is called the geometric multiplicity of $\lambda$. \textit{The essential spectrum} of $T$, denoted by $E\sigma (T)$, is the set of $\lambda\in \sigma(T)$ such that at least one of the following holds:
\begin{itemize}
\item[(i)] $R(\lambda I-T)$ is not closed;
\item[(ii)] $\lambda$ is a limit point of $\sigma(T)$;
\item[(iii)] the generalized eigenspace of $T$ with respect to $\lambda$, denoted by $N_{\lambda}(T)$, is infinite dimensional.
\end{itemize}
where $N_{\lambda}(T)$ is the smallest closed subspace of $Y$ containing $\cup_{k=1}^{\infty} N((\lambda I- T)^{k})$ and where $N((\lambda I- T)^{k})$ denotes the null space of $(\lambda I- T)^{k}$, $k=1,2,\cdots$. The \textit{point spectrum} of $T$, denoted by $P\sigma(T)$, is the set of complex numbers $\lambda$ such that $Tx=\lambda x$ for some nonzero $x\in X$. If $\lambda\in P\sigma(T)$, then $\lambda$ is called an \textit{eigenvalue} of $T$ and a nonzero vector $x\in X$ such that $Tx=\lambda x$ is called an \textit{eigenvector} of $T$ corresponding to the eigenvalue $\lambda$.
\end{definition}

\begin{definition}
The growth bound, $\omega_{0}(\hat{B})\in[-\infty, \infty)$ of $\hat{B}: D(\hat{B})\subset X\rightarrow X$, where $\hat{B}$ is the infinitesimal generator of the strongly continuous linear semigroup $\left\{T_{L}(t),t\geq 0\right\}$ in $X$ as in Theorem \ref{sec:T6.1}, is defined as follows:
\begin{align*}
\omega_{0}(\hat{B}):=\lim_{t\rightarrow \infty}\frac{\ln(\left\|T_{L}(t)\right\|_{\mathcal{L}(X)})}{t}.
\end{align*}
\textit{The essential growth bound}, $\omega_{0,ess}(\hat{B})\in[-\infty, \infty)$ of $\hat{B}$ is defined by:
\begin{align*}
\omega_{0,ess}(\hat{B}):=\lim_{t\rightarrow \infty}\frac{\ln(\left\|T_{L}(t)\right\|_{ess})}{t},
\end{align*}
where, $\left\|T_{L}(t)\right\|_{ess}$ defined by $\left\|T_{L}(t)\right\|_{ess}$ $:=\kappa(T_{L}(t)\tilde{\mathcal{B}}_{X}(0,1))$ is the essential norm of $T_{L}(t)$, and $\tilde{\mathcal{B}}_{X}(0,1)=\left\{x\in X: \left\|x\right\|_{X}\leq 1\right\}$, and for each bounded set $\tilde{\mathcal{B}} \subset X$, we define the Kuratovsky measure of non-compactness as:
\begin{align*}
\kappa(\tilde{\mathcal{B}}):=\inf\left\{\epsilon>0: \tilde{\mathcal{B}}\text{ can be covered by a finite number of balls of radius}\leq \epsilon\right\}.
\end{align*}
\end{definition}

We state the following theorem 
which links the local stability of an equilibrium $\hat{\phi}$ of (\ref{sec:1}) to the spectral properties of $\hat{B}$:

\begin{theorem}\label{sec:T6.2}  Let H.1-H.2 hold. Let $U(t), \ t\geq 0$, be the strongly continuous nonlinear semigroup in $X_{+}$ as in Theorem \ref{sec:T3.3} with infinitesimal generator $\mathcal{A}$. Let $T_{L}(t), \ t\geq 0$, be the strongly continuous linear semigroup in $X$ as in Theorem \ref{sec:T6.1} with infinitesimal generator $\hat{B}$, defined by (\ref{sec:6.1}). Let $\hat{\phi}\in X_{+}$ be an equilibrium solution of the system (\ref{sec:1}). Let $\omega_{0, ess}(\hat{B})<0$. The following hold:
\begin{itemize}
\item[(a)] If $\sup_{\lambda\in \sigma(\hat{B})-E \sigma(\hat{B})}Re(\lambda)<0$, then $\hat{\phi}$ is a locally exponentially asymptotically stable equilibrium of the system (\ref{sec:1}).
\item[(b)] If there exists $\lambda_{1}\in \sigma(\hat{B})$ such that $Re(\lambda_{1})>0$ and $\sup_{\lambda\in \sigma(\hat{B})-E \sigma(\hat{B})\ \lambda \neq \lambda_{1}}$ $Re(\lambda)<Re(\lambda_{1})$, then $\hat{\phi}$ is an unstable equilibrium of the system (\ref{sec:1}).
\end{itemize}
\end{theorem}

In the following theorem, the existence of the projector was first proved in \cite{W}, and in \cite{E} it is shown the spectrum consists of finite points.

\begin{theorem}
Let $\hat{B}: D(\hat{B})\subset X\rightarrow X$ be the infinitesimal generator of the strongly continuous linear semigroup $\left\{T_{L}(t)\right\}_{t\geq 0}$ in $X$ as in Theorem \ref{sec:T6.1}. Then
\begin{align*}
\omega_{0}(\hat{B})=\max(\omega_{0,ess}(\hat{B}), \max_{\lambda\in \sigma(\hat{B})- E\sigma(\hat{B})} Re(\lambda)).
\end{align*}
Assume in addition that $\omega_{0,ess}(\hat{B})<\omega_{0}(\hat{B})$. Then, for each $\gamma\in (\omega_{0,ess}(\hat{B}), \omega_{0}(\hat{B})]$, $\left\{\lambda\in \sigma(\hat{B}): Re(\lambda)\geq \gamma\right\}\subset P\sigma(\hat{B})$ is nonempty, finite and contains only poles of the resolvent of $\hat{B}$. Moreover, there exists a finite rank bounded linear operator of projection $\mathcal{P}: X\rightarrow X$ satisfying the following properties:
\begin{itemize}
\item[(i)] $\mathcal{P}(\lambda-\hat{B})^{-1}=(\lambda-\hat{B})^{-1}\mathcal{P}$, $\forall\lambda\in \rho(\hat{B})$;
\item[(ii)] $\sigma(\hat{B}_{\mathcal{P}(X)})=\left\{\lambda\in\sigma(\hat{B}): \text{ Re } (\lambda)\geq \gamma\right\}$;
\item[(iii)] $\sigma(\hat{B}_{(I-\mathcal{P})(X)})=\sigma(\hat{B})-\sigma(\hat{B}_{\mathcal{P}(X)})$.
\end{itemize}
\end{theorem}

The following definition for a strongly continuous linear semigroup $T_{L}(t),$ $ \ t\geq 0$, to be \textit{irreducible} is from \cite{Ce}, section 7.1, pp.165:

\begin{definition}\label{sec:Def1}
A positive  strongly continuous semigroup $T_{L}(t), \ t\geq 0$, in the Banach space $X$ is
\textit{irreducible} if and only if for every $0<x\in X$ and $0<\phi\in X^{*}$, there exists $t\geq 0$ such that $\left\langle T_{L}(t)x, \phi\right\rangle >0$.
\end{definition}

The following Proposition 
gives equivalent conditions for a strongly continuous linear semigroup $T_{L}(t),$ $ \ t\geq 0$, to be irreducible:

\begin{proposition}\label{sec:P6.1}
For a positive  strongly continuous semigroup $T_{L}(t), \ t\geq 0$,
with generator $\hat{B}$ in a Banach lattice $X$, the following statements are  equivalent.
\begin{itemize}
\item[(i)] $T_{L}(t), \ t\geq 0$, is irreducible;
\item[(ii)] The resolvent $R(\lambda, \hat{B})$ satisfies $R(\lambda, \hat{B})f$ is strictly positive for some $\lambda>s(\hat{B})$ and all $0<f\in X$, where $s(\hat{B})=\sup\left\{ Re(\lambda): \lambda\in \sigma(\hat{B})\right\}$ is the spectral bound of $\hat{B}$.
\end{itemize}
\end{proposition}
The compactness and boundedness results for the strongly continuous linear semigroup $T_{L}(t), \ t \geq 0$ in $X$ as in Theorem \ref{sec:T6.1} follow from: 
\begin{theorem}\label{sec:PT6}
Let H.1-H.2 hold. Let $\mathcal{F}'(\hat{\phi})$ and $\mathcal{G}'(\hat{\phi})$ be bounded linear operators given by (\ref{sec:2'})-(\ref{sec:3'}).
Let $T_{L}(t), \ t\geq 0$ be the strongly continuous linear semigroup in $X$ as in Theorem \ref{sec:T6.1}. If $\phi\in X$ and $\left\{T_{L}(t)\phi: \ t\geq 0\right\}$ is bounded in $L^{1}$, then $\left\{T_{L}(t)\phi: \ t\geq 0\right\}$ has compact closure in $L^{1}$.
\end{theorem}

We will show the positivity and irreducibility of the strongly continuous linear semigroup $T_{L}(t), \ t \geq 0$ in $X$ as in Theorem \ref{sec:T6.1} under the following assumption, for the linear operators $\mathcal{F}'(\hat{\phi})$ and $\mathcal{G}'(\hat{\phi})$ as in (\ref{sec:2'})-(\ref{sec:3'}), where $\hat{\phi}$ is a positive equilibrium solution (\ref{sec:4.12}) of (\ref{sec:1}):
\begin{itemize}
\item[H.4.]
\begin{align}
&\eta_{2}'(Q_{0}\hat{\phi})\frac{\partial \beta(a, z)}{\partial z}|_{z = \eta_{2}(Q_{0}\hat{\phi})} \geq 0,\  \text{ for } a\in[a_{\min},a_{\max}].\label{sec:6.2.1}\\
&\eta_{i}'(Q_{i}\hat{\phi})\frac{\partial \mu_{i}(a, z)}{\partial z}|_{z = \eta_{i}(Q_{i}\hat{\phi})}\leq 0,\ i=0,1,\  \text{ for }a\in[0,a_{1}].\label{sec:6.2.2}
\end{align}
\end{itemize}
Conditions (\ref{sec:6.2.1})-(\ref{sec:6.2.2}) indicate that a biological population may exhibit a certain type of behavior at a positive equilibrium of the system (\ref{sec:1}) in which will the population growth rate becomes negative, while the fertility rate growth becomes positive.

\begin{theorem}\label{sec:T6.6}
Let H.1-H.2 and H.4 hold. Then, the strongly continuous linear semigroup $T_{L}(t),t \geq 0$, in $X$ as in Theorem \ref{sec:T6.1} satisfies $T_{L}(t)(X_{+})\subseteq X_{+}$ and is irreducible.
\end{theorem}

\begin{proof}

The positivity of the operator $\mathcal{C}_{1}$ as in (\ref{sec:4'}) follows directly from condition (\ref{sec:6.2.2}). Therefore, we can
restrict ourselves to the operator $\hat{B}-\mathcal{C}_{1}$. The associated differential equation subject to the corresponding boundary condition is given as follows:
\begin{align*}
&u_{t}(a,t)+u_{a}(a,t)=-[\mu_{0}(a, \eta_{0}(Q_{0}\hat{\phi}))+\mu_{1}(a, \eta_{1}(Q_{1}\hat{\phi}))+\mu_{2}(a)]u(a,t),\\
\nonumber & 0<a<a_{1},\ t>0,\\
\nonumber &u(0,t)=\int^{a_{\max}}_{a_{\min}}\beta(a; \eta_{2}(Q_{0}\hat{\phi}))u(a,t)da+\eta_{2}'(Q_{0}\hat{\phi})\int_{0}^{a_{1}}\omega_{0}(a)u(a,t)da\\
\nonumber&\times\int^{a_{\max}}_{a_{\min}}\frac{\partial \beta(a, z)}{\partial z}|_{z = \eta_{2}(Q_{0}\hat{\phi})}\hat{\phi}(a)da, \ t>0,\\
\nonumber &u(a,0)=p_{0}(a),\ 0<a<a_{1}.
\end{align*}
Let $u$ be a solution of above equation and satisfy the given boundary condition. Then, the function $w$ given by
\begin{align*}
w(a,t)=u(a,t)\exp\left\{\int_{0}^{a}[\mu_{0}(\hat{a}, \eta_{0}(Q_{0}\hat{\phi}))+\mu_{1}(\hat{a}, \eta_{1}(Q_{1}\hat{\phi}))+\mu_{2}(\hat{a})]d\hat{a}\right\}.
\end{align*}
satisfies,
\begin{align*}
&w_{t}(a,t)+w_{a}(a,t)=0,\ 0<a<a_{1},\ t>0,\\
&w(0,t)=\psi(w(a,t)), \ t>0,\\
&w(a,0)=p_{0}(a)\exp\left\{\int_{0}^{a}[\mu_{0}(\hat{a}, \eta_{0}(\bar{Q}_{0}))+\mu_{1}(\hat{a}, \eta_{1}(\bar{Q}_{0}))+\mu_{2}(\hat{a})]d\hat{a}\right\},\ 0<a<a_{1}.
\end{align*}
where, $\psi(w(a,t))=$ $\mathcal{F}'(\hat{\phi})(w(a,t)e^{-\int_{0}^{a}[\mu_{0}(\hat{a}, \eta_{0}(Q_{0}\hat{\phi}))+\mu_{1}(\hat{a}, \eta_{1}(Q_{1}\hat{\phi}))+\mu_{2}(\hat{a})]d\hat{a}})$.
This system is associated with the linear semigroup generated by $\mathcal{B}\phi=-\phi'$ for $\phi\in D(\mathcal{B})$, with  $D(\mathcal{B})$ being defined as:
\begin{align*}
D(\mathcal{B})=\left\{\phi\in X: \phi \text{ is absolutely continuous on }[0,a_{1}], \phi'\in L^{1}, \phi(0)=\psi(\phi)\right\}.
\end{align*} 
It then suffices to show that the semigroup generated by $\mathcal{B}$ is positive. We observe that the resolvent equation $\lambda w-\mathcal{B} w= f$ has the  solution $w(a)= e^{-\lambda a}\psi(w)+\int_{0}^{a}e^{-\lambda (a-b)}f(b)db$ for $\lambda \geq 0$ sufficiently large and $f\in X_{+}$. Applying $\psi$ on both sides, we get $\psi(w)=(1-\psi(e^{-\lambda a}))^{-1} \psi(\int_{0}^{a}e^{-\lambda (a-b)}f(b)db)$. From the definition of $\psi$ and condition (\ref{sec:6.2.1}) we obtain that the solution $w$, is positive if $f$ is positive a.e. and $\lambda$ is sufficiently large. Therefore, the resolvent operator of $\mathcal{B}$ is positive for sufficiently large $\lambda$. Then the conclusion follows from Proposition \ref{sec:P6.1} (iii).

\end{proof}


Theorem \ref{sec:T6.6} has the following consequence:

\begin{theorem}\cite{E}
Suppose that conditions H.1-H.2 and H.4 hold true. Then the spectral bound $s(\hat{B})=\sup\left\{\text{Re }\lambda: \lambda\in \sigma(\hat{B})\right\}$ belongs to the spectrum $\sigma(\hat{B})$. Specifically, the spectral bound $s(\hat{B})$ is a dominant eigenvalue, and any other point $\lambda$ in the spectrum has real part less than $s(\hat{B})$.
\end{theorem}

In the following theorem, we show $\omega_{0, ess}(\hat{B})<0$:

\begin{theorem}\label{sec:ES1} Let H.1-H.2 hold. Let $T_{L}(t), \ t\geq 0$, be the strongly continuous linear semigroup in $X$ as in Theorem \ref{sec:T6.1} with infinitesimal generator $\hat{B}$, given by (\ref{sec:6.1}), then,
\begin{align*}
\omega_{0, ess}(\hat{B})=-\infty.
\end{align*}

\end{theorem}

\begin{proof}
Let $T_{L}(t)=\tilde{W}_{1}(t)+\tilde{W}_{2}(t)$, where the mappings $\tilde{W}_{1}(t), \tilde{W}_{2}(t)\in X$ for $t\geq 0$, $\phi\in X$ are defined as follows:
\begin{align}\label{sec:6.1'}
&(\tilde{W}_{1}(t)\phi)(a)=\begin{cases}
0 & \text{ a.e. } \ \ a \in (0,t)\cap[0,a_{1}];\\
(T_{L}(t)\phi)(a) & \text{ a.e. } \ \ a\in (t, a_{1}];
\end{cases}\\
&(\tilde{W}_{2}(t)\phi)(a)=\begin{cases}\label{sec:6.2'}
(T_{L}(t)\phi)(a) & \text{ a.e. } \ \ a \in (0,t)\cap[0,a_{1}];\\
0 & \text{ a.e. } \ \ a\in (t, a_{1}].
\end{cases}
\end{align}
It is readily seen that $\tilde{W}_{1}(t)=0$ for $t>a_{1}$ while $\tilde{W}_{2}(t)$ (see \cite{W}, section 3.4, pp.112) is ultimately compact by using the measure of noncompactness due to Kuratowski by Proposition 3.17 (see \cite{W}, section 3.5, pp.113). Therefore, from Proposition 4.9 (in \cite{W}, section 4.3, pp.166), we obtain $\alpha[T_{L}(t)]\leq \alpha[\tilde{W}_{1}(t)]+\alpha[\tilde{W}_{2}(t)]=0$ for $t>a_{1}$, where $\alpha$ is the measure of noncompactness of $T$ defined in \cite{W}, section 4.3, pp.165. The claim then follows directly.
\end{proof}


\subsection{State space decomposition by invariant subspaces}

\begin{definition}\label{sec:defproj}
Let $Y$ be a Banach space. An everywhere defined bounded linear operator $P$ in $Y$ is called a projection provided that $P^{2}=P$. Let $M_{1}$, $M_{2}$ be linear subspaces of $Y$. Then, $Y$ is the direct sum of $M_{1}$, and $M_{2}$, denoted by $Y=M_{1}\oplus M_{2}$, provided that $M_{1}\cap M_{2}=0$ and for each $y\in Y$ there exists the (necessarily unique) representation $y=y_{1}+y_{2}$, where $y_{1}\in M_{1}$, $y_{2}\in M_{2}$.
\end{definition}

\begin{proposition}
Let $X$ be a Banach space. If $P$ is a projection in $X$, then $I-P$ is a projection in $X$, and $X=M_{1}\oplus M_{2}$, where $M_{1}=P(X)$, $M_{2}=(I-P)X$, and $M_{1}$, $M_{2}$ are closed subspaces of $X$. Conversely, if $X=M_{1}\oplus M_{2}$ is the direct sum of two closed subspaces $M_{1}$, $M_{2}$, then $P_{1}$, $P_{2}$ are projections in $X$, where $P_{i}x=x_{i}$, $x=x_{1}+x_{2}$, $x_{i}\in M_{i}$, $i=1,2$.
\end{proposition}

\begin{definition}
Let the Banach space $X$ have the direct sum representation $X=M_{1}\oplus M_{2}$, where $M_{1}$, $M_{2}$ are closed subspaces of $X$. Let $P_{1}$, $P_{2}$ be the projections induced by $M_{1}$, $M_{2}$, that is, $P_{i}x=x_{i}$, where $x=x_{1}+x_{2}$, $x_{i}\in M_{i}$, $i=1,2$. A closed linear operator $T$ in $X$ is said to be completely reduced by $M_{1}$ and $M_{2}$ provided that $T(M_{1}\cap D(T))\subset M_{1}$, $T(M_{2}\cap D(T))\subset M_{2}$, $P_{1}(D(T))\subset D(T)$ and $P_{2}(D(T))\subset D(T)$.
\end{definition}

\begin{proposition}
Let $T$ be a closed linear operator in the complex Banach space $X$ and let $\lambda_{0}$ be an isolated point of $\sigma(T)$. Then,
\begin{align}
(\lambda I-T)^{-1}=\sum_{k=-\infty}^{\infty}(\lambda-\lambda_{0})^{k}A_{k}.\label{sec:clpro1}
\end{align} 
where for each integer $k$
\begin{align}
A_{k}=(2\Pi i)^{-1}\int_{\Gamma}(\lambda-\lambda_{0})^{-k-1}(\lambda I-T)^{-1}d\lambda.\label{sec:clpro2}
\end{align} 
and $\Gamma$ is a positively oriented circle of sufficiently small radius such that no point of $\sigma(T)$ lies on or inside $\Gamma$. Further, $A_{-1}$ is a projection on $X$. If $\lambda_{0}$ is a pole of $(\lambda I-T)^{-1}$ of order $m$ (that is, $A_{-m}\neq 0$ and $A_{k}=0$ for all $k<-m$), then $\lambda_{0}$ is an eigenvalue of $T$ with index $m$, $R(A_{-1})=N((\lambda_{0} I-T)^{m})$, $R(I-A_{-1})=R((\lambda_{0} I-T)^{k})$ for all $k\geq m$, $X=N((\lambda_{0} I-T)^{m})\oplus R((\lambda_{0} I-T)^{m})$, and $T$ is completely reduced by the two linear subspaces occurring in this direct sum. Also, $R(A_{-1})$ is closed, $R(I-A_{-1})$ is closed, and $T$ restricted to $R(A_{-1})$ is bounded with spectrum $\left\{\lambda_{0}\right\}$. Finally, if $R(A_{-1})$ is finite dimensional, then $\lambda_{0}$ is a pole of $(\lambda I-T)^{-1}$.
\end{proposition}

We need the following result from \cite{W}, sec 4.3, pp. 166. 
\begin{proposition}
Let $T$ be a closed linear operator in the complex Banach space $X$ and let $\lambda_{0}\in \sigma(T)-E\sigma(T)$. Then, $N_{\lambda_{0}}(T)=R(A_{-1})$, where $A_{-1}:=(2\Pi i)^{-1}\int_{\Gamma}(\lambda-\lambda_{0})^{-k-1}(\lambda I-T)^{-1}d\lambda$, $\lambda_{0}$ is a pole of $(\lambda I-T)^{-1}$, and $\lambda_{0}\in P\sigma(T)$.
\end{proposition}

\begin{proposition}
Let $T(t),t\geq 0$, be a strongly continuous semigroup of bounded linear operator in the Banach space $X$ and let $B$ be the infinitesimal generator of $T(t),t\geq 0$. Let $\Lambda=\left\{\lambda_{1},\cdots,\lambda_{k}\right\}$ be a finite set of points in $\sigma(B)$ such that $Re\lambda_{j}>\omega_{0,ess}(B)$ for $j=1,\cdots,k$. Let
\begin{align}
\omega_{0,\Lambda}:=\max\left\{\omega_{0,ess}(B),\sup_{\lambda\in \sigma(B)-E\sigma(B)-\Lambda} Re\lambda\right\}.\label{sec:clpro3}
\end{align}  
and let
\begin{align}
\omega_{0,\Lambda}<\omega<\min\left\{Re\lambda:\;j=1,\cdots,k\right\}.\label{sec:clpro4}
\end{align} 
The following hold:
\begin{itemize}
\item[(i)] Each $\lambda_{j}\in \sigma(B)-E\sigma(B)$ and is therefore isolated in $\sigma(B)$, and if $P_{j}:=(2\Pi i)^{-1}\int_{\Gamma_{j}}(\lambda I-B)^{-1}d\lambda$, $1\leq j \leq k$, where $\Gamma_{j}$ is a positively oriented closed curve in $\mathbb{C}$ enclosing $\lambda_{j}$, but no other point of $\sigma (B)$, and $M_{j}:=R(P_{j})$, then $P_{j}$ is a projection in $X$, $P_{j}P_{h}=0$ for $j\neq h$, and $B$ restricted to $M_{j}$, denoted by $B_{M_{j}}$ is bounded with spectrum consisting of the single point $\lambda_{j}$.
\item[(ii)] If $P:=\sum_{j=1}^{k}P_{j}$, $P_{0}:=1-P$, and $M_{0}:=R(P_{0})$, then $B$ restricted to $M_{0}$, denoted by $B_{M_{0}}$, has spectrum $\sigma(B)-\Lambda$, $P_{j}Bx=BP_{j}x$ for all $x\in D(B)$, $0\leq j \leq k$, $X=M\oplus M_{0}$, where $M=M_{1}\oplus\cdots\oplus M_{k}$, and $B$ is completely reduced by $M$ and $M_{0}$.
\item[(iii)] If $t\geq 0$, then $T(t)P_{j}x=P_{j}T(t)x$ for all $x\in X$, $0\leq j \leq k$, and $T(t)$ is completely reduced by $M$ and $M_{0}$.
\item[(iv)] If for some $j$, $\lambda_{j}$ is a pole of $(\lambda I-B)^{-1}$ of order $m$, then $M_{j}=N((\lambda_{j}I-B)^{m})$, $R(I-P_{j})=R((\lambda_{j}I-B)^{m})$, and $\lambda_{j}$ is an eigenvalue of $B$ with index $m$. 
\item[(v)] There exists a constant $K\geq 1$ such that $\left\|T(t)P_{0}x\right\|\leq K e^{\omega t}\left\|P_{0}x\right\|$ for all $x\in X$, $t\geq 0$.
\item[(vi)] The restriction of $B$ to $M$, denoted by $B_{M}$, is bounded with spectrum consisting of $\Lambda$, $PT(t)x=e^{t B_{M}}Px$ for $x\in X$, $t\geq 0$, where $e^{t B_{M}}Px$, $-\infty<t<\infty$, is the exponential of $tB_{M}$ in $M$, and there exists $K\geq 1$ such that $\left\|e^{t B_{M}}Px\right\|\leq Ke^{\omega t}\left\|Px\right\|$ for $x\in X$ and $t\leq 0$.
\end{itemize}
\end{proposition}

\subsection{The characteristic equation}

In this section, we will follow the linearization procedure to derive the characteristic equation as follows:

Let H.1-H.2 hold. By Theorem \ref{sec:T6.2}, the stability of an equilibrium solution $\hat{\phi}$ of the system (\ref{sec:1}) is determined by $\sigma(\hat{B})-E \sigma(\hat{B})$. Accordingly, let $\lambda\in \mathbb{C}$ and let $\hat{B}\phi=\lambda\phi$ for $\phi\in X$ and $\phi \neq 0$. From the definition of $\hat{B}$, we derive the characteristic equation for $\lambda$:
\begin{align}
&\phi'(a)+\lambda\phi(a)+\mu_{2}(a)\phi(a)+\frac{\partial \mu_{0}(a, z)}{\partial z}|_{z = \eta_{0}(Q_{0}\hat{\phi})}\eta_{0}'(Q_{0}\hat{\phi})(Q_{0}\phi)\hat{\phi}(a)\label{sec:6.3.1}\\
\nonumber &+\frac{\partial \mu_{1}(a, z)}{\partial z}|_{z = \eta_{1}(Q_{1}\hat{\phi})}\eta_{1}'(Q_{1}\hat{\phi})(Q_{1}\phi)\hat{\phi}(a)+\mu_{0}(a, \eta_{0}(Q_{0}\hat{\phi}))\phi(a)\\
\nonumber &+\mu_{1}(a, \eta_{1}(Q_{1}\hat{\phi}))\phi(a)=0, \ a\in[0,a_{1}].\\
&\phi(0)=\int^{a_{\max}}_{a_{\min}}\beta(a; \eta_{2}(Q_{0}\hat{\phi}))\phi(a)da\label{sec:6.3.2}\\
\nonumber&+\eta_{2}'(Q_{0}\hat{\phi})(Q_{0}\phi)\int^{a_{\max}}_{a_{\min}}\frac{\partial \beta(a, z)}{\partial z}|_{z = \eta_{2}(Q_{0}\hat{\phi})}\hat{\phi}(a)da.
\end{align}
From (\ref{sec:6.3.1}), we obtain the general solution $\phi$, which takes the form:
\begin{align}\label{sec:6.3.3}
&\phi(a)=\phi(0)e^{-\lambda a}\Pi(0,a; \eta_{0}(Q_{0}\hat{\phi}), \eta_{1}(Q_{1}\hat{\phi}))\\
\nonumber&-e^{-\lambda a}\Pi(0,a; \eta_{0}(Q_{0}\hat{\phi}), \eta_{1}(Q_{1}\hat{\phi}))\int_{0}^{a} e^{\lambda b}\Pi(b,0; \eta_{0}(Q_{0}\hat{\phi}), \eta_{1}(Q_{1}\hat{\phi}))\\
\nonumber&\times[\frac{\partial \mu_{0}(b, z)}{\partial z}|_{z = \eta_{0}(Q_{0}\hat{\phi})}\eta_{0}'(Q_{0}\hat{\phi})(Q_{0}\phi)+\frac{\partial \mu_{1}(b, z)}{\partial z}|_{z = \eta_{1}(Q_{1}\hat{\phi})}\eta_{1}'(Q_{1}\hat{\phi})(Q_{1}\phi)]\hat{\phi}(b)db.
\end{align}
We have $Q_{i}\phi, \ i=0,1$ satisfies:
\begin{align}\label{sec:6.3.4}
Q_{i}\phi&=\int_{0}^{a_{1}}\omega_{i}(a)\phi(a)da\\
\nonumber&=\phi(0)\int_{0}^{a_{1}}\omega_{i}(a)e^{-\lambda a}\Pi(0,a; \eta_{0}(Q_{0}\hat{\phi}), \eta_{1}(Q_{1}\hat{\phi}))da\\
\nonumber&-\int_{0}^{a_{1}}\omega_{i}(a)\int_{0}^{a} e^{-\lambda (a-b)}\Pi(b,a; \eta_{0}(Q_{0}\hat{\phi}), \eta_{1}(Q_{1}\hat{\phi}))\\
\nonumber&\times[\frac{\partial \mu_{0}(b, z)}{\partial z}|_{z = \eta_{0}(Q_{0}\hat{\phi})}\eta_{0}'(Q_{0}\hat{\phi})(Q_{0}\phi)+\frac{\partial \mu_{1}(b, z)}{\partial z}|_{z = \eta_{1}(Q_{1}\hat{\phi})}\eta_{1}'(Q_{1}\hat{\phi})(Q_{1}\phi)]\\
\nonumber&\times \frac{\hat{Q}_{i}\Pi(0,b;\eta_{0}(\hat{Q}_{0}),\eta_{1}(\hat{Q}_{1}))}{\int^{a_{1}}_{0}\omega_{i}(a)\Pi(0,\tau;\eta_{0}(\hat{Q}_{0}),\eta_{1}(\hat{Q}_{1}))d\tau} dbda\ \ \text{ for }i=0,1.
\end{align}
We use the following basic properties of $\Pi$:
\begin{align*}
&\Pi(b,0;\eta_{0}(\hat{Q}_{0}),\eta_{1}(\hat{Q}_{1}))\Pi(0,b;\eta_{0}(\hat{Q}_{0}),\eta_{1}(\hat{Q}_{1}))=1;\\ &\Pi(b,a;\eta_{0}(\hat{Q}_{0}),\eta_{1}(\hat{Q}_{1}))=\Pi(b,0;\eta_{0}(\hat{Q}_{0}),\eta_{1}(\hat{Q}_{1}))\Pi(0,a;\eta_{0}(\hat{Q}_{0}),\eta_{1}(\hat{Q}_{1})).
\end{align*}
From (\ref{sec:6.3.4}), we obtain the following system of equations for $Q_{i}\phi$, $i=0,1$,
\begin{align}\label{sec:6.3.5}
&(1+I_{\omega_{0}, \mu_{0}}(\lambda))Q_{0}\phi+I_{\omega_{0}, \mu_{1}}(\lambda)Q_{1}\phi=C_{\omega_{0}}(\lambda)\phi(0);\\
\nonumber&I_{\omega_{1},\mu_{0} }(\lambda)Q_{0}\phi+(1+I_{\omega_{1}, \mu_{1}}(\lambda))Q_{1}\phi=C_{\omega_{1}}(\lambda)\phi(0).
\end{align}
where
\begin{align*}
\pi(a;\lambda)&=e^{-\lambda a}\Pi(0,a; \eta_{0}(Q_{0}\hat{\phi}), \eta_{1}(Q_{1}\hat{\phi}));\\
C_{\omega_{i}}(\lambda)&=\int_{0}^{a_{1}}\omega_{i}(a)\pi(a;\lambda)da,\ i=0,1.
\end{align*}
\begin{align*}
&I_{\omega_{i}, \mu_{j}}(\lambda)=\int_{0}^{a_{1}}\omega_{i}(a)\int_{0}^{a} e^{-\lambda (a-b)}\Pi(b,a; \eta_{0}(Q_{0}\hat{\phi}), \eta_{1}(Q_{1}\hat{\phi}))\\
&\times\frac{\partial \mu_{j}(b, z)}{\partial z}|_{z = \eta_{j}(Q_{j}\hat{\phi})}\eta_{j}'(Q_{j}\hat{\phi}) \frac{\hat{Q}_{i}\Pi(0,b;\eta_{0}(\hat{Q}_{0}),\eta_{1}(\hat{Q}_{1}))}{\int^{a_{1}}_{0}\omega_{i}(a)\Pi(0,\tau;\eta_{0}(\hat{Q}_{0}),\eta_{1}(\hat{Q}_{1}))d\tau} dbda\\
&=\frac{\hat{Q}_{i}\eta_{j}'(Q_{j}\hat{\phi})}{\int^{a_{1}}_{0}\omega_{i}(a)\Pi(0,\tau;\eta_{0}(\hat{Q}_{0}),\eta_{1}(\hat{Q}_{1}))d\tau}\int_{0}^{a_{1}}\omega_{i}(a)\pi(a;\lambda)\\
&\times\int_{0}^{a} e^{\lambda b}\frac{\partial \mu_{j}(b, z)}{\partial z}|_{z = \eta_{j}(Q_{j}\hat{\phi})}  dbda; \text{ for }i,j=0,1.
\end{align*}
Solving the system of equations in terms of $Q_{i}\phi$ for $i=0,1$, to obtain
\begin{align}\label{sec:6.3.6}
Q_{0}\phi&=\phi(0)\frac{(1+I_{\omega_{1}, \mu_{1}}(\lambda))C_{\omega_{0}}(\lambda)-I_{\omega_{0}, \mu_{1}}(\lambda)C_{\omega_{1}}(\lambda)}{\Delta(\lambda)}.\\
\nonumber Q_{1}\phi&=\phi(0)\frac{I_{\omega_{1},\mu_{0} }(\lambda)C_{\omega_{0}}(\lambda)-(1+I_{\omega_{0}, \mu_{0}}(\lambda))C_{\omega_{1}}(\lambda)}{-\Delta(\lambda)}.
\end{align}
where $\Delta(\lambda)=(1+I_{\omega_{0}, \mu_{0}}(\lambda))(1+I_{\omega_{1}, \mu_{1}}(\lambda))-I_{\omega_{0}, \mu_{1}}(\lambda)I_{\omega_{1},\mu_{0} }(\lambda)$. Substitute expressions (\ref{sec:6.3.6}) for $Q_{i}\phi,\ i=0,1$ into (\ref{sec:6.3.3}), we obtain
\begin{align}\label{sec:6.3.7}
\phi(a)&=\phi(0)\pi(a;\lambda)\\
\nonumber&-\phi(0)\frac{\hat{Q}_{0}\eta_{0}'(Q_{0}\hat{\phi})(1+I_{\omega_{1}, \mu_{1}}(\lambda))F_{\omega_{0}}(\lambda)}{\Delta(\lambda)}G_{\mu_{0}}(a;\lambda)\\
\nonumber&+\phi(0)\frac{\hat{Q}_{1}\eta_{0}'(Q_{0}\hat{\phi})I_{\omega_{0}, \mu_{1}}(\lambda)F_{\omega_{1}}(\lambda)}{\Delta(\lambda)}G_{\mu_{0}}(a;\lambda)\\
\nonumber&+\phi(0)\frac{\hat{Q}_{0}\eta_{1}'(Q_{1}\hat{\phi})I_{\omega_{1}, \mu_{0}}(\lambda)F_{\omega_{0}}(\lambda)}{\Delta(\lambda)}G_{\mu_{1}}(a;\lambda)\\
\nonumber&-\phi(0)\frac{\hat{Q}_{1}\eta_{1}'(Q_{1}\hat{\phi})(1+I_{\omega_{0}, \mu_{0}}(\lambda))F_{\omega_{1}}(\lambda)}{\Delta(\lambda)}G_{\mu_{1}}(a;\lambda).
\end{align}
where,
\begin{align*}
&F_{\omega_{i}}(\lambda)=\frac{C_{\omega_{i}}(\lambda)}{\int^{a_{1}}_{0}\omega_{i}(a)\Pi(0,a;\eta_{0}(\hat{Q}_{0}),\eta_{1}(\hat{Q}_{1}))da},\ i=0,1.\\
&G_{\mu_{i}}(a;\lambda)=\pi(a;\lambda)\int_{0}^{a} e^{\lambda b}\frac{\partial \mu_{i}(b, z)}{\partial z}|_{z = \eta_{i}(Q_{i}\hat{\phi})}db,\ i=0,1.
\end{align*}
Substitute (\ref{sec:6.3.7}) into equation (\ref{sec:6.3.2}) to obtain the characteristic equation $K(\lambda)-1=0$ for $\lambda\in\mathbb{C}$, where,
\begin{align}\label{sec:6.3.8}
&K(\lambda):=\int^{a_{\max}}_{a_{\min}}\beta(a; \eta_{2}(Q_{0}\hat{\phi}))\pi(a;\lambda)da\\
\nonumber&-H_{\mu_{0}}(\lambda)\eta_{0}'(Q_{0}\hat{\phi})\int^{a_{\max}}_{a_{\min}}\beta(a; \eta_{2}(Q_{0}\hat{\phi}))G_{\mu_{0}}(a;\lambda)da\\
\nonumber&-H_{\mu_{1}}(\lambda)\eta_{1}'(Q_{1}\hat{\phi})\int^{a_{\max}}_{a_{\min}}\beta(a; \eta_{2}(Q_{0}\hat{\phi}))G_{\mu_{1}}(a;\lambda)da\\
\nonumber&+H_{\mu_{0}}(\lambda)\eta_{2}'(Q_{0}\hat{\phi})\int^{a_{\max}}_{a_{\min}}\frac{\partial \beta(a, z)}{\partial z}|_{z = \eta_{2}(Q_{0}\hat{\phi})}\Pi(0,a; \eta_{0}(Q_{0}\hat{\phi}), \eta_{1}(Q_{1}\hat{\phi}))da.
\end{align}
and,
\begin{align*}
&H_{\mu_{0}}(\lambda)=\frac{\hat{Q}_{0}(1+I_{\omega_{1}, \mu_{1}}(\lambda))F_{\omega_{0}}(\lambda)}{\Delta(\lambda)}-\frac{\hat{Q}_{1}I_{\omega_{0}, \mu_{1}}(\lambda)F_{\omega_{1}}(\lambda)}{\Delta(\lambda)};\\
&H_{\mu_{1}}(\lambda)=-\frac{\hat{Q}_{0}I_{\omega_{1}, \mu_{0}}(\lambda)F_{\omega_{0}}(\lambda)}{\Delta(\lambda)}+\frac{\hat{Q}_{1}(1+I_{\omega_{0}, \mu_{0}}(\lambda))F_{\omega_{1}}(\lambda)}{\Delta(\lambda)}.
\end{align*}

\section{Stability or instability of the linear problem}

\subsection{Stability of the trivial equilibrium}

The following result shows the stability or instability of the trivial equilibrium of the system (\ref{sec:1}) (for details see \cite{W}).

\begin{theorem}\label{sec:TL1}  Let H.1 hold and let $\Delta(0)\neq 0$. The trivial equilibrium is locally uniformly exponentially stable if IGC $\leq 1$. It is unstable if IGC $> 1$.
\end{theorem}

\begin{proof}At the trivial equilibrium $\hat{\phi}=0$, we have $H_{\mu_{i}}(\lambda)=0$ since $\hat{Q}_{i}=0$, for $i=0,1$. The characteristic equation (\ref{sec:6.3.8}) is reduced to:
\begin{align*}
1-\int^{a_{\max}}_{a_{\min}}\beta(a; \eta_{2}(0))e^{-\lambda a}\Pi(0,a; \eta_{0}(0), \eta_{1}(0))da=0.
\end{align*}
Let $\lambda\in\mathbb{C}$ and take the real part of $\lambda$ on both sides, we obtain
\begin{align}\label{sec:TE11Q}
&1=\int^{a_{\max}}_{a_{\min}}e^{-Re(\lambda) a}\cos (Im(\lambda)a)\beta(a; \eta_{2}(0))\Pi(0,a; \eta_{0}(0), \eta_{1}(0))da.
\end{align}
Therefore, if IGC $\leq 1$, we have $Re(\lambda)<0$ for all $\lambda\in\mathbb{C}$. Otherwise, if IGC $>1$, there exists a root $\lambda_{0}\in \mathbb{C}$ of (\ref{sec:TE11Q}) with $Re(\lambda_{0})>0$.

\end{proof}

\subsection{Stability of a positive equilibrium}

Before addressing stability or instability of a positive equilibrium $\hat{\phi}$ (\ref{sec:4.12}) of the system (\ref{sec:1}), we want to study some basic properties of $K(z)$, given by (\ref{sec:6.3.8}) for $z\in \mathbb{R}$. We observe that, for $\lambda\in \mathbb{R}$,
$\lim_{\lambda \rightarrow \infty}\int^{a_{\max}}_{a_{\min}}\beta(a;$ $ \eta_{2}(Q_{0}\hat{\phi}))\pi(a;\lambda)da=0$
and $\lim_{\lambda \rightarrow \infty}C_{\omega_{i}}(\lambda)=\lim_{\lambda \rightarrow \infty}F_{\omega_{i}}(\lambda)=0$, $i=0,1$.
For $I_{\omega_{i}, \mu_{j}}(\lambda)$, $i,j=0,1$, we derive that $\lim_{\lambda \rightarrow \infty}I_{\omega_{i}, \mu_{j}}(\lambda)=0$
by Lebesgue Dominated Convergence Theorem. This is because for $0<b\leq a \leq a_{1}$ and $\lambda>0$, $|e^{-\lambda(a-b)}|\leq 1$ and $\lim_{\lambda \rightarrow \infty}e^{-\lambda(a-b)}=0$. Similarly, we have
$\lim_{\lambda \rightarrow \infty}\int^{a_{\max}}_{a_{\min}}\beta(a; \eta_{2}(Q_{0}\hat{\phi}))$ $G_{\mu_{i}}(a;\lambda)da=0$, $i=0,1$.
It then easily follows that as $\lambda\rightarrow \infty$, $\lim_{\lambda \rightarrow \infty}\Delta(\lambda)=\lim_{\lambda \rightarrow \infty}(1+I_{\omega_{0}, \mu_{0}}(\lambda))(1+I_{\omega_{1}, \mu_{1}}(\lambda))-I_{\omega_{0}, \mu_{1}}(\lambda)I_{\omega_{1},\mu_{0} }(\lambda)=1$ and $\lim_{\lambda \rightarrow \infty}H_{\mu_{i}}(\lambda)=0$, $i=0,1$. Therefore, these limits imply that $\lim_{\lambda \rightarrow \infty}K(\lambda)$ $=0$, the limit is taken in real. Moreover, we derive from (\ref{sec:4.5}) that
$\int^{a_{\max}}_{a_{\min}}\beta(a; \eta_{2}(Q_{0}\hat{\phi}))$ $\pi(a;\lambda)da|_{\lambda=0}=1$. Furthermore, 
we obtain,
\begin{align*}
&I_{\omega_{i}, \mu_{0}}(0)=-\hat{\phi}(0) \int_{0}^{a_{1}}\omega_{i}(a)\frac{\partial  \Pi(0,a; \eta_{0}(z), \eta_{1}(Q_{1}\hat{\phi}))}{\partial z}|_{z = Q_{0}\hat{\phi}}da;\\
&I_{\omega_{i}, \mu_{1}}(0)=-\hat{\phi}(0) \int_{0}^{a_{1}}\omega_{i}(a) \frac{\partial \Pi(0,a; \eta_{0}(Q_{0}\hat{\phi}),\eta_{1}(z) )}{\partial z}|_{z = Q_{1}\hat{\phi}}da.\\
\end{align*}
Moreover,
\begin{align*}
&\eta_{0}'(Q_{0}\hat{\phi})\int^{a_{\max}}_{a_{\min}}\beta(a; \eta_{2}(Q_{0}\hat{\phi}))G_{\mu_{0}}(a;0)da\\
&=-\int^{a_{\max}}_{a_{\min}}\beta(a; \eta_{2}(Q_{0}\hat{\phi}))\frac{ \partial \Pi(0,a; \eta_{0}(z), \eta_{1}(Q_{1}\hat{\phi}))}{\partial z}|_{z = Q_{0}\hat{\phi}}da;\\
&\eta_{1}'(Q_{1}\hat{\phi})\int^{a_{\max}}_{a_{\min}}\beta(a; \eta_{2}(Q_{0}\hat{\phi}))G_{\mu_{1}}(a;0)da\\
&=-\int^{a_{\max}}_{a_{\min}}\beta(a; \eta_{2}(Q_{0}\hat{\phi}))\frac{ \partial \Pi(0,a; \eta_{0}(Q_{0}\hat{\phi}),\eta_{1}(z) )}{\partial z}|_{z = Q_{1}\hat{\phi}}da
\end{align*}
Let
\begin{align*}
&D\mathcal{R}(\hat{\phi})=H_{\mu_{0}}(0)\frac{\partial \int^{a_{\max}}_{a_{\min}}\beta(a; \eta_{2}(z))\Pi(0,a; \eta_{0}(z), \eta_{1}(Q_{1}\hat{\phi}))da}{\partial z}|_{z=Q_{0}\hat{\phi}}\\
&+H_{\mu_{1}}(0)\frac{\partial \int^{a_{\max}}_{a_{\min}}\beta(a; \eta_{2}(Q_{0}\hat{\phi}))\Pi(0,a; \eta_{0}(Q_{0}\hat{\phi}), \eta_{1}(z))da}{\partial z}|_{z=Q_{1}\hat{\phi}}.
\end{align*}
$D\mathcal{R}(\hat{\phi})$ is related to the Frechet derivative of $\mathcal{R}(\phi)$ at a equilibrium solution $\hat{\phi}$ of the system (\ref{sec:1}). We derive $K(0)=1+D\mathcal{R}(\hat{\phi})$.
We summarize what we have so far to obtain the following consequence of the instability condition for a nontrivial equilibrium of the system (\ref{sec:1}).

\begin{theorem}\label{sec:TL2}

Let H.1-H.2 hold. If IGC $>1$, a positive equilibrium solution $\hat{\phi}$ of the system (\ref{sec:1}) is linearly unstable if $D\mathcal{R}(\hat{\phi})> 0$.

\end{theorem}

\begin{proof}

The claim holds if we can show that there exists a positive zero of the characteristic equation (\ref{sec:6.3.8}). This result directly follows from $\lim_{\lambda \rightarrow \infty}K(\lambda)=0$ by the Intermediate Value Theorem since $K$ is real and continuous on $\mathbb{R}$ and $K(0) > 1$
if $D\mathcal{R}(\hat{\phi})> 0$.

\end{proof}
For the linear stability of the positive equilibrium, we make the following assumptions:
\begin{itemize}
\item[H.5.]
\begin{align}
&\eta_{0}'(Q_{0}\hat{\phi})\frac{\partial \mu_{0}(b, z)}{\partial z}|_{z = \eta_{0}(Q_{0}\hat{\phi})} \eta_{1}'(Q_{1}\hat{\phi}) \frac{\partial \mu_{1}(s, z)}{\partial z}|_{z = \eta_{1}(Q_{1}\hat{\phi})}\label{sec:Eq6.3}\\
\nonumber&-\eta_{1}'(Q_{1}\hat{\phi})\frac{\partial \mu_{1}(b, z)}{\partial z}|_{z = \eta_{1}(Q_{1}\hat{\phi})}\eta_{0}'(Q_{0}\hat{\phi})\frac{\partial \mu_{0}(s, z)}{\partial z}|_{z = \eta_{0}(Q_{0}\hat{\phi})} \leq 0,\\
\nonumber&\text{ for }(b,s)\in [0,a_{1}]\times[0,a_{1}].\\
&(\omega_{1}(a)\omega_{0}(y)-\omega_{0}(a)\omega_{1}(y))\eta_{1}'(Q_{1}\hat{\phi})\int_{0}^{a}\frac{\partial \mu_{1}(b, z)}{\partial z}|_{z = \eta_{1}(Q_{1}\hat{\phi})}db\geq 0,\label{sec:Eq6.4}\\
\nonumber&\text{ for }(a,y)\in [0,a_{1}]\times[0,a_{1}];\\
\nonumber&(\omega_{1}(y)\omega_{0}(a)-\omega_{0}(y)\omega_{1}(a))\eta_{0}'(Q_{0}\hat{\phi})\int_{0}^{a}\frac{\partial \mu_{0}(b, z)}{\partial z}|_{z = \eta_{0}(Q_{0}\hat{\phi})}db\geq 0,\\
\nonumber&\text{ for }(a,y)\in [0,a_{1}]\times[0,a_{1}].
\end{align}
\end{itemize}

\begin{theorem}\label{sec:T7.1}
Let H.1-H.2, and H.4-H.5 hold. Let IGC $>1$, and let $\hat{\phi}$ be a nontrivial equilibrium of the system (\ref{sec:1}). Then  $\hat{\phi}$ is locally asymptotically stable if and only if $D\mathcal{R}(\hat{\phi}) < 0$.
\end{theorem}

\begin{proof}

By Theorem \ref{sec:T6.6}, we could restrict ourselves to $\lambda\in \mathbb{R}$ to derive the linear stability condition for a positive equilibrium solution $\hat{\phi}$ of the system (\ref{sec:1}). If $D\mathcal{R}(\hat{\phi})<0$, then $\hat{\phi}$ will be linearly asymptotically stable if we can show that the characteristic function $K(z)$ is nonincreasing for $z\geq 0$. First, we observe that $C_{\omega_{i}}(z),\ F_{\omega_{i}}(z)$; $\int^{a_{\max}}_{a_{\min}}\beta(a; \eta_{2}(Q_{0}\hat{\phi}))\pi(a; z)da$; $-I_{\omega_{i}, \mu_{j}}(z)$;
$-\eta_{i}'(Q_{i}\hat{\phi})\int^{a_{\max}}_{a_{\min}}\beta(a;$ $ \eta_{2}(Q_{0}\hat{\phi}))G_{\mu_{i}}(a;z)da$, $i,j=0,1$, for $z\geq 0$ are all nonincreasing functions by H.4. Therefore, to show $K(z)$ is nonincreasing, it suffices to show $H_{\mu_{i}}(z)$, $i=0,1$ is nonincreasing for $z\geq 0$. Next we want to show $\Delta(z)$ is nondecreasing for $z\geq 0$ under (\ref{sec:Eq6.3}). We recall $\Delta(\lambda)=1+I_{\omega_{0}, \mu_{0}}(\lambda)+I_{\omega_{1}, \mu_{1}}(\lambda)+I_{\omega_{0}, \mu_{0}}(\lambda)I_{\omega_{1}, \mu_{1}}(\lambda)-I_{\omega_{0}, \mu_{1}}(\lambda)I_{\omega_{1},\mu_{0} }(\lambda)$. To show $\Delta(\lambda)$ is nondecreasing, it suffices to show that $I_{\omega_{0}, \mu_{0}}(\lambda)I_{\omega_{1}, \mu_{1}}(\lambda)-I_{\omega_{0}, \mu_{1}}(\lambda)I_{\omega_{1},\mu_{0} }(\lambda)$ is nondecreasing for $\lambda\geq 0$.
\begin{align*}
&I_{\omega_{0}, \mu_{0}}(\lambda)I_{\omega_{1}, \mu_{1}}(\lambda)-I_{\omega_{0}, \mu_{1}}(\lambda)I_{\omega_{1},\mu_{0} }(\lambda)\\
&=\frac{\hat{Q}_{0}\eta_{0}'(Q_{0}\hat{\phi})}{\int^{a_{1}}_{0}\omega_{0}(a)\Pi(0,\tau;\eta_{0}(\hat{Q}_{0}),\eta_{1}(\hat{Q}_{1}))d\tau}\int_{0}^{a_{1}}\omega_{0}(a)\pi(a; \lambda)\int_{0}^{a} e^{\lambda b}\frac{\partial \mu_{0}(b, z)}{\partial z}|_{z = \eta_{0}(Q_{0}\hat{\phi})}  dbda\\
&\times \frac{\hat{Q}_{1}\eta_{1}'(Q_{1}\hat{\phi})}{\int^{a_{1}}_{0}\omega_{1}(a)\Pi(0,\tau;\eta_{0}(\hat{Q}_{0}),\eta_{1}(\hat{Q}_{1}))d\tau}\int_{0}^{a_{1}}\omega_{1}(a)\pi(a; \lambda)\int_{0}^{a} e^{\lambda b}\frac{\partial \mu_{1}(b, z)}{\partial z}|_{z = \eta_{1}(Q_{1}\hat{\phi})}  dbda\\
&- \frac{\hat{Q}_{0}\eta_{1}'(Q_{1}\hat{\phi})}{\int^{a_{1}}_{0}\omega_{0}(a)\Pi(0,\tau;\eta_{0}(\hat{Q}_{0}),\eta_{1}(\hat{Q}_{1}))d\tau} \int_{0}^{a_{1}}\omega_{0}(a)\pi(a; \lambda)\int_{0}^{a} e^{\lambda b}\frac{\partial \mu_{1}(b, z)}{\partial z}|_{z = \eta_{1}(Q_{1}\hat{\phi})}dbda\\
&\times\frac{\hat{Q}_{1}\eta_{0}'(Q_{0}\hat{\phi})}{\int^{a_{1}}_{0}\omega_{1}(a)\Pi(0,\tau;\eta_{0}(\hat{Q}_{0}),\eta_{1}(\hat{Q}_{1}))d\tau}\int_{0}^{a_{1}}\omega_{1}(a)\pi(a; \lambda)\int_{0}^{a} e^{\lambda b}\frac{\partial \mu_{0}(b, z)}{\partial z}|_{z = \eta_{0}(Q_{0}\hat{\phi})} dbda.\\
&=\hat{\phi}^{2}(0)\eta_{0}'(Q_{0}\hat{\phi})\eta_{1}'(Q_{1}\hat{\phi})[\int_{0}^{a_{1}}\int_{0}^{a_{1}}\omega_{0}(a)\omega_{1}(t)\pi(a; \lambda)\pi(t;\lambda)\int_{0}^{a} e^{\lambda b}\frac{\partial \mu_{0}(b, z)}{\partial z}|_{z = \eta_{0}(Q_{0}\hat{\phi})}db\\
&\times\int_{0}^{t} e^{\lambda s}\frac{\partial \mu_{1}(s, z)}{\partial z}|_{z = \eta_{1}(Q_{1}\hat{\phi})}  dsdadt- \int_{0}^{a_{1}}\int_{0}^{a_{1}}\omega_{0}(a)\omega_{1}(t)\pi(a;\lambda)\pi(t;\lambda)\\
&\times\int_{0}^{a} e^{\lambda b}\frac{\partial \mu_{1}(b, z)}{\partial z}|_{z = \eta_{1}(Q_{1}\hat{\phi})}db\int_{0}^{t} e^{\lambda s}\frac{\partial \mu_{0}(s, z)}{\partial z}|_{z = \eta_{0}(Q_{0}\hat{\phi})} dsdadt]\\
&=\hat{\phi}^{2}(0)\eta_{0}'(Q_{0}\hat{\phi})\eta_{1}'(Q_{1}\hat{\phi})\int_{0}^{a_{1}}\int_{0}^{a_{1}}\omega_{0}(a)\omega_{1}(t)\pi(a;\lambda)\pi(t;\lambda)\int_{0}^{a}\int_{0}^{t} e^{\lambda b}e^{\lambda s}\\
&\times[\frac{\partial \mu_{0}(b, z)}{\partial z}|_{z = \eta_{0}(Q_{0}\hat{\phi})}\frac{\partial \mu_{1}(s, z)}{\partial z}|_{z = \eta_{1}(Q_{1}\hat{\phi})}\\
&-\frac{\partial \mu_{1}(b, z)}{\partial z}|_{z = \eta_{1}(Q_{1}\hat{\phi})}\frac{\partial \mu_{0}(s, z)}{\partial z}|_{z = \eta_{0}(Q_{0}\hat{\phi})}]dbdsdadt.
\end{align*}
Therefore, $\frac{1}{\Delta(\lambda)}$ is nonincreasing for $\lambda\geq 0$.  We recall that $F_{\omega_{i}}(z)$, $i=0,1$ is nonincreasing for $z\geq 0$. To show $H_{\mu_{i}}(\lambda)$, $i=0,1$ is nonincreasing for $\lambda\geq 0$, it suffices to show that
\begin{align}
\nonumber&\hat{Q}_{0}I_{\omega_{1}, \mu_{1}}(\lambda)F_{\omega_{0}}(\lambda)-\hat{Q}_{1}I_{\omega_{0}, \mu_{1}}(\lambda)F_{\omega_{1}}(\lambda)\\
&=\hat{\phi}(0)[I_{\omega_{1}, \mu_{1}}(\lambda)C_{\omega_{0}}(\lambda)-I_{\omega_{0}, \mu_{1}}(\lambda)C_{\omega_{1}}(\lambda)];\label{sec:7.1}\\
\nonumber&-\hat{Q}_{0}I_{\omega_{1}, \mu_{0}}(\lambda)F_{\omega_{0}}(\lambda)+\hat{Q}_{1}I_{\omega_{0}, \mu_{0}}(\lambda)F_{\omega_{1}}(\lambda)\\
&=\hat{\phi}(0)[-I_{\omega_{1}, \mu_{0}}(\lambda)C_{\omega_{0}}(\lambda)+I_{\omega_{0}, \mu_{0}}(\lambda)C_{\omega_{1}}(\lambda)].\label{sec:7.2}
\end{align}
are nonincreasing for $\lambda\geq 0$. For (\ref{sec:7.1}), we obtain
\begin{align*}
&I_{\omega_{1}, \mu_{1}}(\lambda)C_{\omega_{0}}(\lambda)-I_{\omega_{0}, \mu_{1}}(\lambda)C_{\omega_{1}}(\lambda)\\
&=\frac{\hat{Q}_{1}\eta_{1}'(Q_{1}\hat{\phi})}{\int^{a_{1}}_{0}\omega_{1}(a)\Pi(0,\tau;\eta_{0}(\hat{Q}_{0}),\eta_{1}(\hat{Q}_{1}))d\tau}\int_{0}
^{a_{1}}\omega_{1}(a)\pi(a; \lambda)\\
&\times\int_{0}^{a} e^{\lambda b}\frac{\partial \mu_{1}(b, z)}{\partial z}|_{z = \eta_{1}(Q_{1}\hat{\phi})}  dbda\int_{0}^{a_{1}}\omega_{0}(a)\pi(a; \lambda)da\\
&- \frac{\hat{Q}_{0}\eta_{1}'(Q_{1}\hat{\phi})}{\int^{a_{1}}_{0}\omega_{0}(a)\Pi(0,\tau;\eta_{0}(\hat{Q}_{0}),\eta_{1}(\hat{Q}_{1}))d\tau} \int_{0}^{a_{1}}\omega_{0}(a)\pi(a; \lambda)\\
&\times\int_{0}^{a} e^{\lambda b}\frac{\partial \mu_{1}(b, z)}{\partial z}|_{z = \eta_{1}(Q_{1}\hat{\phi})}dbda\int_{0}^{a_{1}}\omega_{1}(a)\pi(a; \lambda)da\\
&=\hat{\phi}(0)\eta_{1}'(Q_{1}\hat{\phi})[\int_{0}^{a_{1}}\int_{0}^{a_{1}}[\omega_{1}(a)\omega_{0}(t)-\omega_{0}(a)\omega_{1}(t)]\pi(a; \lambda)\pi(t; \lambda)\\
&\times\int_{0}^{a} e^{\lambda b}\frac{\partial \mu_{1}(b, z)}{\partial z}|_{z = \eta_{1}(Q_{1}\hat{\phi})}  dbdadt.
\end{align*}
Similarly, for (\ref{sec:7.2}), we obtain
\begin{align*}
&-I_{\omega_{1}, \mu_{0}}(\lambda)C_{\omega_{0}}(\lambda)+I_{\omega_{0}, \mu_{0}}(\lambda)C_{\omega_{1}}(\lambda)=\hat{\phi}(0)\eta_{0}'(Q_{0}\hat{\phi})\int_{0}^{a_{1}}\int_{0}^{a_{1}}\\
&[\omega_{1}(t)\omega_{0}(a)-\omega_{0}(t)\omega_{1}(a)]\pi(a; \lambda)\pi(t; \lambda)\int_{0}^{a} e^{\lambda b}\frac{\partial \mu_{0}(b, z)}{\partial z}|_{z = \eta_{0}(Q_{0}\hat{\phi})}  dbdadt.
\end{align*}
Therefore, if (\ref{sec:Eq6.4}) holds, then $H_{\mu_{i}}(\lambda)$, $i=0,1$ is nonincreasing for $\lambda\geq 0$. Then, the conclusion follows.
\end{proof}

\chapter{Asymptotic Behavior of the model}\label{chap:asym}

In section 4 we established conditions which guarantee the existence of either only the trivial equilibrium or also a positive equilibrium. The next natural step is to study the stability of an equilibrium of the model (\ref{sec:1}). Our approach to this problem involves the use of the invariance principle of J. LaSalle through finding the smallest closed set to which a trajectory will converge as times goes to infinity.

\section{Preliminaries}
\begin{definition}
The functions $t \mapsto U(t)\phi$ for fixed $\phi\in X_{+}$ are \textit{(positive) trajectories} of the nonlinear semigroup associated with system (\ref{sec:1}), defined for all positive times $t\in [0, \infty)$. $\left\{U(t)\phi: t\leq 0\right\}$ is the \textit{negative orbit (or trajectory)} through $\phi\in X_{+}$. 
$\left\{U(t)\phi: t\in\mathbb{R}\right\}$ is a \textit{complete orbit (or trajectory)} through $\phi\in X_{+}$. 
Let $U(t), \ t\geq 0$, be a strongly continuous nonlinear semigroup in the closed subset $C$ of the Banach space $X$. The \textit{omega-limit set} of $\phi$ for $\phi\in C$, denoted by $\Omega(\phi)$, is 
\begin{align*} 
\Omega(\phi)=\left\{x_{1}\in X: \exists \left\{t_{k}\right\}_{k=1}^{\infty}\in\mathbb{R}_{+}\text{ such that }t_{k} \rightarrow \infty \text{ and } U(t_{k})\phi \rightarrow x_{1} \right\}.
\end{align*}
The \textit{alpha-limit set} of $\phi$ for $\phi\in C$, denoted by $\alpha(\phi)$, is 
\begin{align*} 
\alpha(\phi)=\left\{x_{1}\in X: \exists \left\{t_{k}\right\}_{k=1}^{\infty}\in\mathbb{R}_{-}\text{ such that }t_{k} \rightarrow -\infty \text{ and } U(t_{k})\phi \rightarrow x_{1} \right\}.
\end{align*} 
\end{definition}

\begin{definition}
A set $B\subset X$ is said to \textit{attract} a set $C\subset X$ under $U(t)$ if $\text{dist}(U(t)C,B)\rightarrow 0$ as $t\rightarrow \infty$. A set $S\subset X$ is said to be \textit{invariant} if, for any $\phi_{0}\in S$, there is a complete orbit $\left\{U(t)\phi_{0}:t\in\mathbb{R}\right\}$ through $\phi_{0}$ such that $\left\{U(t)\phi_{0}:t\in\mathbb{R}\right\}\subset S$. For a given continuous map $U:X\rightarrow X$, a compact invariant set $A$ is said to be a \textit{maximal compact invariant set} if every compact invariant set of $U$ belongs to $A$. An invariant set $\mathcal{A}_{0}$ is said to be \textit{a global attractor} if $\mathcal{A}_{0}$ is a maximal compact invariant set which attracts each bounded set $B\subset X$(see \cite{Ha1}). 
\end{definition}

\begin{definition}
Let $Y$ be a Banach space and $T(t): Y\rightarrow Y$, for $t\geq 0$, be a strongly continuous semigroup on $Y$, $T(t),t\geq 0$ is \textit{point dissipative} in $Y$ if there exists a bounded nonempty set $B$ in $Y$ such that for any $y\in Y$, there exists a $t_{0}=t_{0}(y,B)$, such that $T(t)y\in B$ for $t\geq t_{0}$ \cite{Ha}. $T(t),t\geq 0$ is \textit{asymptotically smooth} if every positively invariant bounded set is attracted by a compact subset \cite{Ma}. 
\end{definition}

\begin{definition}
Let $M_{0}$ be an open subset of $X_{+}$, and $\partial M_{0} = X_{+} - M_{0}$, if $U(t)\partial M_{0}\subset\partial M_{0}$, and $U(t)M_{0}\subset M_{0}$, for $\forall t\geq 0$, then the strongly continuous nonlinear semigroup $U(t),t\geq 0$, defined on $X_{+}$ is \textit{uniformly persistent} with respect to $(M_{0}, \partial M_{0})$, if there exists some $\epsilon >0$ such that $\liminf_{t\rightarrow \infty}d(U(t)x, \partial M_{0})\geq\epsilon$ for $\forall x\in M_{0}$, where the distance $d$ is induced by the $L^{1}$-norm \cite{ Ha1,Ha}. 
\end{definition}

If we can show these trajectories have compact closures, then the following Proposition from (\cite{Wa}, Proposition 4.1, pp.166 and Theorem 4.1, pp.167) assures the existence of the smallest closed set to which the trajectory approaches as time approaches infinity. 

\begin{proposition}
Let $U(t), \ t\geq 0$ be a strongly continuous nonlinear semigroup in the closed subset $C$ of the Banach space $X$ and let $x\in C$. Then $\Omega(x)$ is closed, positive invariant, and a subset of the closure of $\left\{U(t)x:\ t\geq 0\right\}$. If $\left\{U(t)x:\ t\geq 0\right\}$ has compact closure, then $\Omega(x)$ is nonempty, compact, connected and invariant. Moreover, $U(t)x$ approaches  $\Omega(x)$ as $t$ approaches infinity in the sense that 
\begin{align*} 
\liminf_{t \rightarrow \infty\ x_{1}\in\Omega(x)}\left\|U(t)x-x_{1}\right\|=0.
\end{align*}
Further, $\Omega(x)$ is the smallest closed set that $U(t)x$ approaches as $t$ approaches infinity, in the sense that if $U(t)x$ approaches a set $C_{1}\subset C$ as $t$ approaches infinity, then $\Omega(x)\subset \bar{C}_{1}$. 
\end{proposition} 

Further, the following LaSalle's Invariance Principle (\cite{Wa}, Theorem 4.2, pp.168) provides a method to identify the \textit{omega-limit set}.

A \textit{Lyapunov function} for a strongly continuous nonlinear semigroup $U(t),\ t \geq 0$ (in the closed set $C$ of the Banach space $X$) on $C_{1}$ ($C_{1}\subseteq C$) is a continuous function $V: C \rightarrow \mathbb{R}$ such that for all $x\in C_{1}$,
\begin{align*} 
\dot{V}(x)&:=\limsup_{t \rightarrow 0^{+}}\frac{V(U(t)x)-V(x)}{t}\leq 0.
\end{align*}
(where we allow the possibility that $\dot{V}(x)=-\infty$).

\begin{proposition}
Let $U(t), \ t\geq 0$, be a strongly continuous nonlinear semigroup in the closed subset $C$ of the Banach space $X$, let $C_{1}\subseteq C$ and let $V$ be a Lyapunov function for $U(t), \ t\geq 0$, on $\bar{C}_{1}$. Let $x\in C$ be such that $\left\{U(t)x:\ t\geq 0\right\}\subset C_{1}$ and $\left\{U(t)x:\ t\geq 0\right\}$ has compact closure. Then $\Omega(x)\subset M^{+}$, where $M^{+}$ is the largest positive invariant subset of $M_{1}:=\left\{x_{1}\in \bar{C}_{1}: \dot{V}(x_{1})=0\right\}$. (In fact, $\Omega(x)\subset M^{+}\cap\left\{x_{1}\in \bar{C}_{1}: \dot{V}(x_{1})=k\right\}$ for some constant $k$). Further, $U(t)x$ approaches  $M$ as $t$ approaches infinity, where $M$ is the largest invariant subset of $M_{1}$.
\end{proposition} 

The main technical difficulty involved in showing  the trajectories of $U$ have compact closure is that it is not automatic that a bounded trajectory lies in a compact set in an infinite dimensional Banach space. The following Theorem (which is an adaption of Theorem 3.5 in \cite{W} section 3.4, pp.112) allows us to work around this difficulty. 

\begin{theorem}\label{sec:T5.1}
Let H.1 hold. Let $U(t), \ t\geq 0$ be the strongly continuous nonlinear semigroup in $X_{+}$ as in Theorem \ref{sec:T3.3} and let $U(t), \ t\geq 0$ have the property that if $t>0$ and $M$ is a bounded subset of $X_{+}$, then there exists $r>0$ such that $\left\|U(s)\phi\right\|_{X} \leq r $ for all $\phi\in M$, $0 \leq s \leq t$. If $\phi\in X_{+}$ and $\left\{U(t)\phi: \ t\geq 0\right\}$ is bounded in $L^{1}$, then $\left\{U(t)\phi: \ t\geq 0\right\}$ has compact closure in $L^{1}$.
\end{theorem}

The proof of this theorem is accomplished by decomposing $U(t)$ into $U(t)=W_{1}(t)+W_{2}(t)$, with mappings $W_{1}(t), \ W_{2}(t)\in X_{+}$ for $t\geq 0$ given as follows: for $\phi\in X_{+}$,
\begin{align} 
&(W_{1}(t)\phi)(a)=\begin{cases}\label{sec:W1}
0 & \text{ a.e. } \ \ a \in (0,t)\cap[0,a_{1}];\\
(U(t)\phi)(a) & \text{ a.e. } \ \ a\in (t, a_{1}];
\end{cases}\\
&(W_{2}(t)\phi)(a)=\begin{cases}\label{sec:W2}
(U(t)\phi)(a) & \text{ a.e. } \ \ a \in (0,t)\cap[0,a_{1}];\\
0 & \text{ a.e. } \ \ a\in (t, a_{1}].
\end{cases}
\end{align}
It is easy to see that $W_{1}(t)=0$ for $t>a_{1}$ while $W_{2}(t)$ (see \cite{W} section 3.4 pp.112) is ultimately compact by using the measure of noncompactness due to Kuratowski under suitable hypotheses on $\mathcal{F}, \mathcal{G}$ (\ref{sec:2})-(\ref{sec:3}). Theorem \ref{sec:T5.1} assures that a trajectory has compact closure, provided that it is bounded. From Theorem \ref{sec:T3.3}, we obtain $|W_{1}(t)|\leq e^{\omega a_{1}}$, for $t\geq 0$, where $\omega=|\mathcal{F}|+|\mathcal{G}|$. Let H.1 hold. The boundedness of $W_{2}(t), \ t\geq 0$ follows from Proposition 3.13 (in \cite{W}, section 3.4, pp.100). Therefore, the boundedness of a trajectory of the strongly continuous nonlinear semigroup $U(t), \ t\geq 0$ in $X_{+}$ as in Theorem \ref{sec:T3.3} follows.

Applying Theorem \ref{sec:T5.1}, we obtain the following corollary:

\begin{corollary}\label{sec:LMCOR}
Let H.1 hold. Then, the strongly continuous nonlinear semigroup $U(t), \ t\geq 0$ in $X_{+}$ as in Theorem \ref{sec:T3.3} is bounded dissipative and asymptotically smooth. 

\end{corollary}

\section{The linear problem}
\begin{itemize}
\item[H.6.] Let $\tilde{\beta}\in L^{\infty}_{+}[a_{\min},a_{\max}]$, 
and $\tilde{\mu}\in L^{\infty}_{+}[0,a_{1}]$.
\end{itemize}
In this section, we study some basic properties of the following linear age-structured model to apply the comparison argument:
\begin{align}\label{sec:LM1}
&l_{t}(a,t)+l_{a}(a,t)=-\tilde{\mu}(a)l(a,t),\\ 
\nonumber & 0<a<a_{1},\ t>0,\\
\nonumber &l(0,t)=\int_{a_{\min}}^{a_{\max}}\tilde{\beta}(a)l(a,t)da,\ t>0,\\
\nonumber &l(a,0)=p_{0}(a),\ 0<a<a_{1}.
\end{align}
where $p_{0}\in X_{+}$.

We define bounded linear operators $\hat{\mathcal{F}}_{L}:X\rightarrow \mathbb{R}$, $\hat{\mathcal{G}}_{L}:X\rightarrow X$ 
for $\forall\phi\in X$ by,
\begin{align}
&\hat{\mathcal{G}}_{L}(\phi)(a):=-\tilde{\mu}(a)\phi(a).\label{sec:LM2}\\ 
&\hat{\mathcal{F}}_{L}(\phi):=\int_{a_{\min}}^{a_{\max}}\tilde{\beta}(a)\phi(a)da.\label{sec:LM3}
\end{align}
Let $\hat{A}_{L}: D(\hat{A}_{L})\subset X\rightarrow X$ be the linear operator defined by,
\begin{align}\label{sec:LM51} 
&\hat{A}_{L}\phi=-\phi'+\hat{\mathcal{G}}_{L}(\phi) ,\ \text{ for }\phi\in D(\hat{A}_{L}).
\end{align}
where,
\begin{align*} 
&D(\hat{A}_{L})=\left\{\phi\in X: \phi \text{ is absolutely continuous on }[0,a_{1}], \phi'\in L^{1}, \phi(0)=\hat{\mathcal{F}}_{L}(\phi)\right\}.
\end{align*}

The following well-known result in the context of age structured models establishes the semigroup properties of solutions of the system  (\ref{sec:LM1}), we refer to Proposition 3.2 and 3.7 in (\cite{W}, section 3.1, pp.76) for proofs.
\begin{theorem}\label{sec:LM4}
Let H.6 hold. Let $\hat{\mathcal{F}}_{L},\hat{\mathcal{G}}_{L}$ be bounded linear operators defined as in (\ref{sec:LM2})-(\ref{sec:LM3}). If $\phi_{0}\in X$, then the solution $l(a,t)$, for $(a,t)\in[0,a_{1}]\times[0,\infty)$ of the system (\ref{sec:LM1}) is defined on $[0,\infty)$. Further, the family of mappings $\hat{T}_{L}(t), \ t\geq 0$, in $X$ defined by $(\hat{T}_{L}(t)\phi_{0})(a)=l(a,t)$, for $(a,t)\in[0,a_{1}]\times[0,\infty)$, with intial condition $\phi_{0}\in X$, is a strongly continuous semigroup of bounded linear operators in $X$ with the infinitesimal generator $\hat{A}_{L}$, defined by (\ref{sec:LM51}). Moreover,
\begin{align}\label{sec:LMB1}
&(\hat{T}_{L}(t)\phi_{0})(a)=\begin{cases}
B(t-a)\prod(0,a) & \text{ a.e. } \ \ a \in [0,t)\cap[0,a_{1}];\\
\phi_{0}(a-t)\prod(a-t,a) & \text{ a.e. } \ \ a\in[t,a_{1}].
\end{cases}
\end{align}
where $\prod(b,a):=e^{-\int_{b}^{a}\tilde{\mu}(\hat{a})d\hat{a}}$, for $0\leq b\leq a$, and,
\begin{align}\label{sec:LMB}
B(t)&=\int_{a_{\min}}^{t}\tilde{\beta}(a)\prod(0,a)B(t-a)da+\int_{t}^{a_{\max}}\tilde{\beta}(a)\prod(a-t,a)\phi_{0}(a-t)da.
\end{align}
Furthermore,
\begin{align*} 
|\hat{T}_{L}(t)|\leq e^{\omega t}, \ \text{for}\ t\geq 0\ \ \ \ \text{ where }\omega=|\hat{\mathcal{F}}_{L}|+|\hat{\mathcal{G}}_{L}|.
\end{align*}
Further, for all $t\geq 0$, $\hat{T}_{L}(t)(D(\hat{A}_{L}))\subset D(\hat{A}_{L})$ and $(d/dt )\hat{T}_{L}(t)\phi=\hat{A}_{L}\hat{T}_{L}(t)\phi$ $=\hat{T}_{L}(t)\hat{A}_{L}\phi$ for all $\phi\in D(\hat{A}_{L})$.
\end{theorem}

\section{Growth estimates of the linear semigroup}

Applying the same methods as in the previous section, we can readily show, in the similar way as in Theorem \ref{sec:ES1}  and Theorem \ref{sec:T6.6}, that the strongly continuous linear semigroup $\hat{T}_{L}(t), t\geq 0$ in $X_{+}$ as in Theorem \ref{sec:LM4} is eventually compact, positive and irreducible. Therefore, the spectrum of the linear operator consists of eigenvalues of finite multiplicity, which can be determined via zeros of a characteristic function, see \cite{Pa, E, E6} for more details.

The following theorem estimates the essential growth rate of the linear semigroup $\hat{T}_{L}(t), \ t\geq 0$.

\begin{theorem}\label{sec:LMT6} Let H.6 hold. Let $\hat{T}_{L}(t), \ t\geq 0$, be the strongly continuous linear semigroup in $X$ as in Theorem \ref{sec:LM4} with infinitesimal generator $\hat{A}_{L}$, defined by (\ref{sec:LM51}), then,
\begin{align*}
\omega_{0, ess}(\hat{A}_{L})=-\infty.
\end{align*}

\end{theorem}

\begin{proof}
Let $\hat{T}_{L}(t)=\tilde{V}_{1}(t)+\tilde{V}_{2}(t)$, where the mappings $\tilde{V}_{1}(t)$, $\tilde{V}_{2}(t)\in X$ for $t\geq 0$, $\phi\in X$ are defined as follows:
\begin{align}\label{sec:LM6}   
&(\tilde{V}_{1}(t)\phi)(a)=\begin{cases}
0 & \text{ a.e. } \ \ a \in (0,t)\cap[0,a_{1}];\\
(\hat{T}_{L}(t)\phi)(a) & \text{ a.e. } \ \ a\in (t, a_{1}];
\end{cases}\\
&(\tilde{V}_{2}(t)\phi)(a)=\begin{cases}\label{sec:LM7}  
(\hat{T}_{L}(t)\phi)(a) & \text{ a.e. } \ \ a \in (0,t)\cap[0,a_{1}];\\
0 & \text{ a.e. } \ \ a\in (t, a_{1}].
\end{cases}
\end{align} 
It is easy to see that $\tilde{V}_{1}(t)=0$ for $t>a_{1}$ while $\tilde{V}_{2}(t)$ (see \cite{W}, section 3.4, pp.112) is ultimately compact by using the measure of noncompactness due to Kuratowski by Proposition 3.17 (in \cite{W}, section 3.5, pp.113). Therefore, from Proposition 4.9 (in \cite{W}, section 4.3, pp.166) $\alpha[\hat{T}_{L}(t)]\leq \alpha[\tilde{V}_{1}(t)]+\alpha[\tilde{V}_{2}(t)]=0$ for $t>a_{1}$, where $\alpha$ is the measure of noncompactness of $\hat{T}_{L}$ defined in \cite{W}, section 4.3, pp.165. The claim then follows directly.
\end{proof}

The following theorem establishes the positivity and irreducibility of the strongly continuous linear semigroup $\hat{T}_{L}(t),t \geq 0$ in $X_{+}$ as in Theorem \ref{sec:LM4}.

\begin{theorem}\label{sec:TLM7}
Let H.6 hold. Then the strongly continuous linear semigroup $\hat{T}_{L}(t),t \geq 0$ in $X_{+}$ as in Theorem \ref{sec:LM4} is positive and irreducible.
\end{theorem}

\begin{proof}

The associated differential equation subject to the corresponding boundary condition is given by (\ref{sec:LM1}). 
Let $l$ be a solution of (\ref{sec:LM1}). 
Then the function $w$ given by
\begin{align*}
w(a,t)=l(a,t)e^{\int_{0}^{a}\tilde{\mu}(\hat{a})d\hat{a}}.
\end{align*}
satisfies 
\begin{align*}
&w_{t}(a,t)+w_{a}(a,t)=0,\ 0<a<a_{1},\ t>0,\\
&w(0,t)=\psi(w(a,t)), \ t>0,\\
&w(a,0)=p_{0}(a)e^{\int_{0}^{a}\tilde{\mu}(\hat{a})d\hat{a}},\ 0<a<a_{1}.
\end{align*}
where, $\psi(w(a,t))=$ $\hat{\mathcal{F}}_{L}(w(a,t)e^{-\int_{0}^{a}\tilde{\mu}(\hat{a})d\hat{a}})$.
Solutions of this system form a strongly continuous linear semigroup with infinitesimal generator $\mathcal{B}\phi=-\phi'$ for $\phi\in D(\mathcal{B})$, and $D(\mathcal{B})$ is given by,
\begin{align*}
D(\mathcal{B})=\left\{\phi\in X: \phi \text{ is absolutely continuous on }[0,a_{1}], \phi'\in L^{1}, \phi(0)=\psi(\phi)\right\}. 
\end{align*}
It then suffices to show that the semigroup generated by $\mathcal{B}$ is nonnegative. We observe that the resolvent equation $\lambda w-\mathcal{B} w= f$ has the  solution $w(a)= e^{-\lambda a}\psi(w)+\int_{0}^{a}e^{-\lambda (a-b)}f(b)db$ for $\lambda \geq 0$ sufficiently large and $f\in X$. Applying $\psi$ on both sides, we get $\psi(w)=(1-\psi(e^{-\lambda a}))^{-1} \psi(\int_{0}^{a}e^{-\lambda (a-b)}f(b)db)$. From the definition of $\psi$ 
we obtain that the solution $w$, is nonnegative if $f$ is nonnegative a.e. and $\lambda$ is sufficiently large. Therefore, the resolvent operator of $\mathcal{B}$ is positive for sufficiently large $\lambda$. Then the conclusion follows from Proposition \ref{sec:P6.1} (iii).

\end{proof}

Consider the characteristic equation $\Delta(\lambda)$ of the system (\ref{sec:LM1}) defined by, 
\begin{align*}
\Delta(\lambda):=1-\int_{a_{\min}}^{a_{\max}}e^{-\lambda a}\tilde{\beta}(a)\prod(0,a)da.
\end{align*}
We define $\mathcal{\tilde{R}}_{0}=\int_{a_{\min}}^{a_{\max}}\tilde{\beta}(a)\prod(0,a)da$. 
We state some results (see Theorem 4.9 \cite{W}, sec 4.3, pp.184-187), which is an adaption for the strongly continuous linear semigroup $\hat{T}_{L}(t),\ t \geq 0$ in $X$ as in Theorem \ref{sec:LM4} .

\begin{theorem}\label{sec:LMCC2}
Let H.6 hold. Let $\hat{T}_{L}(t),\ t \geq 0$ be the strongly continuous linear semigroup in $X$ as in Theorem \ref{sec:LM4} with infinitesimal generator $\hat{A}_{L}$ as in (\ref{sec:LM51}). The following hold:
\begin{align}
&\text{If }\Delta(\lambda)=0,\text{ then }\lambda\in P\sigma(\hat{A}_{L}).\label{sec:LMEE6}\\
&\text{If }\lambda\in \rho(\hat{A}_{L}), \text{ then for } \psi\in X, a\in[0,a_{1}],\label{sec:LMEE7}\\
\nonumber&((\lambda I-\hat{A}_{L})^{-1}\psi)(a)=\int_{0}^{a}e^{\lambda (a-b)}\prod(b,a)\psi(b)db\\
\nonumber&+e^{-\lambda a}\prod(0,a)\Delta(\lambda)^{-1}\int_{a_{\min}}^{a_{\max}}\tilde{\beta}(b)e^{-\lambda b}[\int_{0}^{b}e^{\lambda \tau}\prod(\tau,b)\psi(\tau)d\tau]db.
\end{align}

\end{theorem}

\begin{theorem}\label{sec:LMCC3}
Let H.6 hold. Let $\hat{T}_{L}(t),t \geq 0$ be the strongly continuous linear semigroup in $X$ as in Theorem \ref{sec:LM4} with infinitesimal generator $\hat{A}_{L}$ as in (\ref{sec:LM51}). The following are equivalent:
\begin{align}
&\lambda_{0}\in\sigma(\hat{A}_{L}).\label{sec:LMEE8}\\
&\lambda_{0}\in\sigma(\hat{A}_{L})-E\sigma(\hat{A}_{L}).\label{sec:LMEE9}\\
&\lambda_{0} \text{ is a pole of }(\lambda I-\hat{A}_{L})^{-1} \text{ of order }m.\label{sec:LMEE10}\\
&\lambda_{0} \text{ is a pole of }1/\Delta(\lambda) \text{ of order }m.\label{sec:LMEE11}\\
&\lambda_{0} \text{ is a zero of }\Delta(\lambda) \text{ of order }m.\label{sec:LMEE12}
\end{align}
\end{theorem}

\begin{theorem}\label{sec:LMCC11}
Let H.6 hold and let $\sup_{\Delta(\lambda)=0}\text{Re }\lambda<0$. Then, the zero equilibrium of the system (\ref{sec:LM1}) is globally exponentially asymptotically stable.

\end{theorem}

\begin{theorem}\label{sec:LMCC1} Let H.6 hold. Let $\hat{T}_{L}(t),t \geq 0$ be the strongly continuous linear semigroup in $X$ as in Theorem \ref{sec:LM4} with infinitesimal generator $\hat{A}_{L}$ as in (\ref{sec:LM51}), and let there exist an eigenvalue $\lambda_{0}$ of the linear system (\ref{sec:LM1}) such that $\lambda_{0}$ is real, $\sup_{\Delta(\lambda)=0, \lambda\neq \lambda_{0}} \text{Re }\lambda<\lambda_{0}$, and $\lambda_{0}$ is a simple zero of $\Delta(\lambda)$. Let $\mathcal{P}_{\lambda_{0}}:X\rightarrow X$ be defined by,
\begin{align}\label{sec:LMEE4}
\mathcal{P}_{\lambda_{0}}\phi:=(2\pi i)^{-1}\int_{\Gamma}(\lambda I-\hat{A}_{L})^{-1}d\lambda,\ \text{ for }\phi\in X.
\end{align}
where $\Gamma$ is a positively oriented closed curve in $\mathbb{C}$ enclosing $\lambda_{0}$, but no other point of $\sigma(\hat{A}_{L})$. Then,
\begin{align}\label{sec:LMEE14}
\lim_{t\rightarrow \infty}e^{-\lambda_{0} t}\hat{T}_{L}(t)\phi=\mathcal{P}_{\lambda_{0}}\phi,\ \text{ for all }\phi\in X.
\end{align}
\end{theorem}

In the following theorem (which is an adaption of Theorem 4.10 in \cite{W}, sec. 4.3, pp.188), we derive the projection onto the eigenspace associated with the dominant eigenvalue $\lambda_{0}$ of the linear system (\ref{sec:LM1}).  

\begin{theorem}\label{sec:TLM81}
Let H.6 hold, and let $\hat{T}_{L}(t),t \geq 0$ be the strongly continuous linear semigroup in $X$ as in Theorem \ref{sec:LM4} with infinitesimal generator $\hat{A}_{L}$ as in (\ref{sec:LM51}). If $\mathcal{\tilde{R}}_{0}<1$, then the zero equilibrium is globally exponentially asymptotically stable. If $\mathcal{\tilde{R}}_{0}>1$, then there exists a unique positive real number $\lambda_{0}$ such that 
\begin{align}\label{sec:LMEE13}
\int_{a_{\min}}^{a_{\max}}e^{-\lambda_{0} a}\tilde{\beta}(a)\prod(0,a)da=1.
\end{align}
and for all $\phi\in X$,
\begin{align}\label{sec:LMEE3}
&\lim_{t\rightarrow \infty}e^{-\lambda_{0} t}(\hat{T}_{L}(t)\phi)(a)=e^{-\lambda_{0} a}\prod(0,a)\frac{\int_{a_{\min}}^{a_{\max}}\tilde{\beta}(b)e^{-\lambda_{0} b}[\int_{0}^{b}e^{\lambda_{0} \tau}\prod(\tau,b)\phi(\tau)d\tau]db}{\int_{a_{\min}}^{a_{\max}}\tilde{\beta}(a)ae^{-\lambda_{0} a}\prod(0,a)da}.
\end{align}
where the limit is taken in the norm of $L^{1}$.

\end{theorem}

\begin{proof}
Define $q: \mathbb{C}\rightarrow \mathbb{C}$ by,
\begin{align*}
q(\lambda):=\int_{a_{\min}}^{a_{\max}}e^{-\lambda a}\tilde{\beta}(a)\prod(0,a)da.
\end{align*}
Notice that $q(\lambda)=1$ if and only if $\Delta(\lambda)=0$,  that is, $\lambda$ is an eigenvalue. 

Assume that $q(0)<1$ and there eixsts $\lambda$ such that $q(\lambda)=1$. Then, 
\begin{align}\label{sec:LMEE1}
1=\text{Re }q(\lambda)&=\int_{a_{\min}}^{a_{\max}}e^{-\text{Re }\lambda\: a}\cos(\text{Im }\lambda\: a)\tilde{\beta}(a)\prod(0,a)da\\
\nonumber&\leq \int_{a_{\min}}^{a_{\max}}e^{-\text{Re }\lambda\: a}\tilde{\beta}(a)\prod(0,a)da\\
\nonumber&=q(\text{Re }\lambda).
\end{align}
Since $q$ is continuous and strictly decreasing on $\mathbb{R}$, $\text{Re }\lambda<0$ and there must exist a real number $\lambda_{0}\in[\text{Re }\lambda,0)$ such that $q(\lambda_{0})=1$. Further, (\ref{sec:LMEE1}) shows that for any $\lambda$ such that $q(\lambda)=1$, we must have $\text{Re }\lambda\leq\lambda_{0}$. Thus, $\sup_{\Delta(\lambda)=0}\text{Re }\lambda\leq \lambda_{0}<0$. The global exponential asymptotical stability of the zero equilibrium follows from Theorem \ref{sec:LMCC11}.

Assume that $q(0)>1$. Since $\lim_{\lambda\in\mathbb{R},\lambda\rightarrow \infty}q(\lambda)=0$, there exists a unique positive number $\lambda_{0}$ satisfying $\int_{a_{\min}}^{a_{\max}}e^{-\lambda_{0} a}\tilde{\beta}(a)\prod(0,a)da=1$. We claim that there can be only finitely many $\lambda$ with $\text{Re }\lambda>0$ and $q(\lambda)=1$. Assume to the contrary that there exists an infinite sequence $\left\{z_{k}\right\}$ such that $\text{Re }z_{k}>0$ and $q(z_{k})=1$ for $k\in \mathbb{N}$. Since $\Delta(\lambda)$ is holomorphic for $\text{Re }\lambda>-\infty$, its zeros cannot accumulate in this region (see \cite{Ru}, pp.209). From (\ref{sec:LMEE1}), we see that $0<\text{Re }z_{k}<\lambda_{0}$ for all $k=1,2,\cdots$. Thus, $\lim_{k\rightarrow \infty}|\text{Im }z_{k}|=\infty$. Observe that for each  $k=1,2,\cdots$,
\begin{align}\label{sec:LMEE2}
1&=\int_{a_{\min}}^{a_{\max}}e^{-\text{Re }z_{k}\: a}\cos(\text{Im }z_{k}\: a)\tilde{\beta}(a)\prod(0,a)da\\
\nonumber&< \int_{a_{\min}}^{a_{\max}}\cos(\text{Im }z_{k}\: a)\tilde{\beta}(a)\prod(0,a)da.
\end{align}
But by the Riemann-Lebesgue theorem (see \cite{Ro}, pp.90),
\begin{align*}
\lim_{k\rightarrow \infty}\int_{a_{\min}}^{a_{\max}}\cos(\text{Im }z_{k}\: a)\tilde{\beta}(a)\prod(0,a)da=0.
\end{align*}
which is a contradiction.

Thus, $\sup_{\Delta(\lambda)=0,\lambda\neq \lambda_{0}}\text{Re }\lambda<\lambda_{0}$. To finish the proof it suffices by Theorem \ref{sec:LMCC1} to show that $\lambda_{0}$ is a simple zero of $\Delta(\lambda)$ and $\mathcal{P}_{\lambda_{0}}\phi$ in (\ref{sec:LMEE4}) is given by the right hand side of (\ref{sec:LMEE3}). The claim that $\lambda_{0}$ is a simple zero of $\Delta(\lambda)$ follows from the fact that
\begin{align*}
\Delta'(\lambda_{0})=\int_{a_{\min}}^{a_{\max}}ae^{-\lambda_{0} a}\tilde{\beta}(a)\prod(0,a)da>0.
\end{align*}
Further, the residue of $\frac{1}{\Delta(\lambda)}$ at $\lambda_{0}$ is 
\begin{align*}
\frac{1}{\Delta'(\lambda_{0})}=\frac{1}{\int_{a_{\min}}^{a_{\max}}ae^{-\lambda_{0} a}\tilde{\beta}(a)\prod(0,a)da}.
\end{align*}
(see \cite{Ru}, pp.215). The claim that $\mathcal{P}_{\lambda_{0}}\phi$ in (\ref{sec:LMEE4}) is given by the right hand side of (\ref{sec:LMEE3}) now follows directly from (\ref{sec:LMEE7}) and Theorem \ref{sec:LMCC3}.

\end{proof}







The fact that $\lambda_{0}$ is a dominant eigenvalue follows from Theorem \ref{sec:TLM7}. Using the resolvent formula, we apply the results on irreducible positive semigroups in Banach lattices \cite{W87}  to the system (\ref{sec:LM1}) to obtain the following theorem (we also refer to \cite{T1,Y, A1} for more results on this topic):

\begin{theorem}\label{sec:TLM9}

Let H.6 hold. Let $\hat{T}_{L}(t),t \geq 0$ be the strongly continuous linear semigroup in $X$ as in Theorem \ref{sec:LM4} with infinitesimal generator $\hat{A}_{L}$ as in (\ref{sec:LM51}). Let $\mathcal{\tilde{R}}_{0}>1$. Then $\lambda_{0}>0$ is a simple dominant eigenvalue of $\hat{A}_{L}$, that is, 
\begin{align*} 
\hat{T}_{L}(t)\mathcal{P}_{\lambda_{0}}=\mathcal{P}_{\lambda_{0}}\hat{T}_{L}(t)=e^{\lambda_{0}t}\mathcal{P}_{\lambda_{0}},\ \forall t\geq 0.
\end{align*}
and there exist constants $\epsilon, \eta >0$ such that,
\begin{align*} 
\left\|\hat{T}_{L}(t)(I-\mathcal{P}_{\lambda_{0}})\right\|\leq \eta e^{(\lambda_{0}-\epsilon)t}\left\|I-\mathcal{P}_{\lambda_{0}}\right\|,\ \forall t\geq 0.
\end{align*}
\end{theorem}

\begin{lemma}\label{sec:LMLA}
Let H.6 hold. Let $\tau^{*}>a_{\min}$. Assume that there exists some $\tilde{\delta}>0$ such that $\int_{a_{\min}}^{a_{\max}}\tilde{\beta}(a)$ $l(a,t)da\geq \tilde{\delta}$ for $\forall t\in[0,\tau^{*}]$. Then,
\begin{align*} 
\int_{a_{\min}}^{a_{\max}}\tilde{\beta}(a)l(a,t)da> 0, \ \text{ for }t\geq 0.
\end{align*}
where $l(a,t)$ is a solution of the system (\ref{sec:LM1}).
\end{lemma}

\begin{proof}
Since $\tau^{*}>a_{\min}$, $\int_{a_{\min}}^{\tau^{*}}\tilde{\beta}(a)$ $\prod(0,a)da>0$. Therefore, there exists some $\epsilon\in (0,\tau^{*})$ such that $Q:=\int_{\epsilon}^{\tau^{*}}\tilde{\beta}(a)\prod(0,a)da>0$. For $t\geq\tau^{*}$, we have
\begin{align*}
B(t)&=\int_{a_{\min}}^{t}\tilde{\beta}(a)\prod(0,a)B(t-a)da+\int_{t}^{a_{\max}}\tilde{\beta}(a)\prod(a-t,a)p_{0}(a-t)da\\
&\geq \int_{\epsilon}^{\tau^{*}}\tilde{\beta}(a)\prod(0,a)B(t-a)da\\
&\geq \inf_{r\in[t-\tau^{*},t-\epsilon]}B(r)\int_{\epsilon}^{\tau^{*}}\tilde{\beta}(a)\prod(0,a)da\\
&=\inf_{r\in[t-\tau^{*},t-\epsilon]}B(r)Q.
\end{align*}
Consider $t\in[\tau^{*},\tau^{*}+\epsilon]$, we have
\begin{align*}
B(t)\geq \inf_{r\in[0,\tau^{*}]}B(r)Q\geq \tilde{\delta} Q.
\end{align*}
We recall that $\int_{a_{\min}}^{a_{\max}}\tilde{\beta}(a)l(a,t)da=l(0,t)=B(t)$, it follows that $\int_{a_{\min}}^{a_{\max}}\tilde{\beta}(a)$ $l(a,t)da> 0$ is satisfied for all $t\in [\tau^{*},\tau^{*}+\epsilon]$. The result follows on $[\tau^{*}+n \epsilon,\tau^{*}+(n+1)\epsilon]$ by using induction argument, for $n\in \mathbb{Z}$.
\end{proof}

We define $\tilde{a}:=\sup\left\{a>0: \beta(a;\cdot)> 0 \right\}$.
Also we have $\tilde{a}\geq a_{\min}>0$ and $\tilde{a}=\sup\left\{a>0: \tilde{\beta}(a)> 0 \right\}$ from H.1 and H.6. Let 
\begin{align*} 
M_{0}:=\left\{\phi\in X_{+}: \int_{0}^{\tilde{a}}\phi(a)da>0\right\}, \text{ and } \partial M_{0}=X_{+}-M_{0}.
\end{align*}

\begin{proposition}\label{sec:LMP1}
Let H.6 hold. Then for any $\phi_{0}\in \partial M_{0}$, the solution of the system (\ref{sec:LM1}) satisfies,
\begin{align*} 
\int_{a_{\min}}^{a_{\max}}\tilde{\beta}(a)l(a,t)da=0,\  \text{ for }t\geq 0.
\end{align*}
and therefore $l(a,t)$ is also a solution of the following system,
\begin{align}\label{sec:LMS}
&l_{t}(a,t)+l_{a}(a,t)=-\tilde{\mu}(a)l(a,t),\\ 
\nonumber & 0<a<a_{1},\ t>0,\\
\nonumber &l(0,t)=0,\ t>0,\\
\nonumber &l(a,0)=\phi_{0}(a),\ 0<a<a_{1}.
\end{align}
and satisfies, $\left\|l(\cdot,t)\right\|_{X}\leq e^{-\omega t}\left\|\phi_{0}\right\|_{X}$, for $\forall t\geq 0$, where $\omega=\inf_{a\in[0,a_{1}]}\tilde{\mu}(a)$.
\end{proposition}

\begin{proof}
Let $B(t)$ be the solution of the Volterra integral equation (\ref{sec:LMB}). We observe that since $\phi_{0}\in \partial M_{0}$, we deduce that, for $t\leq a_{\max}$,
\begin{align*}
&\int_{t}^{a_{\max}}\tilde{\beta}(a)\prod(a-t,a)\phi_{0}(a-t)da=\int_{0}^{a_{\max}-t}\tilde{\beta}(a+t)\prod(a,a+t))\phi_{0}(a)da=0.
\end{align*}
If $t> a_{\max}$, it is evident that the above term is not in the equation (\ref{sec:LMB}). Therefore, in both cases, for $\phi_{0}\in \partial M_{0}$ the equation (\ref{sec:LMB}) becomes,
\begin{align*}
B(t)&=\int_{a_{\min}}^{t}\tilde{\beta}(a)\prod(0,a)B(t-a)da.
\end{align*}
which has the unique solution $B(t)=0$ for $t\geq 0$. Then we deduce from  the Volterra integral equation (\ref{sec:LMB1}) that $l(a,t)=0$ for $0\leq a\leq t$. In particular, $l(0,t)=0$, therefore, $l$ is a solution of the system (\ref{sec:LMS}). For $t<a$, we obtain,
\begin{align*}
&l(a,t)=\phi_{0}(a-t)\prod(a-t,a)\leq e^{-\omega t}\phi_{0}(a-t).
\end{align*}
where $\omega=\inf_{a\in[0,a_{1}]}\tilde{\mu}(a)$. Then the conclusion directly follows.
\end{proof}

The following result follows from Theorem \ref{sec:TLM9}.

\begin{proposition}\label{sec:LMPP1}
Let H.6 hold. Let $l(a,t)$ be a solution of the system (\ref{sec:LM1}) with $\phi_{0}\in M_{0}$. If $\mathcal{\tilde{R}}_{0}>1$, then there exists a unique positive real number $\lambda_{0}$ and some $\epsilon^{*}=\epsilon^{*}(\phi_{0})>0$ and $t^{*}=t^{*}(\phi_{0})>0$ such that 
\begin{align*} 
\lim_{t\rightarrow \infty}e^{-\lambda_{0} t}\int_{a_{\min}}^{a_{\max}}\tilde{\beta}(a)l(a,t)da\geq \epsilon^{*}.
\end{align*}
\end{proposition}

\begin{proof}
Let $\phi:=\mathcal{P}_{\lambda_{0}}\phi_{0}$. Since $\mathcal{\tilde{R}}_{0}>1$, by Theorem \ref{sec:TLM81}, then there exists a unique positive real number $\lambda_{0}$ such that $\int_{a_{\min}}^{a_{\max}}e^{-\lambda_{0} a}\tilde{\beta}(a)\prod(0,a)=1$. The corresponding projection on the eigenspace associated with the eigenvalue $\lambda_{0}$ is given by (\ref{sec:LMEE3}). Since $\phi_{0}\in M_{0}$, it directly follows that $\int_{a_{\min}}^{a_{\max}}\tilde{\beta}(b)e^{-\lambda_{0} b}[\int_{0}^{b}e^{\lambda_{0} \tau}$ $\prod(\tau,b)\phi_{0}(\tau)d\tau]db>0$. This implies that $\phi(a)>0$ for $a\in [0,a_{1}]$. Thus, there exists some $\epsilon^{*}=\epsilon^{*}(\phi_{0})>0$ such that $\int_{a_{\min}}^{a_{\max}}\tilde{\beta}(a)$ $\phi(a)da\geq \epsilon^{*}$. Then, applying Theorem \ref{sec:TLM9}, we obtain,
\begin{align*} 
l(\cdot,t)&=\hat{T}_{L}(t)\phi_{0}\\
&=\hat{T}_{L}(t)\mathcal{P}_{\lambda_{0}}\phi_{0}+\hat{T}_{L}(t)(I-\mathcal{P}_{\lambda_{0}})\phi_{0}\\
&=e^{\lambda_{0} t}\mathcal{P}_{\lambda_{0}}\phi_{0}+\hat{T}_{L}(t)(I-\mathcal{P}_{\lambda_{0}})\phi_{0}\\
&=e^{\lambda_{0} t}\phi+\hat{T}_{L}(t)(I-\mathcal{P}_{\lambda_{0}})\phi_{0}.
\end{align*}
It directly follows from Theorem \ref{sec:TLM9} that, 
\begin{align*} 
\lim_{t\rightarrow \infty}\left\|e^{-\lambda_{0} t}\hat{T}_{L}(t)(I-\mathcal{P}_{\lambda_{0}})\right\|_{X}=0.
\end{align*}
Therefore,
\begin{align*} 
\lim_{t\rightarrow \infty}e^{-\lambda_{0} t}\int_{a_{\min}}^{a_{\max}}\tilde{\beta}(a)l(a,t)da=\int_{a_{\min}}^{a_{\max}}\tilde{\beta}(a)\phi(a)da\geq \epsilon^{*}.
\end{align*}
Then the conclusions follow.
\end{proof}

The next proposition shows that if $\mathcal{\tilde{R}}_{0}\leq 1$, the claim of Proposition \ref{sec:LMPP1} still holds.

\begin{proposition}\label{sec:LMPP2}
Let H.6 hold. Let $l(a,t)$ be a solution of the system (\ref{sec:LM1}) with $\phi_{0}\in M_{0}$ and let $\mathcal{\tilde{R}}_{0}\leq 1$. Then there exists some  $t^{*}=t^{*}(\phi_{0})>0$ such that 
\begin{align*} 
\int_{a_{\min}}^{a_{\max}}\tilde{\beta}(a)l(a,t)da>0,\ \text{ for }\forall t\geq t^{*}.
\end{align*}
\end{proposition}

\begin{proof}
Let $\alpha> 0$ and let $\tilde{l}(a,t)=e^{\alpha(t-a)}l(a,t)$ for $t\geq 0$, where $l(a,t)$ is a solution of the system (\ref{sec:LM1}) with $\phi_{0}\in M_{0}$. Then, 
\begin{align*} 
\tilde{l}_{t}(a,t)+\tilde{l}_{a}(a,t)&=e^{\alpha(t-a)}(l_{t}(a,t)+l_{a}(a,t))\\
&=-\tilde{\mu}(a)\tilde{l}(a,t).
\end{align*}
and,
\begin{align*}
\tilde{l}(0,t)&=e^{\alpha t}l(0,t)\\
\nonumber&=e^{\alpha t}\int_{a_{\min}}^{a_{\max}}\tilde{\beta}(a)l(a,t)da\\
\nonumber&=\int_{a_{\min}}^{a_{\max}}\tilde{\beta}(a)e^{\alpha a}\tilde{l}(a,t)da.
\end{align*}
Therefore, $\tilde{l}(a,t)$ satisfies the following system,
\begin{align}\label{sec:LMPP3}
&\tilde{l}_{t}(a,t)+\tilde{l}_{a}(a,t)=-\tilde{\mu}(a)\tilde{l}(a,t),\\ 
\nonumber & 0<a<a_{1},\ t>0,\\
\nonumber &\tilde{l}(0,t)=\int_{a_{\min}}^{a_{\max}}\tilde{\beta}(a)e^{\alpha a}\tilde{l}(a,t)da,\ t>0,\\
\nonumber &\tilde{l}(a,0)=e^{-\alpha a}\phi_{0}(a),\ 0<a<a_{1}.
\end{align}
We choose $\alpha> 0$ sufficiently large such that,
\begin{align*} 
\mathcal{\hat{R}}_{00}:=\int_{a_{\min}}^{a_{\max}}\tilde{\beta}(a)e^{\alpha a}\prod(0,a)da>1.
\end{align*}
Then, the claim follows by applying Proposition \ref{sec:LMPP1} to the system (\ref{sec:LMPP3}). 

\end{proof}



 


We make further assumption on $\beta$ and $\mu_{i}$, $i=0,1$:
\begin{itemize}
\item[H.7.] There exist $\bar{\beta},\underline{\beta}\in L^{\infty}_{+}[a_{\min},a_{\max}]$ such that $\bar{\beta}(a)\geq\beta(a; z)\geq \underline{\beta}(a)\geq 0$ for $(a,z)\in([a_{\min},a_{\max}]\times [0,\infty))$ and $\beta(a;z)$
is non-increasing for $z\geq 0$.
\item[] 
\item[] There exist $\bar{\mu}_{i},\underline{\mu_{i}}\in L^{\infty}_{+}[0,a_{1}]$ such that $\bar{\mu}_{i}(a)\geq\mu_{i}(a; z)\geq \underline{\mu_{i}}(a)\geq 0$ for $(a,z)\in([0,a_{1}]\times [0,\infty))$, for $i=0,1$ and $\mu_{i}(a,z)$
is non-decreasing for $z\geq 0$. $\bar{\mu}_{1}(a)=\underline{\mu}_{1}(a)=0$, for $a>a_{\min}$.
\end{itemize}

\begin{proposition}\label{sec:LMPP9}
Let H.1 and H.6-H.7 hold. 
Then, for the same initial distribution $p_{0}\in X_{+}$, we have $\underline{l}(a,t)\leq p(a,t)\leq \check{l}(a,t)$ for $(a,t)\in[0,a_{1}]\times[0,\infty)$, where, $p(a,t)$ is a solution of the nonlinear system (\ref{sec:1}), $\underline{l}(a,t)$ and $\check{l}(a,t)$ are solutions of linear systems (\ref{sec:LM1G}) and (\ref{sec:LM1U}) for $(a,t)\in[0,a_{1}]\times[0,\infty)$. Furthermore, $\left\|\underline{l}(\cdot,t)\right\|_{X}\leq \left\|p(\cdot,t)\right\|_{X}\leq \left\|\check{l}(\cdot,t)\right\|_{X}$, where,
\begin{align}\label{sec:LM1G}
&\underline{l}_{t}(a,t)+\underline{l}_{a}(a,t)=-[\bar{\mu}_{0}(a)+\bar{\mu}_{1}(a)+ \mu_{2}(a)]\underline{l}(a,t),\\ 
\nonumber & 0<a<a_{1},\ t>0,\\
\nonumber &\underline{l}(0,t)=\int_{a_{\min}}^{a_{\max}}\underline{\beta}(a)\underline{l}(a,t)da,\ t>0,\\
\nonumber &\underline{l}(a,0)=p_{0}(a),\ 0<a<a_{1}.
\end{align}

\begin{align}\label{sec:LM1U}
&\check{l}_{t}(a,t)+\check{l}_{a}(a,t)=-[\underline{\mu_{0}}(a)+\underline{\mu_{1}}(a)+ \mu_{2}(a)]\check{l}(a,t),\\ 
\nonumber & 0<a<a_{1},\ t>0,\\
\nonumber &\check{l}(0,t)=\int_{a_{\min}}^{a_{\max}}\bar{\beta}(a)\check{l}(a,t)da,\ t>0,\\
\nonumber &\check{l}(a,0)=p_{0}(a),\ 0<a<a_{1}.
\end{align}

\end{proposition}

\begin{proof}
By Theorem \ref{sec:LM4}, we obtain solutions of systems (\ref{sec:LM1G}), (\ref{sec:LM1U}) form strongly continuous linear semigroups, denoted by $T_{L}(t),t\geq 0$ and $T_{U}(t),t\geq 0$. The Volterra integral formula for linear systems (\ref{sec:LM1G}) and (\ref{sec:LM1U}) can be derived directly from (\ref{sec:LMB1})-(\ref{sec:LMB}), where $\underline{l}(a,t):=(T_{L}(t)p_{0})(a)$ and $\check{l}(a,t):=(T_{U}(t)p_{0})(a)$, for $(a,t)\in[0,a_{1}]\times[0,\infty)$. We apply the method of characteristics to nonlinear system (\ref{sec:1}) to obtain the following integral equations, (for more details we refer to \cite{W}, sec 1.3, pp.11.)
\begin{align}\label{sec:VT1}
p(a,t)=\begin{cases}
p(0,t-a)e^{-\int_{t-a}^{t}(\mu_{0}(a+s-t,\eta_{0}(Q_{0}(s)))+\mu_{1}(a+s-t,\eta_{1}(Q_{1}(s)))+\mu_{2}(a+s-t))ds}; & \\
\text{ a.e. } \ \ a \in (0,t)\cap[0,a_{1}];\\
p_{0}(a-t)e^{-\int_{0}^{t}(\mu_{0}(a+s-t,\eta_{0}(Q_{0}(s)))+\mu_{1}(a+s-t,\eta_{1}(Q_{1}(s)))+\mu_{2}(a+s-t))ds}; &\\
\text{ a.e. } \ \ a\in[t,a_{1}].
\end{cases}
\end{align}
Define $\tilde{B}(t)=p(0,t)$ and substitute the formula for $p(a,t)$ into $Q_{i}(t)$, $i=0,1$ and the boundary condition to obtain,
\begin{align*}
Q_{0}(t)&=\int_{0}^{t}\omega_{0}(a)\tilde{B}(t-a)e^{-\int_{t-a}^{t}(\mu_{0}(a+s-t,\eta_{0}(Q_{0}(s)))+\mu_{1}(a+s-t,\eta_{1}(Q_{1}(s)))+\mu_{2}(a+s-t))ds}da\\
&+\int_{t}^{a_{1}}\omega_{0}(a)p_{0}(a-t)e^{-\int_{0}^{t}(\mu_{0}(a+s-t,\eta_{0}(Q_{0}(s)))+\mu_{1}(a+s-t,\eta_{1}(Q_{1}(s)))+\mu_{2}(a+s-t))ds}da.\\
Q_{1}(t)&=\int_{a_{\max}}^{t}\omega_{1}(a)\tilde{B}(t-a)e^{-\int_{t-a}^{t}(\mu_{0}(a+s-t,\eta_{0}(Q_{0}(s)))+\mu_{1}(a+s-t,\eta_{1}(Q_{1}(s)))+\mu_{2}(a+s-t))ds}da\\
&+\int_{t}^{a_{1}}\omega_{1}(a)p_{0}(a-t)e^{-\int_{0}^{t}(\mu_{0}(a+s-t,\eta_{0}(Q_{0}(s)))+\mu_{1}(a+s-t,\eta_{1}(Q_{1}(s)))+\mu_{2}(a+s-t))ds}da.
\end{align*}
\begin{align*}
\tilde{B}(t)&=\int_{a_{\min}}^{t}\beta(a;\eta_{2}(Q_{0}(t)))\tilde{B}(t-a)\\
&\times e^{-\int_{t-a}^{t}(\mu_{0}(a+s-t,\eta_{0}(Q_{0}(s)))+\mu_{1}(a+s-t,\eta_{1}(Q_{1}(s)))+\mu_{2}(a+s-t))ds}da\\
&+\int_{t}^{a_{\max}}\beta(a;\eta_{2}(Q_{0}(t)))p_{0}(a-t)\\
&\times e^{-\int_{0}^{t}(\mu_{0}(a+s-t,\eta_{0}(Q_{0}(s)))+\mu_{1}(a+s-t,\eta_{1}(Q_{1}(s)))+\mu_{2}(a+s-t))ds}da.
\end{align*}
The claim follows from the assumptions.
\end{proof}

The following invariance property of the strongly continuous linear semigroup $T_{L}(t),\ t\geq 0$ corresponding to the system (\ref{sec:LM1G})
is a consequence of Proposition \ref{sec:LMPP1} and Proposition \ref{sec:LMPP2}. 

\begin{proposition}\label{sec:LMPP4}
Let H.6-H.7 hold. Then,
\begin{align*} 
T_{L}(t) (M_{0})\subseteq M_{0}, \ \text{ for } t\geq 0.
\end{align*}
\end{proposition}

\begin{proof}
Let $\phi_{0}\in M_{0}$. Assume that there exists $t_{1}>0$ such that $T_{L}(t_{1})\phi_{0}$ $\in \partial M_{0}$, then we have $\int_{a_{\min}}^{a_{\max}}\underline{\beta}(a)$ $l(a,t)da=0$ for $t= t_{1}$. Further, since $T_{L}(t_{1})\phi_{0}$ $\in \partial M_{0}$, then by Proposition \ref{sec:LMP1}, we have $\int_{a_{\min}}^{a_{\max}}\underline{\beta}(a)$ $l(a,t)da=0$ for $t\geq t_{1}$, where $l(a,t)=(T_{L}(t)\phi_{0})(a)$, for $(a,t)\in[0,a_{1}]\times[0,\infty)$, and this contradicts Proposition \ref{sec:LMPP2}. Therefore, there does not exist such a $t_{1}$. 
\end{proof}

\begin{proposition}\label{sec:LMPP5}
Let H.6-H.7 hold. Let $U(t), \ t\geq 0$ be the strongly continuous nonlinear semigroup in $X_{+}$ as in Theorem \ref{sec:T3.3}. 
Then, $M_{0}$ is positively invariant under $U(t), t\geq 0$. Furthermore, for every $\phi_{0}\in M_{0}$, there exists $t^{*}>0$ such that
\begin{align*} 
\int_{a_{\min}}^{a_{\max}}\beta(a; \eta_{2} (Q_{0}(t)))p(a,t)da>0,\ \text{ for }\forall t\geq t^{*}.
\end{align*}
\end{proposition}

\begin{proof}
Let $T_{L}(t),\:t\geq 0$ be the strongly continuous linear semigroup corresponding to the system (\ref{sec:LM1G}) with initial value $\phi_{0}$. By Proposition \ref{sec:LMPP9}, $(U(t)\phi_{0})(a)\geq (T_{L}(t)\phi_{0})(a)$ for $(a,t)\in[0,a_{1}]\times[0,\infty)$. Let $\phi_{0}\in M_{0}$. Assume that there exists $t_{1}>0$ such that $U(t_{1})\phi_{0}$ $\in \partial M_{0}$, then we have $\int_{a_{\min}}^{a_{\max}}\beta(a; \eta_{2} (Q_{0}(t)))p(a,t)da=0$ for $t= t_{1}$. This by comparison argument, implies that $T_{L}(t_{1})\phi_{0}$ $\in \partial M_{0}$, which contradicts Proposition \ref{sec:LMPP4}. Therefore, there does not exist such a $t_{1}$. Finally, applying Proposition \ref{sec:LMPP2}, we obtain $\int_{a_{\min}}^{a_{\max}}\beta(a; \eta_{2} (Q_{0}(t)))p(a,t)da\geq \int_{a_{\min}}^{a_{\max}}\underline{\beta}(a)i(a,t)da>0,\ \text{ for }\forall t\geq t^{*}$.

\end{proof}

\begin{proposition}\label{sec:LMPP8}
Let H.6-H.7 hold. Let $U(t),t\geq 0$ be the strongly continuous nonlinear semigroup in $X_{+}$ as in Theorem \ref{sec:T3.3}. Then, $\partial M_{0}$ is positively invariant under $U(t),t\geq 0$. Furthermore, the trivial equilibrium is globally exponentially stable for $U(t), t\geq 0$ restricted to $\partial M_{0}$.
\end{proposition}

\begin{proof}
Let $\phi_{0}\in \partial M_{0}$ and let $T_{U}(t),t\geq 0$ be the strongly continuous linear semigroup corresponding to the system (\ref{sec:LM1U}).
By Proposition \ref{sec:LMPP9}, we obtain $(U(t)\phi_{0})(a)\leq (T_{U}(t)\phi_{0})(a)$ for $(a,t)\in[0,a_{1}]\times[0,\infty)$.
Then by Proposition \ref{sec:LMP1}, we deduce that 
\begin{align*}
0\leq \int_{a_{\min}}^{a_{\max}}\beta(a; \eta_{2} (Q_{0}(t)))(U(t)\phi_{0})(a)da\leq \int_{a_{\min}}^{a_{\max}}\bar{\beta}(a)(T_{U}(t)\phi_{0})(a)da=0.\\
\end{align*}
and the result follows.

\end{proof}

\section{Uniform Persistence}

\begin{theorem}[Uniform Persistence]\label{sec:LMTUP}
Let H.1, H.6-H.7 hold. 
If $\mathcal{R}_{0}>1$, the strongly continuous nonlinear semigroup $U(t),t\geq 0$, in $X_{+}$ as in Theorem \ref{sec:T3.3} is uniformly persistent with respect to the pair $(\partial M_{0}, M_{0})$, namely, there exists $\epsilon> 0$ such that,
\begin{align*}
\liminf_{t\rightarrow \infty}\left\|U(t)\phi\right\|_{X}\geq\epsilon, \ \text{ for every }\phi\in M_{0}.
\end{align*}

\end{theorem}

\begin{proof}
Let $\hat{\omega}=\max\left\{\left\|\omega_{i}\right\|_{L^{\infty}},i=0,1,2\right\}$. To apply Theorem 4.1 in \cite{Ha} for the claim, it is sufficient to show that there exists $\epsilon> 0$ such that for each $\phi_{0}\in M_{0}$, there exists $t_{0}\geq 0$ such that $\left\|U(t_{0})\phi_{0}\right\|_{X}\geq \epsilon$. Let $\mathcal{\tilde{R}}_{0}(\bar{Q}_{0},\bar{Q}_{1})=\int_{a_{\min}}^{a_{\max}}\beta(a;$ $ \eta_{2} (\bar{Q}_{0}))\prod(0,a; \eta_{0}(\bar{Q}_{0}),\eta_{1}(\bar{Q}_{1}))da$, for $\bar{Q_{i}}\geq 0$ and $i=0,1$. We observe from (\ref{sec:4.1}) that $\mathcal{\tilde{R}}_{0}(\bar{Q}_{0},\bar{Q}_{1})$ is continuous for $(\bar{Q}_{0},\bar{Q}_{1})\in [0,\infty)\times[0, \infty)$ with $\mathcal{\tilde{R}}_{0}(0,0)=\mathcal{R}_{0}>1$. 
Therefore, 
there exists some neighbourhood of $(0,0)$ in the right half plane of $R^{2}$, denoted by $\mathcal{O}:=[0,\bar{\delta})\times[0,\bar{\delta})$ such that for any $(\bar{Q}_{0},\bar{Q}_{1})\in \mathcal{O}$, we have $\int_{a_{\min}}^{a_{\max}}\beta(a; \eta_{2} (\bar{Q}_{0}))\prod(0,a; \eta_{0} (\bar{Q}_{0}),\eta_{1} (\bar{Q}_{1}))$ $da> 1$. We argue by contradiction, assume for $0<\epsilon=\min\left\{\frac{\bar{\delta}}{2\hat{\omega}},\frac{\bar{\delta}}{2}\right\}<\bar{\delta}$ fixed, there exists some $\phi_{0}\in M_{0}$ such that $\left\|U(t)\phi_{0}\right\|_{X}\leq \epsilon$, $\forall t\geq 0$. Then, we consider the following linear system, and let $ i(a,t)$ be a solution:
\begin{align*}
&i_{t}(a,t)+i_{a}(a,t)=-[\mu_{0}(a, \eta_{0} (\frac{\bar{\delta}}{2}))+\mu_{1}(a,\eta_{1}(\frac{\bar{\delta}}{2}))+ \mu_{2}(a)]i(a,t),\\ 
\nonumber & 0<a<a_{1},\ t>0,\\
\nonumber &i(0,t)=\int_{a_{\min}}^{a_{\max}}\beta(a; \eta_{2} (\frac{\bar{\delta}}{2}))i(a,t)da,\ t>0,\\
\nonumber &i(a,0)=\phi_{0}(a),\ 0<a<a_{1}.
\end{align*}
This, by Proposition \ref{sec:LMPP9}, implies that $p(a,t)=(U(t)\phi_{0})(a)\geq i(a,t)$ for $(a,t)\in[0,a_{1}]\times[0,\infty)$. Since $\mathcal{R}_{0}>1$, by Theorem \ref{sec:TLM9}, we deduce that $\mathcal{P}_{\lambda_{0}}i(\cdot,t)=e^{\lambda_{0} t}\mathcal{P}_{\lambda_{0}}\phi_{0}$ for some $\lambda_{0}> 0$, where $\mathcal{P}_{\lambda_{0}}$ is the projection on the eigenspace corresponding to the eigenvalue $\lambda_{0}$. Further, by Theorem \ref{sec:TLM81}, $\left\|\mathcal{P}_{\lambda_{0}}\phi_{0}\right\|_{X}>0$ for $\phi_{0}\in M_{0}$. It then follows that $\lim_{t\rightarrow \infty}\left\|\mathcal{P}_{\lambda_{0}}i(\cdot,t)\right\|_{X}=\infty$. Therefore, $\lim_{t\rightarrow \infty}\left\|p(\cdot,t)\right\|_{X}=\infty$, which contradicts with $\left\|p(\cdot,t)\right\|_{X}\leq \epsilon,\ \forall t\geq 0$. Therefore, the stable manifold of the trivial equilibrium does not intersect $M_{0}$. Furthermore, by Corollary \ref{sec:LMCOR}, the strongly continuous nonlinear semigroup $U(t), \ t\geq 0$ in $X_{+}$ as in Theorem \ref{sec:T3.3} is point dissipative and asymptotically smooth and the trajectory of a bounded set is bounded. The trivial equilibrium is global stable in $\partial M_{0}$. Therefore, Theorem 4.2 in \cite{Ha} implies the uniform persistence of $U(t), \ t\geq 0$.

\end{proof}

We use the results of \cite{ST, Go,Ha1,Ha,Ma,M04, MZ,DA, Lak}, to obtain the following theorem.

\begin{theorem}\label{sec:ga}
Let H.1 hold. Let $U(t), \ t\geq 0$ in $X_{+}$ be the strongly continuous nonlinear semigroup as in Theorem \ref{sec:T3.3}. Assume that $U(t), \ t\geq 0$ is bounded dissipative, asymptotically smooth and uniformly persistent with respect to ($M_{0}$, $\partial M_{0}$). There exists a global attractor $\mathcal{A}_{0}\in M_{0}$ under $U(t), \ t\geq 0$, which is a compact set, satisfying,
\begin{itemize}
\item[(i)] $\mathcal{A}_{0}$ is invariant under the semigroup $U(t), \ t\geq 0$, namely, $U(t)\mathcal{A}_{0}=\mathcal{A}_{0}$, for $t\geq 0$;
\item[(ii)] $\mathcal{A}_{0}$ attracts the bounded sets of $M_{0}$ under $U(t), \ t\geq 0$ that is, for every bounded set $C\in M_{0}$, $\lim_{t\rightarrow \infty}$ $\hat{\delta}(U(t)C,\mathcal{A}_{0})=0$, where, the semi-distance $\hat{\delta}(C,\mathcal{A}_{0}):=\sup_{x\in C}\inf_{y\in \mathcal{A}_{0}}\left\|x-y\right\|$.
\end{itemize} 
Moreover, the subset $\mathcal{A}_{0}$ is locally asymptotically stable.

\end{theorem}

\begin{proposition}\label{sec:LMPP7}
Let H.1, H.6-H.7 hold. Let $U(t),t\geq 0$, be the strongly continuous nonlinear semigroup in $X_{+}$ as in Theorem \ref{sec:T3.3}. 
There exists some $\delta>0$ such that for every $\phi_{0}\in \mathcal{A}_{0}$,
\begin{align*} 
\int_{a_{\min}}^{a_{\max}}\underline{\beta}(a)\phi_{0}(a)da\geq \delta.
\end{align*}
where $\mathcal{A}_{0}$ is the global attractor given in Theorem \ref{sec:ga}. 
\end{proposition}

\begin{proof}

By Corollary \ref{sec:LMCOR} and Theorem \ref{sec:LMTUP}, the strongly continuous nonlinear semigroup $U(t), \ t\geq 0$ in $X_{+}$ as in Theorem \ref{sec:T3.3} is point dissipative, asymptotically smooth and uniformly
persistent with respect to $(M_{0}, \partial M_{0})$. Therefore, by Theorem \ref{sec:ga}, there exists a global attractor $\mathcal{A}_{0}\in M_{0}$, for $U(t), t\geq 0$, which is invariant under $U(t), \:t\geq 0$. 
Let $T_{L}(t), t\geq 0$ be the strongly continuous linear semigroup corresponding to the system (\ref{sec:LM1G}) with $i(a,t)=(T_{L}(t)\phi_{0})(a)$ for $(a,t)\in[0,a_{1}]\times[0,\infty)$. By Proposition \ref{sec:LMPP9}, we obtain $(U(t)\phi_{0})(a)\geq (T_{L}(t)\phi_{0})(a)$ for $(a,t)\in[0,a_{1}]\times[0,\infty)$.
Further, by Proposition \ref{sec:LMPP2}, there exists $t^{*}=t^{*}(\phi_{0})>0$ such that for $t>t^{*}$, we have 
\begin{align*} 
\int_{a_{\min}}^{a_{\max}}\underline{\beta}(a)(U(t)\phi_{0})(a)da\geq \int_{a_{\min}}^{a_{\max}}\underline{\beta}(a)(T_{L}(t)\phi_{0})(a)da>0,\ \text{ for }t\geq t^{*}.
\end{align*}

For each fixed $\phi\in X_{+}$, $U(t)\phi$ is continuous in $t$ and the mapping $\phi\rightarrow U(t)\phi$ is continuous for $\phi\in X_{+}$, $t\geq 0$, and also the mapping 
$\phi\rightarrow\int_{a_{\min}}^{a_{\max}}\underline{\beta}(a)\phi(a)da$ is continuous in $L^{1}$ norm. Therefore, for any $\phi_{0}\in \mathcal{A}_{0}$, there eixsts $r=r(\phi_{0})>0$ and $\tau^{*}>0$ such that for $\tilde{\phi}_{0}\in \mathcal{A}_{0}$ satisfying $\left\|\tilde{\phi}_{0}-\phi_{0}\right\|_{X}\leq r(\phi_{0})$, we have 
$\int_{a_{\min}}^{a_{\max}}\underline{\beta}(a)(U(t)\tilde{\phi}_{0})(a)da>0$, for $t\in[t^{*}(\phi_{0}),t^{*}(\phi_{0})+\tau^{*}]$. Then applying Lemma \ref{sec:LMLA}, we deduce that 
$\int_{a_{\min}}^{a_{\max}}\underline{\beta}(a)$ $(U(t)\tilde{\phi})(a)da>0$, for $t\geq t^{*}(\phi_{0})$. Therefore, we obtain an open cover for $\mathcal{A}_{0}$, that is, $\mathcal{A}_{0}\subset \cup_{\phi_{0}\in\mathcal{A}_{0}}N(\phi_{0})$, where $N(\phi_{0}):=\left\{\phi\in\mathcal{A}_{0}: \left\|\phi-\phi_{0}\right\|_{X}<r(\phi_{0})\right\}$. Then, we deduce from the compactness of $\mathcal{A}_{0}$ that there exists a finite cover such that $\mathcal{A}_{0}\subset \cup_{j=1}^{m}N(\phi_{j})$ for some $m>0$ and for each $\phi\in N(\phi_{j})$, we have $\int_{a_{\min}}^{a_{\max}}\underline{\beta}(a)$ $(U(t)\phi)(a)da>0$, for $t\geq t^{*}(\phi_{j})$, where $j=1,2,\cdots,m$. Let $\hat{t}=\max_{j=1}^{m}\left\{t_{j}\right\}$. For every $\phi_{0}\in \mathcal{A}_{0}$ and $t>\hat{t}$, 
$\int_{a_{\min}}^{a_{\max}}\underline{\beta}(a)$ 
$(U(t)\phi_{0})(a)da>0$. Using the fact that $U(t)\mathcal{A}_{0}=\mathcal{A}_{0}$, $\forall t\geq 0$, we have for any initial value $\phi_{0}\in\mathcal{A}_{0}$, there exists a complete orbit $\left\{U(t):t\in\mathbb{R}\right\}$ through $\phi_{0}$ in $\mathcal{A}_{0}$. Therefore, for any $\phi_{0}\in \mathcal{A}_{0}$, we have $\phi_{0}=U(t)(U(-t)\phi_{0})$ with $U(-t)\phi_{0}\in \mathcal{A}_{0}$. This implies for any $t\geq 0$, $\int_{a_{\min}}^{a_{\max}}\underline{\beta}(a)$ $(U(t)\phi_{0})(a)da=\int_{a_{\min}}^{a_{\max}}\underline{\beta}(a)$ $(U(t+t_{1})U(-t_{1})\phi_{0})(a)da>0$, for $t_{1}>\hat{t}$. 
It then follows that for every $\phi_{0}\in \mathcal{A}_{0}$,
$\int_{a_{\min}}^{a_{\max}}\underline{\beta}(a)\phi_{0}(a)da>0$. Now we define the functional $\hat{\mathcal{F}}: \mathcal{A}_{0}\rightarrow [0,\infty)$ by $\hat{\mathcal{F}}(\phi):=\int_{a_{\min}}^{a_{\max}}\underline{\beta}(a)\phi(a)da$ for $\forall\phi \in\mathcal{A}_{0}$. Then using the continuity for $\phi\rightarrow\int_{a_{\min}}^{a_{\max}}\underline{\beta}(a)\phi(a)da$ again, we have $\hat{\mathcal{F}}$ is continuous on $\mathcal{A}_{0}$. Let $\delta:=\inf_{\phi\in \mathcal{A}_{0}}\hat{\mathcal{F}}(\phi)\geq 0$. We claim that $\delta >0$. If not, 
by the definition of infimum there exists a sequence $\left\{\phi_{n}\right\}_{n\in\mathbb{Z}}\subset \mathcal{A}_{0}$ such that $\lim_{n\rightarrow \infty}\hat{\mathcal{F}}(\phi_{n})=0$. Then we use the compactness of $\mathcal{A}_{0}$ to choose a convergent subsequence $\left\{\phi_{n_{j}}\right\}$ for $j=1,2,\cdots$ such that $\lim_{j\rightarrow \infty}\phi_{n_{j}}=\bar{\phi}$, where $\bar{\phi}\in \mathcal{A}_{0}$, . This implies $\hat{\mathcal{F}}(\bar{\phi})=0$ which contradicts $\int_{a_{\min}}^{a_{\max}}\underline{\beta}(a)\bar{\phi}(a)da>0$ since $\bar{\phi}\in \mathcal{A}_{0}$. Therefore, the claim follows.

\end{proof}

\section{Global stability analysis}

\subsection{Global stability of the trivial equilibrium}

\begin{theorem}\label{sec:LMLL}
Let H.1 and H.6-H.7 hold. 
If $\mathcal{R}_{0}< 1$. Then the trivial equilibrium is global asymptotically stable.

\end{theorem}

\begin{proof}
Let $U(t),t\geq 0$, be the strongly continuous nonlinear semigroup in $X_{+}$ as in Theorem \ref{sec:T3.3}. Let $p(a,t)=(U(t)\phi_{0})(a)$, for $\phi_{0}\in X_{+}$. Define the function $V: X_{+}\rightarrow \mathbb{R}$ by
\begin{align}\label{sec:LMF1}
V[\Psi]:=\int_{a_{\min}}^{a_{\max}}\beta(a; \eta_{2} (0))\int_{0}^{a}e^{-\int_{s}^{a}(\mu_{0}(\hat{a}, \eta_{0} (0))+\mu_{1}(\hat{a},\eta_{1}(0))+ \mu_{2}(\hat{a}))d\hat{a}}\Psi(s)dsda.
\end{align}
Then
\begin{align*}
\frac{d}{dt}V[p(\cdot,t)]&=\int_{a_{\min}}^{a_{\max}}\beta(a; \eta_{2} (0))\int_{0}^{a}e^{-\int_{s}^{a}(\mu_{0}(\hat{a}, \eta_{0} (0))+\mu_{1}(\hat{a},\eta_{1}(0))+ \mu_{2}(\hat{a}))d\hat{a}}\frac{\partial }{\partial t}p(s,t)dsda\\
&=\int_{a_{\min}}^{a_{\max}}\beta(a; \eta_{2} (0))\int_{0}^{a}e^{-\int_{s}^{a}(\mu_{0}(\hat{a}, \eta_{0} (0))+\mu_{1}(\hat{a},\eta_{1}(0))+ \mu_{2}(\hat{a}))d\hat{a}}[-\frac{\partial }{\partial s}p(s,t)\\
&-(\mu_{0}(s, \eta_{0} (Q_{0}(t)))+\mu_{1}(s,\eta_{1}(Q_{1}(t)))+ \mu_{2}(s))p(s,t)]dsda\\
&\leq\int_{a_{\min}}^{a_{\max}}\beta(a; \eta_{2} (0))\int_{0}^{a}e^{-\int_{s}^{a}(\mu_{0}(\hat{a}, \eta_{0} (0))+\mu_{1}(\hat{a},\eta_{1}(0))+ \mu_{2}(\hat{a}))d\hat{a}}[-\frac{\partial }{\partial s}p(s,t)\\
&-(\mu_{0}(s, \eta_{0} (0))+\mu_{1}(s,\eta_{1}(0))+ \mu_{2}(s))p(s,t)]dsda\\
&=-\int_{a_{\min}}^{a_{\max}}\beta(a; \eta_{2} (0))[e^{-\int_{s}^{a}(\mu_{0}(\hat{a}, \eta_{0} (0))+\mu_{1}(\hat{a},\eta_{1}(0))+ \mu_{2}(\hat{a}))d\hat{a}}p(s,t)]_{s=0}^{a}da\\
&=\int_{a_{\min}}^{a_{\max}}\beta(a; \eta_{2} (0))e^{-\int_{0}^{a}(\mu_{0}(\hat{a}, \eta_{0} (0))+\mu_{1}(\hat{a},\eta_{1}(0))+ \mu_{2}(\hat{a}))d\hat{a}}dap(0,t)\\
&-\int_{a_{\min}}^{a_{\max}}\beta(a; \eta_{2} (0))p(a,t)da.
\end{align*}
Applying the boundary condition $p(0,t)=\int_{a_{\min}}^{a_{\max}}\beta(a; \eta_{2} (Q_{0}(t)))p(a,t)da\leq\int_{a_{\min}}^{a_{\max}}$ $\beta(a; \eta_{2} (0))p(a,t)da,\ t\geq 0$ to obtain,
\begin{align*}
\frac{d}{dt}V[p(\cdot,t)]&=(\int_{a_{\min}}^{a_{\max}}\beta(a; \eta_{2} (0))\prod(0,a;\eta_{0} (0),\eta_{1}(0))da-1)\\
&\times\int_{a_{\min}}^{a_{\max}}\beta(a; \eta_{2} (0))p(a,t)da.
\end{align*}
Since $\mathcal{R}_{0}< 1$, we deduce that,
\begin{align*}
\frac{d}{dt}V[p(\cdot,t)]\leq 0,\ \text{ for }t\geq 0.
\end{align*}
And $\frac{d}{dt}V[p(\cdot,t)]=0$ if and only if $\int_{a_{\min}}^{a_{\max}}\beta(a; \eta_{2} (0))p(a,t)da=0$. Therefore, $\phi_{0}\in \partial M_{0}$, and by Proposition \ref{sec:LMPP8}, the trivial equilibrium is globally asymptotically stable in $\partial M_{0}$. Then, we deduce that the omega-limit set of a trajectory only consists of the trivial equilibrium. Finally, we observe that $V[p(\cdot,t)]$ is decreasing for $t\geq 0$ and is zero at the omega-limit set. Then we deduce that it is identically zero, and the result follows.

\end{proof}

\begin{remark}
Another way to show the global stability of the trivial equilibrium is by applying the invariance principle, which is given by the following Theorem from (\cite{W}, section 4.2, pp.158):

\begin{theorem} Let H.1 and H.7 hold. Let $U(t), \ t\geq 0$ be the strongly continuous nonlinear semigroup in $X_{+}$ as in Theorem \ref{sec:T3.3}. 
Let $\beta_{0}(a)=\beta(a,\eta_{2}(0))$ for $a\in [a_{\min},a_{\max}]$ and $\check{\mu}_{0}: [0,a_{1}]\rightarrow [0,\infty)$ be $\check{\mu}_{0}(a)=\mu_{0}(a,\eta_{0}(0))+\mu_{1}(a,\eta_{1}(0))+\mu_{2}(a)$. Let $\beta_{0}, \ \check{\mu}_{0}$ satisfy the following conditions:
\begin{itemize}
\item[(i)]  $\check{\mu}_{0}$ is nondecreasing on $[0,a_{1}]$;
\item[(ii)] $\int_{a_{\min}}^{a_{\max}}e^{\omega a}\beta_{0}(a)e^{-\int_{0}^{a}\check{\mu}_{0}(b)db}da=1$ for some $\omega>0$.
\end{itemize}
Then, $\lim_{t \rightarrow \infty}U(t)\phi=0$ in $L^{1}(0,\infty;\mathbb{R})$ for all $\phi\in X_{+}$.
\end{theorem}

The proof of this theorem follows from (\cite{W}, section 4.2, pp.158) by defining $V: X_{+} \rightarrow \mathbb{R}$,
$V(\phi):=\int_{0}^{a_{1}}\phi(a)q(a)da\ \ \phi\in X_{+}$, where,
$q(a):=\exp[-\omega a+\int_{0}^{a}\check{\mu}_{0}(b)db]\times[1-\int_{0}^{a}e^{\omega b}\beta_{0}(b)\exp[-\int_{0}^{b}\check{\mu}_{0}(\tau)d\tau]db],\ a\in[0,a_{1}]$.
\end{remark}

\subsection{Global stability of the nontrivial equilibrium}
In this section, we use the following Lyapunov functional to show that, if $\mathcal{R}_{0}>1$, solutions of the system (\ref{sec:1}) converge to the nontrivial equilibrium $\hat{\phi}$ (\ref{sec:4.12}). 
We define $V: \mathcal{A}_{0} \rightarrow\mathbb{R}$ by:
\begin{align*} 
V(\phi)&:=\int^{a_{1}}_{0}(|\phi(a)-\hat{\phi}(a)|-|\hat{\phi}(a)\log\frac{\phi(a)}{\hat{\phi}(a)}|)da,\ \text{ for }\phi\in \mathcal{A}_{0}.
\end{align*}
where $\hat{\phi}$ (\ref{sec:4.12}) is the positive equilibrium of the system (\ref{sec:1}) as in Theorem 4.1.

\begin{proposition}\label{sec:prop5.11} 
Let H.1-H.7 hold. 
If $ \mathcal{R}_{0}>1$, $V$ is a Liapunov functional on $\mathcal{A}_{0}$.
\end{proposition}

\begin{proof}

Our goal is to show that $V$ is well defined on the global attractor $\mathcal{A}_{0}$. Obviously, $V$ is not well defined on $M_{0}$ because of the function $\log$ under the integral. Let $\phi_{0}\in \mathcal{A}_{0}$. We use the Volterra formulation (\ref{sec:VT1}) and a comparison argument by   Proposition \ref{sec:LMPP9} and Proposition \ref{sec:LMPP7}, to obtain $(U(t)\phi_{0})(a)=B(t-a)\geq \int_{a_{\min}}^{a_{\max}}\underline{\beta}(l)(U(t-a)\phi_{0})(l)dl\geq \delta$, for $(a,t)\in[0,a_{1}]\times[0,\infty)$ and $t> a$. For $0\leq t\leq a\leq a_{1}$, we have $(U(t)\phi_{0})(a)=(U(t+t_{1})U(-t_{1})\phi_{0})(a)\geq \delta$, where $t_{1}>a_{1}$. Therefore, $(U(t)\phi_{0})(a)\geq \delta$ for all $(a,t)\in[0,a_{1}]\times[0,\infty)$.
On the other hand, we use the fact that $\mathcal{A}_{0}$ is compact and H.6 to derive that there exists some $M>0$ such that for every $\phi_{0}\in \mathcal{A}_{0}$, we have
$\int_{a_{\min}}^{a_{\max}}\bar{\beta}(a)\phi_{0}(a)da\leq M$.  Then, we deduce that $(U(t)\phi_{0})(a)=B(t-a)\leq \int_{a_{\min}}^{a_{\max}}\bar{\beta}(l)(U(t-a)\phi_{0})(l)dl\leq M$, for $(a,t)\in[0,a_{1}]\times[0,\infty)$ and $t> a$. For $0\leq t\leq a\leq a_{1}$, we have $(U(t)\phi_{0})(a)=(U(t+t_{1})U(-t_{1})\phi_{0})(a)\leq M$, where $t_{1}>a_{1}$. Therefore, $(U(t)\phi_{0})(a)\leq M$ for all $(a,t)\in[0,a_{1}]\times[0,\infty)$. It then directly follows that $\delta_{1}\leq\frac{(U(t)\phi)(a)}{\hat{\phi}(a)}\leq \delta_{2}$, for $(a,t)\in[0,a_{1}]\times[0,\infty)$, where $\delta_{1}=\frac{\delta}{\sup_{a\in[0,a_{1}]}\hat{\phi}(a)}$ and $\delta_{2}=\frac{M}{\inf_{a\in[0,a_{1}]}\hat{\phi}(a)}$. It is easily seen from (\ref{sec:4.12}) that $\inf_{a\in[0,a_{1}]}\hat{\phi}(a)>0$ which is due to the fact that $\hat{\phi}\in C^{1}[0,a_{1}]$ and $\hat{\phi}(a)>0$ for $a\in[0,a_{1}]$. We then deduce that $0\leq |\log\frac{(U(t)\phi)(a)}{\hat{\phi}(a)}|\leq \max\left\{|\log\delta_{1}|,|\log\delta_{2}|\right\}$. Therefore, $0 \leq\int^{a_{1}}_{0}\hat{\phi}(a)|\log\frac{(U(t)\phi)(a)}{\hat{\phi}(a)}|da\leq \max\left\{|\log\delta_{1}|,|\log\delta_{2}|\right\}\hat{\left\|\phi\right\|}_{X}$.

Then, for $\phi\in \mathcal{A}_{0}$ and $U(t)\phi\in \mathcal{A}_{0}$, $t\geq 0$, we have
\begin{align*} 
\dot{V}(\phi)&=\lim_{t\rightarrow 0^{+}}\frac{V(U(t)\phi)-V(\phi)}{t}\\
&=\lim_{t\rightarrow 0^{+}}\frac{1}{t}[\int^{a_{1}}_{0}(|(U(t)\phi)(a)-\hat{\phi}(a)|-|\hat{\phi}(a)\log\frac{(U(t)\phi)(a)}{\hat{\phi}(a)}|)da\\
&-\int^{a_{1}}_{0}(|\phi(a)-\hat{\phi}(a)|-|\hat{\phi}(a)\log\frac{\phi(a)}{\hat{\phi}(a)}|)da]\\  
&=\lim_{t\rightarrow 0^{+}}\frac{1}{t}\int^{a_{1}}_{0}(|(U(t)\phi)(a)-\hat{\phi}(a)|-|\phi(a)-\hat{\phi}(a)|)da\\
&-\lim_{t\rightarrow 0^{+}}\frac{1}{t}\int^{a_{1}}_{0}(|\hat{\phi}(a)\log\frac{(U(t)\phi)(a)}{\hat{\phi}(a)}|-|\hat{\phi}(a)\log\frac{\phi(a)}{\hat{\phi}(a)}|)da\\
&=A-B.
\end{align*}
where 
\begin{align*} 
A&=\lim_{t\rightarrow 0^{+}}\frac{1}{t}\int^{a_{1}}_{0}(|(U(t)\phi)(a)-\hat{\phi}(a)|-|\phi(a)-\hat{\phi}(a)|)da\\
&=\lim_{t\rightarrow 0^{+}}\frac{1}{t}\int^{a_{1}}_{0}(\sqrt{((U(t)\phi)(a)-\hat{\phi}(a))^{2}}-\sqrt{(\phi(a)-\hat{\phi}(a))^{2}})da\\
&=\lim _{t\rightarrow 0^{+}}\int^{a_{1}}_{0} \frac{((U(t)\phi)(a)-\hat{\phi}(a))^{2}-(\phi(a)-\hat{\phi}(a))^{2}}{t(\sqrt{((U(t)\phi)(a)-\hat{\phi}(a))^{2}}+\sqrt{(\phi(a)-\hat{\phi}(a))^{2}})}da\\
&=\lim_{t\rightarrow 0^{+}}\int^{a_{1}}_{0}\frac{(U(t)\phi)(a)-\phi(a)}{t} \frac{(U(t)\phi)(a)+\phi(a)-2\hat{\phi}(a)}{(\sqrt{((U(t)\phi)(a)-\hat{\phi}(a))^{2}}+\sqrt{(\phi(a)-\hat{\phi}(a))^{2}})}da\\
&=\int^{a_{1}}_{0}\mathcal{G}(\phi)(a) \frac{\phi(a)-\hat{\phi}(a)}{|\phi(a)-\hat{\phi}(a)|}da\\
&=\int^{a_{1}}_{0}\mathcal{G}(\phi)(a) \text{sgn}(\phi(a)-\hat{\phi}(a))da.
\end{align*}
\begin{align*} 
B&=\lim_{t\rightarrow 0^{+}}\frac{1}{t}\int^{a_{1}}_{0}\hat{\phi}(a)(|\log\frac{(U(t)\phi)(a)}{\hat{\phi}(a)}|-|\log\frac{\phi(a)}{\hat{\phi}(a)}|)da\\
&=\lim_{t\rightarrow 0^{+}}\frac{1}{t}\int^{a_{1}}_{0}\hat{\phi}(a)(\sqrt{(\log\frac{(U(t)\phi)(a)}{\hat{\phi}(a)})^{2}}-\sqrt{(\log\frac{\phi(a)}{\hat{\phi}(a)})^{2}})da\\
&=\lim_{t\rightarrow 0^{+}}\frac{1}{t}\int^{a_{1}}_{0}\hat{\phi}(a)(\sqrt{(\log(U(t)\phi)(a)-\log\hat{\phi}(a))^{2}}-\sqrt{(\log\phi(a)-\log\hat{\phi}(a))^{2}})da\\
&=\lim_{t\rightarrow 0^{+}}\int^{a_{1}}_{0}\hat{\phi}(a) \frac{(\log(U(t)\phi)(a)-\log\hat{\phi}(a))^{2}-(\log\phi(a)-\log\hat{\phi}(a))^{2}}{t(\sqrt{(\log(U(t)\phi)(a)-\log\hat{\phi}(a))^{2}}+\sqrt{(\log\phi(a)-\log\hat{\phi}(a))^{2}})}da\\ 
&=\lim_{t\rightarrow 0^{+}}\int^{a_{1}}_{0}\hat{\phi}(a)\frac{\log(U(t)\phi)(a)-\log\phi(a)}{t}\\ 
&\times \frac{\log(U(t)\phi)(a)+\log\phi(a)-2\log\hat{\phi}(a)}{(\sqrt{(\log(U(t)\phi)(a)-\log\hat{\phi}(a))^{2}}+\sqrt{(\log\phi(a)-\log\hat{\phi}(a))^{2}})}da\\ 
&=\int^{a_{1}}_{0}\hat{\phi}(a)\frac{\mathcal{G}(\phi)(a)}{\phi(a)}\text{sgn}(\phi(a)-\hat{\phi}(a))da.
\end{align*}
Thus,
\begin{align*} 
\dot{V}(\phi)&=\int^{a_{1}}_{0}\mathcal{G}(\phi)(a) \text{sgn}(\phi(a)-\hat{\phi}(a))da-\int^{a_{1}}_{0}\hat{\phi}(a)\frac{\mathcal{G}(\phi)(a)}{\phi(a)}\text{sgn}(\phi(a)-\hat{\phi}(a))da\\
&=\int^{a_{1}}_{0}\frac{\phi(a)-\hat{\phi}(a)}{\phi(a)}\mathcal{G}(\phi)(a) \text{sgn}(\phi(a)-\hat{\phi}(a))da\\
&=\int^{a_{1}}_{0}\frac{\mathcal{G}(\phi)(a)}{\phi(a)}|\phi(a)-\hat{\phi}(a)|da\\
&=-\int^{a_{1}}_{0}(\mu_{0}(a, \eta_{0}(Q_{0}(\phi)))+\mu_{1}(a, \eta_{1}(Q_{1}(\phi)))+\mu_{2}(a))|\phi(a)-\hat{\phi}(a)|da.
\end{align*}
All terms are obviously nonpositive. Therefore, $V(\phi)$ is a Liapunov function for $U(t), t\geq 0$ in $X_{+}$ as in Theorem \ref{sec:T3.3}.
\end{proof}

\begin{theorem}\label{sec:T5.11} 
Let H.1-H.7 hold. 
If $ \mathcal{R}_{0}>1$, the unique positive equilibrium is globally asymptotically stable for the semigroup generated by the system (\ref{sec:1}). 
\end{theorem}

\begin{proof}

We derive from Proposition \ref{sec:prop5.11} that $\dot{V}(\phi)=0$ if and only if $\int^{a_{1}}_{0}\frac{\mathcal{G}(\phi)(a)}{\phi(a)}$ $|\phi(a)-\hat{\phi}(a)|da=0$, that is, $\phi(a)=\hat{\phi}(a)$, a.e. on $[0,a_{1}]$, because $\frac{\mathcal{G}(\phi)(a)}{\phi(a)}<0$ a.e. on $[0,a_{1}]$. Thus the only invariant set contained in the set $\dot{V}(\phi)=0$ is the positive equilibrium $\hat{\phi}$. Hence, LaSalle's theorem implies the largest invariant set is $\hat{\left\{\phi\right\}}$.
For $\forall\phi\in M_{0}$. Let $\tilde{\phi}\in \omega(\phi)$. There exists $\left\{t_{n}\right\}_{n\in \mathbb{Z}}\rightarrow \infty$ such that $\lim_{t_{n}\rightarrow \infty}U(t_{n})\phi=\tilde{\phi}$. Since $\tilde{\phi}\in \omega(\phi)\subseteq \left\{\bar{\phi}: \dot{V}(\bar{\phi})=0,\bar{\phi}\in\mathcal{A}_{0}\right\}$. Therefore, $\tilde{\phi}=\hat{\phi}$. We claim $\lim_{t\rightarrow \infty}U(t)\phi=\hat{\phi}$. If not, by Theorem \ref{sec:T5.1}, we have $\overline{\left\{U(t)\phi:t\geq 0\right\}}$ is compact, and therefore, there exists a sequence $\left\{t_{m}\right\}_{m\in\mathbb{Z}}\rightarrow \infty$ such that $\lim_{t_{m}\rightarrow \infty}U(t_{m})\phi=\check{\phi}\neq\hat{\phi}$. However, since $\check{\phi}\in\omega(\phi)$, this implies $\check{\phi}=\hat{\phi}$, which gives the contradiction.
   
\end{proof}

We give following examples to illustrate that there are different choices of fertility and mortality functions for which conditions of Theorem \ref{sec:T5.11} hold:

\begin{example}
For example \ref{sec:E4.3}, it is readily seen that, when $ e^{\Lambda}\geq\mathcal{R}_{0}>1$, conditions of Theorem \ref{sec:T5.11} are satisfied, thus, the unique positive equilibrium is globally asymptotically stable for the semigroup generated by the system (\ref{sec:E4.8}). In example \ref{sec:E4.13}, we numerically verify that with all parametric values set as in Table \ref{table:ba}, conditions of Theorem \ref{sec:T5.11} are satisfied and the unique positive equilibrium is globally asymptotically stable for the semigroup generated by the system (\ref{sec:E4.18}). Similarly, in example \ref{sec:E4.9} and example \ref{sec:E4}, one can easily verify that when $\mathcal{R}_{0}>1$, assumptions of Theorem \ref{sec:T5.11} hold and therefore the unique positive equilibrium is globally asymptotically stable for the semigroup generated by the system (\ref{sec:E4.1}) and by the system (\ref{sec:E4.2}). 
\end{example}

\chapter{Discussion}

Our model of age structure in human populations incorporates the extended juvenility and the prolonged post-reproductive population period unique to the human species. Our analysis of the model shows that these features are mathematically stable and robust in the sense that equilibria are recovered from perturbation and reset initial values.  The population could stabilize at a very low level or converge to a higher one depending on the initial population size. It would be a challenge for early humans with fairly small total population size \cite{C} to recover from such a perturbation.  The harsh time value is determined by the value of the unstable nontrivial equilibrium which is an indicator of the survivalship of a population. It is easily seen that a large harsh time value is associated with a reduced chance for a small population to survive especially when the evolution of the population exhibit oscillatory behavior which leads the total population size to fall to a very low level at the bottom. If at certain time the population size drops below the harsh time value, then the strong Allee effect would drive the species to go extinction eventually. In contrast, a smaller harsh time value provides a more friendly environment for early humans to grow from a fairly small total population size at the very beginning and after sufficiently many generations stabilize to a much higher population level. 
Human life expectancy is constrained, as confirmed by our model, in the sense that increasing senescent population becomes a burden on juveniles. Example \ref{sec:E4.13} shows that an increased fertility rate will not balance the cost of the extremely large juvenile mortality caused by the competing adult population. Instead, increased senescent burden on juveniles and increased fertility will lead to oscillatory behavior for the system when the birth process involved is either linear or nonlinear. Our model supports the thesis that the intrinsic age structure of human is essentially unchanged from the hunter-gatherer era to the present.

\begin{table}[ht]
\caption{Baseline model} 
\centering 
\begin{tabular}{| p{3cm} | p{5.5cm} | p{3cm} | } 
\hline 
Term & Form & Baseline value \\ [0.5ex] 
\hline 
 &  & $c_{1}=0.5$  \\[-1ex]
 \raisebox{1.5ex}{Fertility rate} & \raisebox{1.5ex}{$\beta(a,T)=\frac{c_{1}(a - 15)e^{-0.4(a - 15)}}{1+c_{2}T}$} & $c_{2}=0.00022$ \\[1ex]\hline
All-cause mortality &  & $c_{3}=0.01$  \\[-1ex]
(excluding particular
microbial mortalities) & \raisebox{1.5ex}{$\mu_{0}(a,T)= 0.03 +c_{3}e^{-0.04a}+ \eta(a)T$} & $\eta(a)=1.76\times 10^{-9}(a-20)^{2}$ \\[1ex]\hline
Juvenile mortality & $\mu_{1}(a,S)=c_{4}(15-a)S, \ a\in [0,15]$ &   \\[-1ex]
due to senescent  population
burden & $\mu_{1}(a,S)=0, \ a\in (15,\infty)$ & \raisebox{1.5ex}{$c_{4}=10^{-6}$} \\[1ex]\hline
Particular microbial
mortalities & $\mu_{2}(a)$ & 0 \\ [1ex] 
\hline 
\end{tabular}
\label{table:ba} 
\end{table}

\begin{figure}
\begin{center}
 \includegraphics[width=7.2cm,height=7cm]{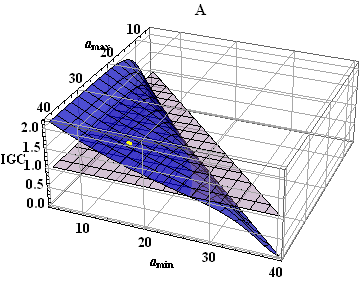}
 \includegraphics[width=5.2cm,height=7cm]{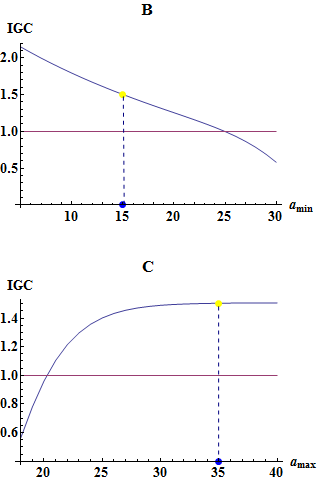}
\caption{The graph (A) is the value of IGC as $a_{\min}$ and $a_{\max}$ vary. The graph (B) is the value of IGC as $a_{\min}$ changes when $a_{\max}=35$ years are held fixed and the graph (C) is the value of IGC as $a_{\max}$ changes when $a_{\min}=15$ years are held fixed. IGC is significantly above one in a wide range of $a_{\min}$ and $a_{\max}$.}\label{sec:Fig1}
\end{center}
\end{figure}

\begin{figure}
\begin{center}
 \includegraphics[width=5.2cm,height=4cm]{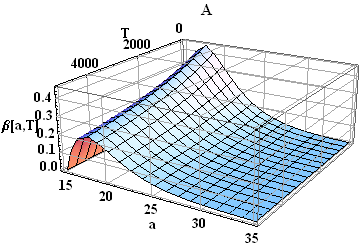}
 \includegraphics[width=5.2cm,height=4cm]{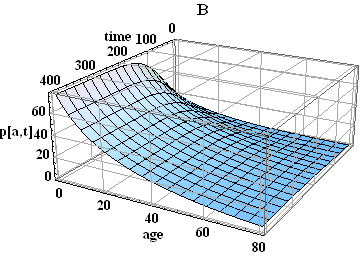}\\
 \includegraphics[width=10.2cm,height=2.5cm]{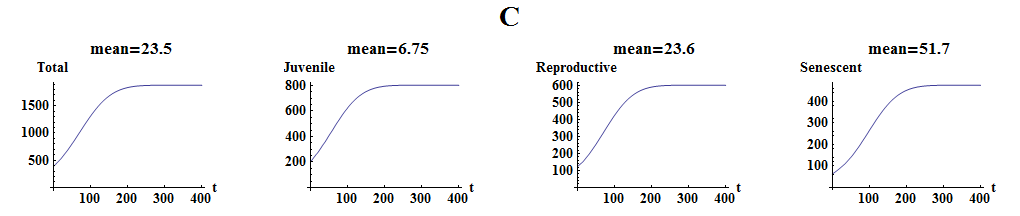}
\caption{The graph (A) is the fertility rate at age $a$ when the total population is $T$. The graph (B) is the time evolution in years of the age structured population density $p(a,t)$ for the baseline parametric values as in Table \ref{table:ba}. The graphs in (C) are time evolution in years of the total population and subpopulations  for the baseline parametric values as in Table \ref{table:ba}. The age structure is robust for those baseline parameter values.}\label{sec:Fig2}
\end{center}
\end{figure}

\begin{figure}\label{figure:ba1}
\begin{center}
 \includegraphics[width=5.2cm,height=4cm]{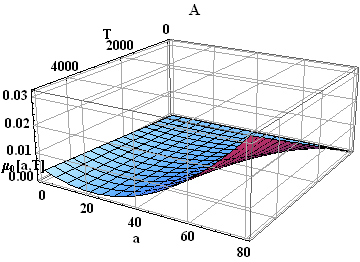}
 \includegraphics[width=5.2cm,height=4cm]{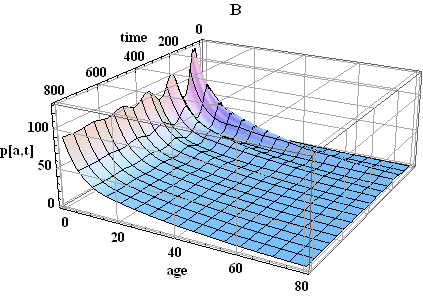}\\
 \includegraphics[width=10.2cm,height=2.5cm]{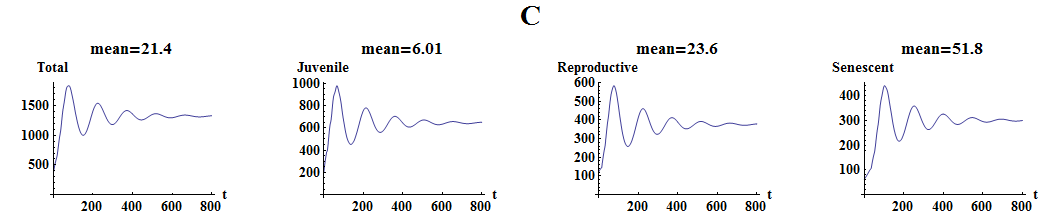}
\caption{The graph (A) is the all-cause mortality at age a when the total population is $T$. The graph (B) is time evolution in years of the age structured population density $p(a,t)$ for the baseline parametric values corresponding to increased fertility rate ($c_{1}=0.85$) and increased senescent burden on juvenile individuals ($c_{4}=2\times 10^{-5}$). The graphs in (C) are time evolution in years of the total population and subpopulations  for the baseline parametric values with increased fertility rate ($c_{1}=0.85$) and increased senescent burden on juvenile individuals ($c_{4}=2\times 10^{-5}$). The age structure is again robust for the parameter values, but undergoes significant oscillations if the initial conditions are perturbed far from the equilibrium values.}\label{sec:Fig3}
\end{center}
\end{figure}

\begin{figure}
\begin{center}
 \includegraphics[width=4.2cm,height=4.9cm]{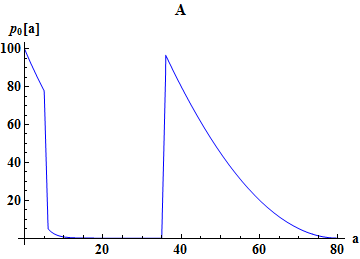}
 \includegraphics[width=6.2cm,height=4.9cm]{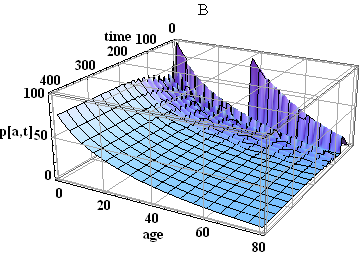}\\
 \includegraphics[width=10.2cm,height=2.5cm]{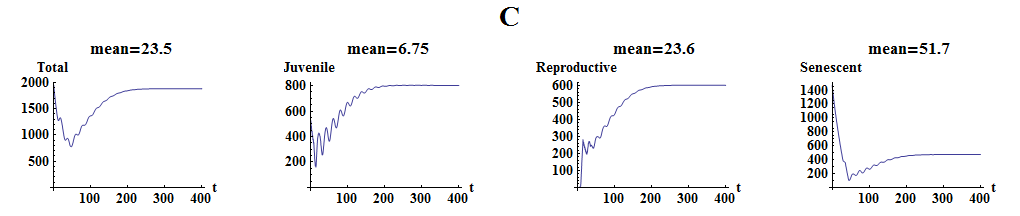}
\caption{The graph (A) is the initial distribution of the population. The graph (B) is time evolution in years of the age structured population density $p(a,t)$ for the baseline parametric values with initial distribution given as graph (A). The graphs in (C) are time evolution in years of the total population and subpopulations  for the baseline parametric values with the initial value given as graph (A). The age structure recovers from an extreme initial age distribution.}\label{sec:Fig5}
\end{center}
\end{figure}


\backmatter


\printbibliography

\addcontentsline{toc}{chapter}{Bibliography}

\end{document}